\documentclass[10pt,reqno]{amsart}

\usepackage[all]{xy}
\usepackage{amscd}
\usepackage{mathrsfs}
\usepackage{amssymb}
\usepackage{amsthm}
\usepackage{amsmath}
\usepackage{enumerate}
\usepackage{indentfirst}
\usepackage{dsfont}

\usepackage[active]{srcltx}
\usepackage{xcolor}
\usepackage{comment}

\usepackage{hyperref}

\newtheorem{thm}{Theorem}[section]
\newtheorem{conj}[thm]{Conjecture}
\newtheorem{cor}[thm]{Corollary}
\newtheorem{lem}[thm]{Lemma}
\newtheorem{prop}[thm]{Proposition}

\theoremstyle{definition}
\newtheorem{defn}[thm]{Definition}
\newtheorem{ex}[thm]{Example}
\theoremstyle{remark}
\newtheorem{rem}[thm]{Remark}
\newtheorem{qu}[thm]{\textbf{Question}}
\numberwithin{equation}{section}
\newtheorem{hypothesis}[thm]{\textbf{Hypothesis}}

\newcommand{\A}{\mathbb{A}}

\newcommand{\Mcodim}{\mathrm{codim}}

\newcommand{\XN}{{\mathcal{X}_0(N)}}
\newcommand{\JN}{{\mathcal{J}_0(N)}}

\newcommand{\cB}{{\mathcal{B}}}
\newcommand{\cF}{{\mathcal{F}}}
\newcommand{\cR}{{\mathcal{R}}}

\newcommand{\cZ}{{\mathcal{Z}}}

\newcommand{\fkg}{{\mathfrak g}}

\newcommand{\fkn}{{\mathfrak n}}
\newcommand{\Fmn}{{\mathfrak n}}

\newcommand{\fkq}{{\mathfrak q}}

\newcommand{\mL}{{\mathscr L}}

\newcommand{\mS}{{\mathscr S}}
\newcommand{\cL}{{\mathcal L}}
\newcommand{\N}{\mathbb{N}}

\newcommand{\Z}{\mathbb{Z}}

\newcommand{\Q}{\mathbb{Q}}

\newcommand{\F}{\mathbb{F}}

\newcommand{\bG}{\mathbf{G}}

\newcommand{\bT}{\mathbf{T}}

\newcommand{\R}{\mathbb{R}}
\newcommand{\T}{\mathbb{T}}

\newcommand{\C}{\mathbb{C}}

\newcommand{\Wh}{\mathrm{Wh}}
\newcommand{\scusp}{\mathrm{sc}}
\newcommand{\Perm}{\mathrm{Perm}}
\newcommand{\ram}{\mathrm{ram}}
\newcommand{\elp}{\mathrm{ell}}
\newcommand{\hs}{\mathrm{hs}}
\newcommand{\Sym}{\mathrm{Sym}}
\newcommand{\ord}{\mathrm{ord}}

\newcommand{\Jac}{\mathrm{Jac}}

\newcommand{\pl}{\hat{\mu}^{\mathrm{pl}}}

\newcommand{\bs}{\backslash}
\newcommand{\Fss}{\mathrm{F}\text{-}\mathrm{ss}}
\newcommand{\semis}{\mathrm{ss}}
\newcommand{\Sh}{\mathrm{Sh}}
\newcommand{\temp}{\mathrm{temp}}
\newcommand{\dep}{\mathrm{dep}}
\newcommand{\cond}{\mathrm{cond}}

\newcommand{\rank}{{\rm rank}}

\def\Gal{{\rm Gal}}

\newcommand{\ind}{{\rm Ind}}
\newcommand{\nind}{{\rm n\textrm{-}ind}}

\newcommand{\Frob}{{\rm Frob}}
\newcommand{\gr}{{\rm gr}}
\newcommand{\Fil}{{\rm Fil}}

\newcommand{\Fr}{{\rm Fr}}

\newcommand{\id}{{\rm id}}

\newcommand{\Sp}{{\rm Sp}}
\newcommand{\sgn}{{\rm sgn}}
\newcommand{\tr}{{\rm tr}\,}

\newcommand{\Hom}{{\rm Hom}}

\newcommand{\End}{{\rm End}}
\newcommand{\Aut}{{\rm Aut}}
\newcommand{\Out}{{\rm Out}}

\newcommand{\GL}{{\rm GL}}
\newcommand{\Lie}{{\rm Lie}\,}

\newcommand{\rec}{{\rm rec}}
\newcommand{\unr}{{\rm unr}}

\newcommand{\et}{{\mathrm{\acute{e}t}}}

\newcommand{\ur}{{\mathrm{ur}}}
\newcommand{\std}{{\mathrm{std}}}

\newcommand{\cE}{{\mathcal{E}}}

\newcommand{\cH}{{\mathcal{H}}}

\newcommand{\cO}{\mathcal{O}}

\newcommand{\ad}{{\rm ad}}

\newcommand{\Irr}{{\rm Irr}}

\newcommand{\vol}{{\rm vol}}

\newcommand{\Res}{\mathrm{Res}}

\def\hat{\widehat}

\def\lg{\langle}
\def\rg{\rangle}
\def\hra{\hookrightarrow}
\def\ra{\rightarrow}

\def\isom{\stackrel{\sim}{\ra}}
\def\ol{\overline}
\def\tilde{\widetilde}

\def\dirlim#1{\lim\limits_{\substack{\longrightarrow\\#1}}}

\def\benu{\begin{enumerate}}
\def\eenu{\end{enumerate}}

\def\beq{\begin{equation}}
\def\eeq{\end{equation}}

\def\bit{\begin{itemize}}
\def\eit{\end{itemize}}

\begin{document}

\title{On Fields of rationality for automorphic representations}

\author{Sug Woo Shin}\email{swshin@math.mit.edu}
\address{Department of Mathematics, Massachusetts Institute of Technology,
77 Massachusetts Avenue, Cambridge, MA 02139, USA$//$ Korea Institute for Advanced Study, 85 Hoegiro,
Dongdaemun-gu, Seoul 130-722, Republic of Korea}
\author{Nicolas Templier}\email{templier@math.princeton.edu}
\address{Department of Mathematics, Princeton University,
Princeton, NJ 08544, USA}

\date{\today; 2010 MSC codes: 11F30, 11F70, 11F80, 11R39, 11S37}

\begin{abstract}

  This paper proves two results on the field of rationality $\Q(\pi)$ for an automorphic representation $\pi$, which is the subfield of $\C$ fixed under the subgroup of $\Aut(\C)$ stabilizing the isomorphism class of the finite part of $\pi$. For general linear groups and classical groups, our first main result is the finiteness of the set of discrete automorphic representations $\pi$ such that $\pi$ is unramified away from a fixed finite set of places, $\pi_\infty$ has a fixed infinitesimal character, and $[\Q(\pi):\Q]$ is bounded. The second main result is that for classical groups, $[\Q(\pi):\Q]$ grows to infinity in a family of automorphic representations in level aspect whose infinite components are discrete series in a fixed $L$-packet under mild conditions.
\end{abstract}

\maketitle


\section{Introduction}\label{s:intro}

\subsection{Modular form case}\label{sub:MF-case}

  Let $S_k(N)$ be the space of cuspforms of weight $k\ge 2$ and level $\Gamma_0(N)$ with $N\ge 1$. Suppose that $f\in S_k(N)$ is an eigenform under the Hecke operator $\{T_p\}$ with eigenvalue $a_p(f)\in \C$ for each prime $p\nmid N$. It is well known that $\{a_p(f)\}_{p\nmid N}$ are algebraic integers and that they generate a finite extension of $\Q$ (in $\C$), to be denoted $\Q(f)$. The field $\Q(f)$ encodes deep arithmetic information about $f$ and is of our main concern here. To wit the significance of $\Q(f)$, the Eichler-Shimura construction associated to a weight 2 form $f$ a $GL_2$-type abelian variety of dimension $[\Q(f):\Q]$ as a quotient of the Jacobian of the modular curve $X_0(N)$. Moreover the two-dimensional $l$-adic Galois representations associated to $f$ are realized with coefficients in the completions of $\Q(f)$ at finite places.


  We are interested in two aspects of $\Q(f)$. The first question is on the growth of $\Q(f)$ in a family of modular forms $f$ with increasing level. Let $\cF_k(N)$ be the set of normalized cuspidal eigenforms of weight $k\ge 2$. These are eigenforms for all $T_p$ ($p\nmid N$) and define $$\cF_k(N)^{\le A}:=\{f\in \cF_k(N): [\Q(f):\Q]\le A\},\quad A\in \Z_{\ge1}.$$
  Serre has proved the following theorem, which serves as a prototype for one of our main results.

\begin{thm}\label{t:intro-Serre}(\cite[Thm 5]{Serre:pl}) Fix $k\ge 2$ and a prime $p$. Then
  $\lim\limits_{N\ra \infty,\atop (N,p)=1} |\cF_k(N)^{\le A}|/|\cF_k(N)|=0$.
\end{thm}

  Let us briefly recall Serre's argument. The key point is to show that
  \beq\label{e:a_p(f)}\left|\left\{ a_p(f): f\in \cF_k(N)^{\le A} \right\}\right|<\infty.\eeq
  This follows from the fact that $a_p(f)$ is an algebraic integer which is the sum of a Weil $p$-number of weight $k-1$ and its complex conjugate. The condition $[\Q(f):\Q]\le A$ implies that $[\Q(a_p(f)):\Q]\le A$, so such a Weil number is a root of a monic polynomial in $\Z[x]$ whose degree and coefficients are bounded only in terms of $p$, $k$, and $A$. Clearly there are only finitely many such polynomials, hence \eqref{e:a_p(f)}. Finally Theorem \ref{t:intro-Serre} is deduced from \eqref{t:intro-Serre} by using a trace formula argument.

  Serre then asked in \cite[\S6.1]{Serre:pl} whether the same type of result would be true without requiring some auxiliary prime $p$ to be coprime to the level. (For instance is the above result valid if the limit is taken along the sequence $N=2,~(2\cdot 3)^2,~ (2\cdot 3\cdot 5)^3,...$?)
  In our paper we generalize Theorem \ref{t:intro-Serre} to higher rank classical groups and partially settles Serre's question in the generalized setting for a sequence of levels $N\to \infty$ such that there exists a prime whose order in $N$ grows to infinity. Moreover we improve on the rate of decay of the quotient as in Theorem \ref{t:intro-Serre} by a logarithmic order.

  Another aspect of $\Q(f)$ is in relation to a finiteness result. Let us begin with recalling a deep theorem of Faltings, who also proved a stronger version in which ``up to isogeny'' is replaced with ``up to isomorphism'' (the Shafarevich conjecture).

\begin{thm}\label{t:intro-Faltings}(\cite[Thm 5]{Fal86}) Fix $n\in \Z_{\ge 1}$ and a finite set of primes $S$. Then there are only finitely many abelian varieties of dimension $n$ having good reduction outside $S$ up to isogeny.
\end{thm}

  The Shimura-Taniyama conjecture, as confirmed by Wiles and Breuil-Conrad-Diamond-Taylor, translates the case $n=1$ of the above theorem into a finiteness result about modular forms: namely there are only finitely many newforms $f$ such that $[\Q(f):\Q]=1$ which are contained in $\cF_2(N)$ for some level $N$ whose prime divisors are all contained in $S$. With this motivation an automorphic analogue of the above finiteness theorem will be pursued in this paper.

  To formulate and make progress toward the problems raised in this subsection we are going to introduce some definitions, concepts, and conjectures before stating the main results.

\subsection{C-algebraic automorphic representations}

Algebraicity of automorphic forms and representations has been studied by Shimura, Waldspurger, Harder, Harris, and many other mathematicians. Regarding automorphic representations of $GL_n$ the definition of algebraicity was first formulated by Clozel~\cite{Clo90} and recently extended to arbitrary connected reductive groups by Buzzard and Gee~\cite{BG}. In fact one main point of their paper is to distinguish between the two possible definitions of algebraicity, namely C-algebraicity and L-algebraicity, the former generalizing Clozel's notion.
  In this article our attention is restricted to C-algebraic representations mainly because these are expected to be exactly the ones having number fields as their fields of rationality. (There is also W-algebraicity recently suggested by Patrikis \cite{Pat}, but again C-algebraicity is believed to be the exact condition to ensure the finiteness of the field of rationality over $\Q$.)

  To be precise let $G$ be a connected reductive group over $\Q$. To avoid vacuous statements we assume throughout the paper that the rank of the groups under consideration is at least one. Let $\pi=\otimes_v \pi_v=\pi^\infty\otimes \pi_\infty$ be an automorphic representation of $G(\A)$. Here $\pi^\infty$ and $\pi_\infty$ denote the finite and infinite components. We say that $\pi$ is C-algebraic if, loosely speaking, the infinitesimal character of $\pi_\infty$ is integral after a shift by the half sum of all positive roots (for some thus for all choices of positivity on the set of roots). When $\sigma$ is a field automorphism of $\C$, let $(\pi^\infty)^\sigma$ denote the $G(\A^\infty)$-representation on the underlying vector space of $\pi^\infty$ twisted by a $\sigma$-linear automorphism. For any $\pi$ define its field of rationality as the field of the definition of its isomorphism class, i.e.
  \beq\label{e:intro-field-of-rat}\Q(\pi):=\{z\in \C: \sigma(z)=z,~\forall\sigma\in \Aut(\C)~\mbox{s.t.}~(\pi^\infty)^{\sigma}\simeq \pi^\infty\}.\eeq
  The following was conjectured by Clozel (for $G=GL_n$) and Buzzard-Gee.

  \begin{conj}\label{c:intro-C-alg}
    $\pi$ is C-algebraic if and only if $\Q(\pi)$ is finite over $\Q$.
  \end{conj}
  It is worth noting that in the special but subtle case of Maass cusp forms for $GL_2$ over $\Q$, Sarnak~\cite{Sar02} classified the forms with integer coefficients, showing in particular that they are C-algebraic (i.e. Laplace eigenvalue being $1/4$), and made a remark on the transcendence of $\Q(\pi)$.

   According to the conjecture C-algebraic representations are the most suitable for studying questions on the growth of fields of rationality. To obtain unconditional results, we show that $\Q(\pi)$ is a number field for cohomological representations $\pi$, which form a large subset inside the set of C-algebraic representations, by adapting an argument of Clozel using arithmetic cohomology spaces. See \S\ref{sub:Clozel-BHR} below. Note that if $G$ is semisimple then any $\pi$ such that $\pi_\infty$ is a discrete series is always cohomological.

%

\subsection{Conjectures}

  Let us highlight two interesting conjectures that we were led to formulate during our investigation of fields of rationality for automorphic representations. Some partial results and remarks are found in the next subsection as well as in the main body of our paper.

  The first conjecture, a small refinement of the well-known Fontaine-Mazur conjecture, is not directly concerned with field of rationality but rather with integrality of local parameters (e.g. Satake parameters or Frobenius eigenvalues of a Galois representation). The question arises naturally as a weak form of integrality is needed to answer a generalization of Theorem \ref{t:intro-Serre}.

\begin{conj}\label{c:intro-FM}

  Let $F$ be a number field and $\rho:\Gal(\ol{F}/F)\ra GL_n(\ol{\Q}_l)$ a continuous irreducible representation unramified outside finitely many places. The following are equivalent:
\benu
\item $\rho$ is de Rham at every place $v|l$ with nonnegative Hodge-Tate weights (adopting the convention that the cyclotomic character has Hodge-Tate weight $-1$).
\item the Weil-Deligne representation associated with $\rho$ at every finite place $v\nmid l$ is integral and pure of weight $w\in \Z$ which is independent of $v$,
\item $\rho$ appears as a subquotient of $H_{\et}^i(X\times_F \ol{F},\ol{\Q}_l)$ for some proper smooth scheme $X$ over $F$ and some $i\in\Z_{\ge 0}$.
\eenu

\end{conj}

  The motivation for the conjecture comes from our effort to obtain Theorem \ref{t:intro-growth} below (which generalizes Theorem \ref{t:intro-Serre}), where we need a version of the statement that $a_p(f)$ is an algebraic \emph{integer}. We derive a partial result toward Conjecture \ref{c:intro-FM} (Proposition \ref{p:Galois-reps}) for the Galois representations arising from (conjugate) self-dual automorphic representations by exploiting the fact that they appear in the cohomology of Shimura varieties. This serves as a crucial ingredient in the proof of Theorem \ref{t:intro-growth}.

  The second conjecture is on the finiteness of automorphic representations with bounded field of rationality. It is an automorphic analogue of (the isogeny version of) the Shafarevich conjecture and its analogue for Galois representations formulated by Fontaine and Mazur (\cite[I.\S3]{FM95}). Theorem \ref{t:intro-fin} below partially confirms the conjecture.

\begin{conj}\label{c:intro-fin}
 Fix $A\in\Z_{\ge 1}$, $S$ a finite set of places of $F$ containing all infinite places, and an infinitesimal character $\chi_\infty$ for $G(F\otimes_\Q \R)$. Then there are only finitely many discrete automorphic representations $\pi$ of $G(\A_F)$ with infinitesimal character $\chi_\infty$ such that $\pi^S$ is unramified and $[\Q(\pi):\Q]\le A$.
\end{conj}


\subsection{Main results}\label{sub:main-results}

  Let us make it clear at the outset that our results concerning quasi-split classical (i.e. symplectic, orthogonal\footnote{As we will never deal with the usual (disconnected) orthogonal groups, special orthogonal groups will be called orthogonal groups in favor of simpler terminology. We will be precise where we have to be.}, or unitary) groups rely on Arthur's endoscopic classification \cite{Arthur} and its analogue for unitary groups due to Mok \cite{Mok}. (However our finiteness theorem for general linear groups, cf. Theorem \ref{t:intro-fin} below, is unconditional.) The classification is based on some unproven assertions on the stabilization of the twisted trace formula for $GL_n$ and a little more, which are hoped to be proved in the near future. So we are making the same hypotheses as Arthur does in his work. (Also see \cite[1.18]{BMM} and the footnote around Hypothesis \ref{hypo:TwTF} for a discussion of the hypotheses.) We only deal with quasi-split groups mainly because the analogous theorems for inner forms are not complete (see the last chapter of \cite{Arthur} for a sketch), but our argument should apply equally well to the inner forms. With this in mind we have written the argument in such a way that our main theorems remain true for non-quasi-split classical groups with little change in the proof once the necessary classification becomes available. As a matter of fact, Theorem \ref{t:intro-growth} in case (i) is almost an unconditional theorem for (not necessarily quasi-split) unitary groups thanks to the base change results for cohomological representations in \cite{Lab}. (Unlike Arthur's work, the latter are not conditional on the full stabilization of the twisted trace formula or any other hypotheses.)


  Our first main result is a finiteness theorem for automorphic representations with bounded field of rationality. It is worth emphasizing that we allow arbitrary infinitesimal characters (e.g. those corresponding to C-algebraic Maass forms in the case of $GL_2$ over $\Q$) even including transcendental ones (in which case the set of $\pi$ is expected to be empty by Conjecture \ref{c:intro-C-alg}).

\begin{thm}\label{t:intro-fin}(Theorems \ref{t:fin-GL_n}, \ref{t:fin-classical})
  Conjecture \ref{c:intro-fin} is true for general linear groups and quasi-split classical groups.
\end{thm}

  Our second main result is on the growth of field of rationality in a family of automorphic representations. We work with a quasi-split classical group $G$ over $\Q$ for simplicity (in the main body $G$ is over any totally field) and introduce a family in level aspect with prescribed local conditions as in \cite{ST11cf}. Let $n_x\in\Z_{\ge 1}$, $\xi$ be an irreducible algebraic representation of $G$ over $\C$ whose highest weight is regular, $S_0$ be a finite set of finite primes (which could be empty so that no local condition may be imposed), and $\hat{f}_{S_0}$ be a well-behaved function on the unitary dual of $G(\Q_{S_0})$.
  The family in question is a sequence $$\cF_x=\cF(n_x,\hat{f}_{S_0},\xi),~x\in \Z_{\ge 1}\quad \mbox{such that}\quad
  n_x\ra\infty~\mbox{as}~x\ra\infty,$$ where each $\cF_x$ consists of discrete automorphic representations $\pi$ of $G$ which, loosely speaking, has level $n_x$, weight $\xi$, and prescribed local conditions at $S_0$ by $\hat{f}_{S_0}$. Then each $\cF_x$ is a finite set whose cardinality $|\cF_x|$ tends to infinity as $x\ra\ \infty$. Actually in our formulation $\cF_x$ is a multi-set in that each $\pi$ is weighted by the dimension of the fixed vectors of $\pi^\infty$ under the principal congruence subgroup of level $n_x$. (See \S\ref{sub:growth-level} for the precise definition of $\cF_x$ and $|\cF_x|$.) For $A\in \Z_{\ge1}$ define
  $$\cF_x^{\le A}:=\{\pi\in \cF_x: [\Q(\pi):\Q]\le A\}.$$
  Note that we have $[\Q(\pi):\Q]<\infty$ for every $\pi\in \cF_x$ since $\pi$ is cohomological in that $\pi_\infty\otimes \xi$ has non-vanishing Lie algebra cohomology. We prove a theorem roughly saying that the field of rationality grows generically in the family $\{\cF_x\}_{x\ge 1}$ in the case (i) or (ii) below. Note that (ii) includes the level sequence $2, (2\cdot 3)^2, (2\cdot 3\cdot 5)^3, (2\cdot3\cdot5\cdot7)^4, ...$ for instance. Unfortunately neither (i) nor (ii) includes the sequence $2, 2\cdot 3, 2\cdot 3\cdot 5, ...$

\begin{thm}\label{t:intro-growth}(Theorems \ref{t:growth-coeff}, \ref{p:level-prime}) Let $G\neq \{1\}$ be a quasi-split classical group, or a non-quasi-split unitary group. Suppose there exists a prime $p\not \in S_0$, at which $G$ is unramified, such that either
\benu
\item $(n_x,p)=1$ for all but finitely many $x$, or
\item $\ord_p(n_x)\ra \infty$ as $x\ra\infty$.
\eenu Then for every $A\in \Z_{\ge1}$,
 $\lim_{x\ra\infty} \frac{|\cF_x^{\le A}|}{|\cF_x|}=0.$ Moreover let $S_{\unr}$ be the number of primes $p$ satisfying (i) (which could be infinite) and such that $G$ is unramified at $p$. Put $R_{\unr}:=\sum_{p\in S_{\unr}} \rank\, G_{\Q_p}$. Then
 $$ |\cF_x^{\le A}|=O(|\cF_x|/(\log |\cF_x|)^R),\quad \forall R\le R_{\unr}.$$
\end{thm}

Especially pleasing features of the theorem are that some arbitrarily high ramification can be treated as seen in (ii) and that the upper bound has a logarithmic power-saving. The case (ii) seems to be new already in the case of modular forms while the logarithmic saving generalizes~\cite{Royer:dimension-rang,GJS:spectra}.
  It would be nice to prove (or disprove) the theorem without (i) and (ii). We can do it under some restrictive hypotheses (which are too special to be discussed here) but do not know any general type of result.

  It is natural to ask whether $|\cF_x^{\le A}|=O(|\cF_x|^{\delta})$ for some $\delta<1$ for a level-aspect family $\cF_x$ (whose level $n_x\ra\infty$) for an arbitrary reductive group $G$, cf. Question \ref{q:coeff1} below.  This is already challenging for $G=\GL(2)$~\cite[p.89]{Serre:pl}. The above theorem does not achieve this. However we do provide a nearly optimal answer under a hypothesis on $\{n_x\}_{x\ge 1}$ (Corollary \ref{c:growth-finite}).
  Let $G$ be a group as in Theorem \ref{t:intro-fin} and suppose that $\{n_x\}_{x\ge 1}$ is supported on a finite set $S$ of finite primes in the sense that for all but finitely many $x$, every prime divisor of $n_x$ is in $S$. Then $|\cF_x^{\le A}|=O(1)$.
  This is actually an easy corollary of Theorem \ref{t:intro-fin}. Again no condition on infinitesimal characters at $\infty$ is needed (so it applies to C-algebraic Maass forms when $G=GL_2$ for instance).

  In the following we sketch the proof of Theorem \ref{t:intro-fin} and Theorem \ref{t:intro-growth}.
  Both theorems take local finiteness results as key inputs. The former theorem in the case of $GL_n$ uses
\begin{prop}\label{p:intro-bounded-ram} Fix $A\ge 1$ and a prime $p$ and an integer $n\ge 2$. There exists a constant $C=C(A,p,n)$ such that every irreducible smooth representation of $GL_n(\Q_p)$
  with $[\Q(\pi_p):\Q]\le A$ has conductor $\le C$. (Here $\Q(\pi_p)$ is the field of rationality for $\pi_p$ defined as in \eqref{e:intro-field-of-rat}.)
\end{prop}
  For the proof of the proposition we pass to the Galois side via the local Langlands correspondence and examine the representation of the inertia group. Note that a suitable normalization of the local Langlands correspondence preserves the field of rationality. Since the inertia representation must have finite image, it is possible to conclude with some elementary representation theory and ramification theory for local fields. Once the proposition is in place, Theorem \ref{t:intro-fin} is an easy consequence of Harish-Chandra's finiteness theorem for automorphic forms.

  Theorem \ref{t:intro-growth} requires a more arithmetic kind of local finiteness theorem. When $G$ is a quasi-split classical group, we show the following far-reaching generalization of the finiteness of Weil numbers
  (\S\ref{sub:MF-case}) to the case for higher rank groups allowing arbitrary ramification at $p$.

\begin{prop}\label{p:intro-finite}(Corollary \ref{c:sparsity}) Fix $A\ge 1$, a prime $p$, and an irreducible algebraic representation $\xi$ of $G$. Then the set of irreducible tempered representations $\pi_p$ of $G(\Q_p)$ with $[\Q(\pi_p):\Q]\le A$ which may be realized as the $p$-components of discrete $\xi$-cohomological automorphic representations $\pi$ of $G(\A)$ is finite.
\end{prop}

A crucial input in the proof is the properties of the Galois representations associated with $\pi$ concerning weight and integrality, which we justify along the way. The integrality here is the same kind as in Conjecture \ref{c:intro-FM}.(ii). In fact this consideration led us to formulate the conjecture. To associate Galois representations, work of Arthur and Mok is applied to transfer $\pi$ to a suitable general linear group, and the field $\Q(\pi)$ has to be kept track of during the transfer. To this end we check the nontrivial fact that the transfer from $G$ to the general linear group is rational in the sense that it commutes with the $\Aut(\C)$-action on the coefficients. It would be of independent interest that a similar argument would show that many other endoscopic transfers are rational 	(sometimes with respect to the $\Aut(\C/F)$-action for a number field $F$).

Both (i) and (ii) of Theorem \ref{t:intro-growth} are deduced from Proposition~\ref{p:intro-finite} via the theorem (proved earlier by us in \cite{Shi-Plan} and \cite{ST11cf}) that $\pi_p$'s are equidistributed with respect to the Plancherel measure for $G(\Q_p)$. The equidistribution reduces the proof to showing that the set in Proposition \ref{p:intro-finite} has negligible Plancherel measure in the subset of the unitary dual of $G(\Q_p)$ consisting of representations whose levels are at most (the $p$-part of) $n_x$. Part (i) results from the fact that the Plancherel measure is atomless when restricted to the unramified unitary dual. The saving by $(\log |\cF_x|)^R$ in the denominator comes from the quantitative Plancherel equidistribution theorem~\cite{ST11cf} and a uniform approximation of characteristic functions in the unramified unitary dual by Hecke functions of bounded degree. For part (ii) observe that the condition there implies that the mass of the set in Proposition \ref{p:intro-finite}, which may \emph{not} be zero since some points may correspond to discrete series, becomes negligible relative to the mass of the level $\le n_x$ part of the unitary dual as $\ord_(n_x)\to \infty$.

\subsection{Organization}

Section~\ref{s:field-of-rat} introduces basic notions such as C-algebraic, C-arithmetic, and cohomological automorphic representations as well as the field of rationality for local and global representations, and then builds background materials. The key result is that cohomological representations are C-algebraic and strongly C-arithmetic, indicating that a good playground to study field of rationality is the world of cohomological representations. We included many supplementary results which do not play roles in proving main theorems but are interesting in their own right.
Section 3 is mainly local and Galois-theoretic. We prove the fundamental proposition that a Weil-Deligne representation with bounded field of rationality has bounded ramification and transfer the result to the automorphic side via the local Langlands correspondence for $GL_n$. Section 4 is global in nature and draws deep facts from both Galois and automorphic sides. It is shown that the Galois representations associated with (conjugate) self-dual automorphic representations of $GL_n$ are pure and integral. The remainder of Section~\ref{s:aut-classical} is concerned with the twisted endoscopic transfer and classification theorems for quasi-split classical groups relative to $GL_n$. This is where Arthur's work is invoked. Section~\ref{s:finiteness} proves key local finiteness results to be used in the proof for main theorems. The basic strategy is to prove something for $GL_n$ and transfer the result to classical groups or vice versa. To play this game the rationality of endoscopic transfer as proved in \S\ref{sub:rational-transfer} is essential. The culmination of Section 5 is the finiteness theorems in \S\ref{sub:fin-results}.
  In the last Section \ref{s:growth-coeff} we prove several results on the field of rationality for families of automorphic representations in level aspect and conclude with remarks on counting elliptic curves and some outlook.

\subsection{Acknowledgments}

  We are grateful to Wee Teck Gan, Peter Sarnak, Jean-Pierre Serre and David Vogan for their helpful comments and the referee for a careful reading.
  The authors acknowledge support from the National Science Foundation under agreements No. DMS-1162250 and DMS-1200684.

\subsection{Notation and convention}\label{sub:notation}

\begin{itemize}

  \item $\ol{k}$ denotes an algebraic closure of $k$ for any field $k$.
  \item $\Res_{k'/k}$ denotes the Weil restriction of scalars from a finite extension field $k'$ to $k$.
  \item $\ind$ and $\nind$ denote the unnormalized and normalized inductions from parabolic subgroups, respectively.
  \item $F$ is a number field, $\Gamma_F:=\Gal(\ol{F}/F)$, and $W_F$ is the Weil group.
  \item $q_v$ denotes the cardinality of the residue field and $\Frob_v$ is the geometric Frobenius element at $v$ if $v$ is a finite place of $F$,
  \item $S_\infty$ is the set of all infinite places of $F$,
  \item $\A_F$ is the ring of ad\`{e}les over $F$; $\A_F^S$ is the restricted product of $F_v$ for all $v\notin S$;
  $\A_F^\infty:=\A_F^{S_\infty}$.
  \item $G$ is a connected reductive group over $F$,
  \item $\hat{G}$ is the dual group, ${}^L G$ is the $L$-group.
  \item $G(F_v)^{\wedge}$ is the unitary dual of $G(F_v)$.
  \item $\Irr(G(F_v))$ denotes the set of isomorphism classes of irreducible smooth representations of $G(F_v)$. Write $\Irr^{\temp}(G(F_v))$ (resp. $\Irr^{\ur}(G(F_v))$) for the subset consisting of tempered (resp. unramified - for a choice\footnote{Such a choice will always be implicit whenever we mention unramified representations in this article.} of a hyperspecial subgroup of $G(F_v)$ if it exists) representations.	
  \item $\rho\in X^*(T)\otimes_{\Z}\Q$ is the half sum of all positive roots when a choice is made of a maximal torus $T$ and a Borel subgroup $B$ such that $T\subset B$ ($\rho$ is also viewed as the half sum of all positive coroots on the dual side, cf. \S\ref{sub:C-alg} below).
  \item $\cH(H,k)$ denotes the $k$-algebra of locally constant compactly supported functions on $H$ where $H$ is a locally compact totally disconnected group and $k$ is a field, and $\cH_U(H,k)$ the sub $k$-algebra of bi-$U$-invariant functions where $U$ is an open compact subgroup of $H$. (For instance $H=G(\A_F^\infty)$ or $H=G(F_v)$ in the notation above.)
  \item Given $G$ as above, hyperspecial subgroups $U^{\hs}_v$ are fixed at finite places $v$ outside the set $S_{\ram}$ of finitely many $v$ such that $G$ is ramified over $F_v$. We identify $\cH(G(\A_F^\infty),k)$ with the restricted tensor product $\otimes'_{v\nmid \infty} \cH(G(F_v),k)$ with respect to $\cH_{U^{\hs}_v}(G(F_v),k)$ and decompose an irreducible admissible representation $\pi$ of $G(\A_F^\infty)$ as $\pi=\otimes'_{v\nmid \infty} \pi_v$. We speak of unramified representations at finite places $v\notin S_{\ram}$ with respect to $U^{\hs}$.
  \item $\varphi_v:W_{F_v}\times SL_2(\C)\ra {}^L G$ ($v$ finite) and $\varphi_v:W_{F_v}\ra {}^L G$ ($v$ infinite) are notation for local $L$-parameters; the associated local $L$-packets are denoted $LP(\varphi_v)$ (in the cases where the local Langlands correspondence is established).
  \item Fix a field embedding $\ol{\Q}\ra \C$ and $\ol{\Q}_l\ra \C$ for each prime $l$ once and for all.
  \item All twisted characters (and intertwining operators for $\theta$ defining them) are normalized as in Arthur's book.
\end{itemize}

\section{Field of rationality}\label{s:field-of-rat}

  The reader may want to compare the contents of our \S\ref{sub:C-alg} and \S\ref{sub:Clozel-BHR} to \S3.1 and \S7 of \cite{BG}.

\subsection{C-algebraicity and coefficient fields}\label{sub:C-alg}



  Let $\pi=\otimes'_v \pi_v$ be an automorphic representation of $G(\A_F)$.
   Let $S$ be a finite set of places of $F$ containing $S_\infty$.
   We recall the definition of C-algebraicity from \cite[Def 3.1.2]{BG} (generalizing the notion of algebraicity in \cite{Clo90}). For each infinite place $v$ of $F$, denote by $\varphi_{\pi_v}:W_{F_v}\ra {}^L G$ the associated parameter via the local Langlands correspondence (\cite{Lan88}).

\begin{defn}\label{d:C-algebraic} For $v|\infty$, $\pi_v$ is \textbf{C-algebraic} if there exists a maximal torus $\hat{T}$ of $\hat{G}$ satisfying $\varphi_{\pi_v}(W_{\C})\subset \hat{T}\times W_{\C}$ with the property that $\varphi_{\pi_v} |_{W_{\C}}:W_{\C}\ra \hat{T}$ (via any $\R$-embedding $\sigma:F_v\hra \C$, after projecting down to $\hat{T}$) belongs to $\rho+X_*(\hat{T})$, where $\rho$ is the half sum of all positive coroots in $\hat{T}$ with respect to a Borel subgroup $\hat{B}$ containing $\hat{T}$. (The latter property is independent of the choice of $\sigma$, $\hat{T}$, and $\hat{B}$. See \cite[2.3]{BG}.) We say that $\pi_v$ is \textbf{regular} if $\varphi_{\pi_v} |_{W_{\C}}$ is not invariant under any nontrivial element of the Weyl group for $\hat{T}$ in $\hat{G}$. If $\pi_v$ is C-algebraic (resp. regular) for every infinite place $v$ then $\pi$ is said to be \textbf{C-algebraic} (resp. \textbf{regular}).
\end{defn}

  We remark that when $G=GL_n$, our notion of $\pi$ being algebraic (resp. regular) coincides with the one in \cite{Clo90}. For the next definition we introduce a twist of a complex representation. For $\tau\in \Aut(\C)$ and a complex representation $(\Pi,V)$ of a group $\Gamma$, denote by $\Pi^\tau$ the representation of $\Gamma$ on $V\otimes_{\C,\tau^{-1}}\C$ via $\Pi\otimes 1$.

\begin{defn}
  The \textbf{field of rationality} $\Q(\pi^S)$ is the fixed field of $\C$ under the group $\{\tau\in \Aut(\C): (\pi^S)^{\tau}\simeq \pi^S\}$.
  If $S=S_\infty$, simply write $\Q(\pi)$ for $\Q(\pi^{S_\infty})$. For a finite place $v$ of $F$, $\Q(\pi_v)$ is defined to be the fixed field under the group $\{\tau\in \Aut(\C): \pi_v^{\tau}\simeq \pi_v\}$.
\end{defn}

  An easy observation is that $\Q(\pi)$ is the composite field of $\Q(\pi_v)$ for all finite $v$ (as a subfield of $\C$).

\begin{rem} Here is another possible notion of rationality, which will not be used in this paper.
  We say that $\pi$ is \emph{defined over} a subfield $E$ of $\C$ if there exists a smooth $E[G(\A_F^\infty)]$-module $\pi^\infty_E$ such that $\pi^\infty_E\otimes_E \C\simeq \pi^\infty$. Similarly $\pi_v$ is said to be defined over $E$ for a finite place $v$ if there exists a smooth $E[G(F_v)]$-module $\pi_{v,E}$ such that $\pi_{v,E}\otimes_E \C\simeq \pi_v$.
  If $\pi$ (resp. $\pi_v$) is defined over $E$ then clearly $\Q(\pi)$ (resp. $\Q(\pi_v)$) contains $E$. A natural question is whether $\pi$ (resp. $\pi_v$) can be defined over $\Q(\pi)$ (resp. $\Q(\pi_v)$) itself.
  If $\pi_v$ is unramified, it is not hard to see that $\pi_v$ is defined over $E$ (independently of the choice of a hyperspecial subgroup of $G(F_v)$) if and only if $E\supset \Q(\pi_v)$, cf. \cite[Lem 2.2.3, Cor 2.2.4]{BG}. The authors do not know whether the analogue holds for general generic $\pi_v$ or $\pi^\infty$. In the case of $G=GL_n$, this has been shown in \cite{Clo90} using the theory of new vectors.
\end{rem}

\begin{rem} Let $v$ be a finite place of $F$ where $\pi_v$ is unramified. It is in general false that the Satake parameters of $\pi_v$ are defined over $\Q(\pi_v)$ (let alone $\Q(\pi)$) in the sense of \cite[Def 2.2.2]{BG} due to an issue with the square root of $q_v$.
\end{rem}

\begin{defn}
  For a finite $v$, we say $\pi_v$ is C-arithmetic if $\Q(\pi_v)$ is finite over $\Q$. An automorphic representation $\pi$ is \textbf{C-arithmetic} if $\Q(\pi^S)$ is finite over $\Q$ for some finite set $S$ containing $S_\infty$. It is \textbf{strongly C-arithmetic} if $\Q(\pi)$ is finite over $\Q$.
\end{defn}

\begin{rem}\label{r:arith-strong-arith}
  Our C-arithmeticity is equivalent to that of \cite{BG}. It is reasonable to believe that $\pi$ is C-arithmetic if and only if it is strongly C-arithmetic, but the only if part does not seem easy to prove directly. At least when $G$ is a torus it can be verified that C-arithmeticity is equivalent to strong C-arithmeticity. Indeed the only if part is true if $G$ is a split torus by strong approximation. If $G$ is a general torus the proof is reduced to the split case via a finite extension $F'/F$ splitting $G$ by employing the fact (see the proof of the theorem 4.1.9 in \cite{BG}) that $G(F)$ and the image of $G(\A^{\infty}_{F'})$ under the norm map together generate an open and closed subgroup of $G(\A^{\infty}_F)$ of finite index.
\end{rem}

\begin{rem}
  Even if $\pi_v$ is C-arithmetic at every finite $v$, it may happen that $\pi$ is not C-arithmetic. For instance when $G=GL_1$ over $\Q$ and
  $\pi=|\cdot|^{1/2}$, we have $\Q(\pi_p)=\Q(p^{1/2})$ for each prime $p$ so $\pi_p$ is C-arithmetic. However $\Q(\pi^S)$ is an infinite algebraic extension of $\Q$. Note that $|\cdot|^{1/2}$ is not C-algebraic.
\end{rem}

  In general there is no reason to expect that $\Q(\pi)$ is finite or algebraic over $\Q$. In this optic the significance of C-algebraicity stems from Conjecture \ref{c:alg-arith} below. An expectedly equivalent conjecture was formulated in \cite[Conj 3.1.6]{BG}, where they put C-arithmetic in place of strongly C-arithmetic. When $G$ is a torus, the two versions of the conjecture are indeed equivalent (Remark \ref{r:arith-strong-arith}), so our conjecture is known to be true by \cite[Thm 4.1.9]{BG} based on work of Weil and Waldschmidt.

\begin{conj}\label{c:alg-arith}
  $\pi$ is C-algebraic if and only if it is strongly C-arithmetic.
\end{conj}

  We remark that there are other reasons why C-algebraic automorphic representations stand out. One reason is that C-algebraicity is a natural necessary condition in the cuspidal case (and not too far from being a sufficient condition) to contribute to cohomology, cf. Lemma \ref{l:coh=C-alg} below. Another reason is that $l$-adic Galois representations are expected to be associated with C-algebraic representations, cf. \cite[Conj 5.3.4]{BG}. (In a simpler way Galois representations should also be attached to L-algebraic representations, which differ from C-algebraic ones by ``twisting''. See Conjectures 3.2.1 and 3.2.2 of \cite{BG}.)



  C-algebraicity and C-arithmeticity are preserved under \emph{unnormalized} parabolic induction. (Compare with \cite[Lem 7.1.1]{BG} and the paragraph above it.)

\begin{lem}\label{l:arithmetic-par-ind}
  Let $M$ be a Levi subgroup of an $F$-rational parabolic subgroup $P$ of $G$. Let $\Pi_M$ be an automorphic representation of $M(\A_F)$. Suppose that $\Pi$ is an irreducible subquotient of the unnormalized induction $\ind^G_P(\Pi_M)$. Then $\Pi_M$ is C-algebraic if and only if $\Pi$ is C-algebraic. If $\Pi_M$ is C-arithmetic (resp. strongly C-arithmetic) then so is $\Pi$.
\end{lem}

\begin{rem}
  The lemma is in fact purely local and the same argument proves the analogue for $M(F\otimes_\Q\R)$-representations.
  For normalized induction, one can prove similar statements with L- in place of C-.
\end{rem}

\begin{proof}
  We may assume $F=\Q$ by reducing the general case via restriction of scalars. Let $T$ be a maximal torus of $M$ over $\C$ and $B$ a Borel subgroup of $G$ over $\C$ containing $T$. Put $B_M:=M\cap B$. Then $\rho,\rho_M\in X^*(T)\otimes_\Z \Q$ are defined. Let $\chi_{\Pi_{M,\infty}}$ (resp. $\chi_{\Pi_\infty}$) denote the character of $X_*(\hat{T})=X^*(T)$ associated to $\varphi_{\Pi_{M,\infty}}$ (resp. $\varphi_{\Pi_{\infty}}$) as in Definition \ref{d:C-algebraic} well-defined up to $W(M,T)$-conjugacy (resp. $W(G,T)$-conjugacy).  Let $\lambda_{\Pi_{M,\infty}}$ (resp. $\lambda_{\Pi_\infty}$) denote the infinitesimal character of $\Pi_{M,\infty}$ (resp. $\Pi_\infty$).

  The condition of the lemma tells us that $\lambda_{\Pi_\infty}$ and $\lambda_{\Pi_{M,\infty}}+(\rho-\rho_M)$ are in the same $W(G,T)$-orbit in $X^*(T)\otimes_\Z \C$. On the other hand, $\lambda_{\Pi_{M,\infty}}$ and $\chi_{\Pi_{M,\infty}}$ are in the same $W(M,T)$-orbit and similarly $\lambda_{\Pi_{\infty}}$ and $\chi_{\Pi_{\infty}}$ are in the same $W(G,T)$-orbit (\cite[Prop 7.4]{Vog93}). Therefore if $\Pi_M$ is C-algebraic then so is $\Pi$.

  We check that $\Pi$ is strongly C-arithmetic if $\Pi_M$ is strongly C-arithmetic. Let $S$ be the finite set of places (including $S_\infty$) outside which $\Pi_M$ is unramified. The assumption tells us that $\Pi_v$ is a subquotient of $\ind^G_P(\Pi_{M,v})$ at every finite place $v$. Hence $\Pi_v^\sigma$ is a subquotient of $\ind^G_P(\Pi_{M,v}^\sigma)$ at every $v$ for every $\sigma\in \Aut(\C)$. (The latter implication fails if normalized induction was used and if $\sigma$ does not fix $q_v^{1/2}$.) For $v\notin S$ and $\sigma\in \Aut(\C/\Q(\Pi_M))$ we see that $\Pi_v$ and $\Pi^\sigma_v$ are isomorphic as both of them are the unique unramified subquotient of $\ind^G_P(\Pi_{M,v})$.
  For finite $v\in S$, $\Pi_v$ is C-arithmetic since $\sigma\in \Aut(\C/\Q(\Pi_M))$ permutes the finitely many irreducible subquotients of $\ind^G_P(\Pi_{M,v}^\sigma)$. Therefore $\Q(\Pi)$ is contained in the finite field extension of $\Q(\Pi_M)$ generated by $\Q(\Pi_v)$ for $v\in S$, hence $\Pi$ is strongly C-arithmetic.

  The above proof also shows that if $\Pi_M$ is C-arithmetic then $\Pi$ is C-arithmetic.
\end{proof}

\subsection{Rationality for cohomological representations}\label{sub:Clozel-BHR}

%

  Temporarily let $G$ be a connected reductive group over $\Q$. Let $\pi$ be an automorphic representation of $G(\A)$. Let $K_\infty$ be a subgroup of $G(\R)$ whose image in $G^{\ad}(\R)$ is a maximal compact subgroup. Let $K^0_\infty$ be the neutral component of $K_\infty$ with respect to the real topology. Let $Q$ be a parabolic subgroup of $G(\C)$ with Levi component $K_{\infty,\C}$. Put $\fkg:=\Lie G(\C)$ and $\fkq:=\Lie Q(\C)$.

\begin{defn}\label{d:cohomological}
  We say that $\pi$ is \textbf{cohomological} (resp. $\ol{\partial}$-\textbf{cohomological}) if $H^i(\fkg,K^0_\infty,\pi_\infty\otimes \xi)\neq 0$ (resp. $H^i(\fkq,K^0_\infty,\pi_\infty\otimes \xi)\neq 0$) for some $i\ge 0$ and some irreducible algebraic representation $\xi$ of $G(\C)$ (resp. $K_{\infty,\C}$). In this case $\pi$ is said to be $\xi$-cohomological (resp. $\xi$-$\ol{\partial}$-cohomological).
\end{defn}

\begin{lem}\label{l:GL(n)-cohomlogical}
  If $G=GL_n$ then every cuspidal regular C-algebraic automorphic representation $\pi$ of $G(\A_F)$ is cohomological.
\end{lem}

\begin{proof}
  Follows from \cite[Lem 3.14]{Clo90}.
\end{proof}

\begin{rem}
  If $\pi_\infty$ is an \textit{arbitrary} regular C-algebraic representation of $GL_n(\R)$, $GL_n(\C)$, or a product thereof, then there is no reason for $\pi_\infty$ to have non-vanishing cohomology as in Definition \ref{d:cohomological}. What makes the above lemma work is the condition that $\pi_\infty$ is (essentially) tempered, which is implied by the cuspidality of $\pi$, cf. \cite[Lem 4.9]{Clo90}.
\end{rem}

  From now on, let $F$ be a number field and $G$ a connected reductive group over $F$. By applying the above definition to $\Res_{F/\Q}G$ we define $K_\infty$, $Q$, $\fkg$, $\fkq$ and make sense of ($\ol{\partial}$-)cohomological representations. In light of the above remark, a sensible generalization of Lemma \ref{l:GL(n)-cohomlogical} would be the following assertion: For any connected reductive group $G$ over $F$, every cuspidal regular C-algebraic automorphic representation of $G(\A_F)$ is cohomological if its infinite component is tempered. (For a general $G$ the latter condition is not a consequence of global cuspidality. Early counterexamples are due to Kurokawa, Howe and Piatetski-Shapiro.)
  We believe that the assertion is true but were not able to verify it. In the converse direction we have

\begin{lem}\label{l:coh=C-alg}
  Any cohomological automorphic representation $\pi$ of $G(\A_F)$ is C-algebraic.
\end{lem}

\begin{proof} We may assume $F=\Q$. Let $T$ be a maximal torus over $\C$ and $B$ a Borel subgroup of $G$ over $\C$ containing $T$.
  Let $\lambda_{\xi^\vee}\in X^*(T)$ be the highest weight vector for $\xi^\vee$ with respect to $(B,T)$ where $\xi$ is as above. Let $\chi_{\pi_\infty}\in X_*(\hat{T})\otimes_\Z \C=X^*(T)\otimes_\Z \C$ be the character determined by $\varphi_{\pi_\infty}|_{W_{\C}}$ as in Definition \ref{d:C-algebraic}. Then $\chi_{\pi_\infty}$ is well-defined up to $W(G,T)$-conjugacy.
  If $\pi$ is $\xi$-cohomological then the infinitesimal character of $\pi_\infty$ is the same as that of $\xi^\vee$, namely $\lambda_{\xi^\vee}+\rho$. Hence $\chi_{\pi_\infty}$ and $\lambda_{\xi^\vee}+\rho$ are in the same $W(G,T)$-orbit (\cite[Prop 7.4]{Vog93}). We conclude that $\chi_{\pi_\infty}-\rho\in X^*(T)$ independently of the choices so far and that $\pi_\infty$ is C-algebraic.
\end{proof}

  Roughly speaking, cohomological (cuspidal) automorphic representations are important in that they are realized in the Betti cohomology (or \'{e}tale cohomology via comparison theorem) of locally symmetric quotients associated with $G$. This plays a fundamental role in Clozel's work for $G=GL_n$, cf. Remark \ref{r:Clozel} below. In work of Blasius-Harris-Ramakrishnan (cf. Proposition \ref{p:BHR} below) they prove C-arithmeticity by realizing cuspidal automorphic representations in the coherent cohomology of Shimura varieties, which is possible for $\ol{\partial}$-cohomological representations.

  We would like to show C-arithmeticity for a large class of cohomological representations by realizing them in the Betti cohomology of locally symmetric quotients with coefficient sheaves defined over number fields. This must be well known to experts, the idea being similar to \cite{Wal85a} and \cite{Clo90}, but we provide some details as there does not seem to be a handy reference for the general case.

  For any sufficiently small open compact subgroup $U\subset G(\A_F^\infty)$, consider the manifold $$S_U(G):=G(F)\bs G(\A_F)/ UK^0_{\infty}$$ with finitely many connected components. Let $\xi$ be an irreducible algebraic representation of $\Res_{F/\Q}G$ over $\C$ and denote by $\cL_\xi$ the associated local system of $\C$-vector spaces on $S_U(G)$. (By abuse of notation we omit the reference to $U$ in $\cL_\xi$.) Such a $\xi$ admits a model $\xi_E$ over a number field $E$ (so that $\xi_E\otimes_E \C\simeq \xi$) and one can use the highest weight theory to show that $\cL_\xi$ also admits a model $\cL_{\xi,E}$, a local system of $E$-vector spaces. For $i\ge 0$ define
\beq\label{e:H(Betti)}
  H^i(S(G),\cL_\xi):=\dirlim{U} H^i(S_U(G),\cL_\xi)
\eeq
  and similarly $H^i(S(G),\cL_{\xi,E})$. The usual Hecke action equips $H^i(S(G),\cL_\xi)$ (resp. $H^i(S(G),\cL_{\xi,E})$) with the structure of admissible $\C[G(\A^\infty)]$-module (resp. $E[G(\A^\infty)]$-module), where admissibility corresponds to the fact that $H^i(S_U(G),\cL_\xi)=H^i(S(G),\cL_\xi)^U$ is finite dimensional.

   Much work has been done to decompose $H^i(S(G),\cL_\xi)$ by means of automorphic representations. When $S_U(G)$ are compact, Matsushima's formula does the job. Results in the general case are due to Franke, Harder, Li, Schwermer and others. 
   This enables us to show C-arithmeticity for cuspidal representations.

\begin{prop}\label{p:arithmetic-cohomological}
  Let $\pi$ be a cuspidal $\xi$-cohomological automorphic representation of $G(\A_F)$. Then
 \benu
 \item $\pi^\infty$ is a $G(\A^\infty_F)$-module direct summand of $H^i(S(G),\cL_\xi)$ for some $i\ge 0$.
 \item $\pi$ is strongly C-arithmetic.
 \eenu
\end{prop}

\begin{rem}\label{r:Clozel}
  Clozel has shown this for general linear groups (\cite[Th 3.13, Lem 3.14, 3.15]{Clo90}). We are adapting his ideas to the case of arbitrary reductive groups. (See also the last paragraph of \S7 in \cite{BG} for the case of trivial coefficients.)
\end{rem}

\begin{rem}
  When $G=GL_n$ we know moreover that $\Q(\pi)$ is a totally real or CM field, cf. \cite[Cor 6.2.3]{Pat}. The argument requires to know the subtle point that twists of $\pi^\infty$ by $\Aut(\C)$ are finite parts of automorphic representations of $G(\A_F)$. As this is not known in general, it seems difficult to check whether $\Q(\pi)$ is a totally real or CM field for an arbitrary reductive group. However see Proposition \ref{p:BHR}.(ii) below.
\end{rem}

\begin{proof}

  Part (i) follows from the description of the cuspidal part of $H^i(S(G),\cL_\xi)$ via Lie algebra cohomology (\cite[(13.6)]{Sch10}, \cite{FS98}). Note that the cuspidal part is a direct summand, cf. page 242 of \cite{Sch10}. Part (ii) can be shown by arguing as in the proof of \cite[Prop 3.16]{Clo90}. The argument is sketched here for the convenience of the reader.

  Let $U=\prod_{v\nmid \infty} U_v\subset G(\A^\infty_F)$ be a sufficiently small open compact subgroup such that $(\pi^\infty)^U\neq 0$. Then $(\pi^\infty)^U$ is a direct summand of $H^i(S_U(G),\cL_\xi)$ and moreover irreducible as a $\cH_U(G(\A^\infty_F),\C)$-module. (This follows from the irreducibility criterion of \cite[p.179]{Fla79}.)
  Since the field of definition is the same for $\pi^\infty$ as a $G(\A^\infty_F)$-module and for $(\pi^\infty)^U$ as a $\cH_U(G(\A^\infty_F),\C)$-module, it is enough to show that the isomorphism class of $(\pi^\infty)^U$ is fixed under a finite index subgroup of $\Aut(\C)$.

   We start by finding a model of $(\pi^\infty)^U$ on a $\ol{\Q}$-vector space.
   Burnside's theorem implies that irreducible $\cH_U(G(\A^\infty_F),\C)$-module subquotients of $H^i(S_U(G),\cL_\xi)$ and those of the $\cH_U(G(\A^\infty_F),\ol{\Q})$-module $H^i(S_U(G),\cL_{\xi,E}\otimes_E \ol{\Q})$ correspond bijectively.\footnote{Consider the Jordan-H\"older quotients $M_1,...,M_k$ of $H^i(S_U(G),\cL_{\xi,E}\otimes_E \ol{\Q})$. By Burnside's theorem, the $\ol{\Q}$-algebra morphism from $\cH_U(G(\A^\infty_F),\ol{\Q})$ to $\End_{\ol{\Q}}(M_j)$ is onto. So the Jordan-H\"older quotients remain irreducible after $\otimes_{\ol{\Q}} \C$.} In particular there is an irreducible $\cH_U(G(\A^\infty_F),\ol{\Q})$-module subquotient $W$ of $H^i(S_U(G),\cL_{\xi,E}\otimes_E \ol{\Q})$ such that $W\otimes_{\ol{\Q}} \C\simeq (\pi^\infty)^U$. Since $\sigma\in \Gal(\ol{\Q}/E)$ induces a $\sigma$-linear self-automorphism of $H^i(S_U(G),\cL_{\xi,E}\otimes_E \ol{\Q})$ as a $\cH_U(G(\A^\infty_F),\ol{\Q})$-module, the induced action permutes the irreducible subquotients of $H^i(S_U(G),\cL_{\xi,E}\otimes_E \ol{\Q})$ (the point being that $\cH_U(G(\A^\infty_F),\ol{\Q})$ has a natural $\Q$-structure). We see from the finite-dimensionality of the latter space that the isomorphism class of $W$ is fixed by a finite index subgroup of $\Gal(\ol{\Q}/E)$ as desired.
\end{proof}

\begin{cor}\label{c:Clozel-noncuspidal} Let $M$ be a Levi subgroup of an $F$-rational parabolic subgroup of $G$.
  Any automorphic representation of $G(\A_F)$ appearing as a subquotient of an unnormalized parabolic induction of a cuspidal cohomological automorphic representation of $M(\A_F)$ is C-algebraic and strongly C-arithmetic.
\end{cor}

\begin{proof}
  Immediate from Lemma \ref{l:arithmetic-par-ind}, Lemma \ref{l:coh=C-alg} and Proposition \ref{p:arithmetic-cohomological}.
\end{proof}

  In the rest of this subsection we briefly recall some results of Blasius, Harris and Ramakrishnan for the sake of completeness, even though their results will not be used in this paper. Under a restrictive hypothesis (cf. \cite[\S0.1]{BHR94}), namely that $\Res_{F/\Q}G$ is of hermitian symmetric type so that $G(F\otimes_\Q\R)/K_\infty$ admits a $G(F\otimes_\Q\R)$-invariant complex structure, the three authors have shown:

\begin{prop}\label{p:BHR} Keep the hypothesis in the above paragraph. Let $\pi$ be any automorphic representation of $G(\A_F)$ such that $\pi_\infty$ is a nondegenerate limit of discrete series or a discrete series representation of $G(F\otimes_\Q \R)$ whose restriction to the maximal $\R$-split torus of $(\Res_{F/\Q}G)(\R)$ is algebraic. Then
\benu
\item any such $\pi$ is $\ol{\partial}$-cohomological, C-algebraic and
\item if $\pi$ is moreover cuspidal then $\Q(\pi)$ is either a totally real or a CM field (in particular $\pi$ is strongly C-arithmetic).
\eenu
\end{prop}

\begin{rem}
  One can extend part (ii) beyond the cuspidal case by applying Lemma \ref{l:arithmetic-par-ind} as it was done in Corollary \ref{c:Clozel-noncuspidal}.
\end{rem}

\begin{proof}
 This is Theorems 3.2.1 and 4.4.1 of \cite{BHR94} except for the C-algebraicity of $\pi$, which is easy to deduce from the description of the infinitesimal character of $\pi_\infty$ in \cite[Thm 3.2.1]{BHR94} by an argument as in the proof of Lemma \ref{l:coh=C-alg}. Note that a subfield of a CM field is either totally real or CM.
\end{proof}

\subsection{Satake parameters under functoriality}\label{sub:Satake-parameters}

  Let $H$ and $G$ be connected reductive groups over a number field $F$. We form their $L$-groups using the full Galois group over $F$ rather than a finite Galois group or the Weil group. (Later we use the Weil group in the case of even orthogonal groups. In that case the material of this subsection can still be adapted. See \S\ref{sub:transfer-app}.) Let $\eta:{}^L H\ra {}^L G$ be an L-morphism. Let $(\hat{B}_H,\hat{T}_H)$ (resp. $(\hat{B},\hat{T})$) be a pair of a Borel subgroup of $\hat{H}$ (resp. $\hat{G}$) and a maximal torus contained in it. We may choose $(\hat{B},\hat{T})$ such that $\eta(\hat{T}_H)\subset \hat{T}$ (and $\eta(\hat{B}_H)\subset \hat{B}$ but the latter is unnecessary for us). These data determine $\rho_H\in \frac{1}{2} X_*(\hat{T}_H)$ and $\rho\in \frac{1}{2} X_*(\hat{T})$ as the half sums of all positive coroots in $\hat{T}_H$ and $\hat{T}$, respectively. Moreover $\eta$ induces $\eta_*:X_*(\hat{T}_H)\ra X_*(\hat{T})$.

\begin{defn}\label{d:C-preserving}
    An $L$-morphism $\eta:{}^L H \ra {}^L G$ is said to be \textbf{C-preserving} if $\rho-\eta_*(\rho_H)$ at each $v|\infty$ belongs to $X_*(\hat{T})$ (rather than just $\frac{1}{2}X_*(\hat{T})$).
\end{defn}

In view of Definition \ref{d:C-algebraic}, a C-preserving L-morphism carries L-packets of C-algebraic representations to L-packets of C-algebraic representations at infinite places.
 The C-preserving property does not depend on the choice of maximal tori and Borel subgroups. Indeed one can go between different maximal tori in $\hat{H}$ (resp. $\hat{G}$) by conjugation. Moreover if $\hat{T}_H$ is fixed, another choice of $\hat{B}_H$ changes $\rho_H$ by a Weyl group element $w_H$ for $\hat{H}$, but clearly $w_H\rho_H-\rho_H\in X_*(\hat{T}_H)$ so $\rho-\eta_*(\rho_H)$ is shifted by an element of $X_*(\hat{T})$ (rather than just $\frac{1}{2} X_*(\hat{T})$). A similar argument shows the independence of the choice of $\hat{B}$ as well.

  The aim of this subsection is to show that for a C-preserving $L$-morphism, the transfer of unramified representations is compatible with twisting by field automorphisms of $\C$. We begin with some preparation. Let $S$ be a finite set of places of $F$ containing $S_\infty$ such that $H$, $G$ and $\eta$ are unramified whenever $v\notin S$. From now on assume $v\notin S$. Let $A_v$ be a maximal $F_v$-split torus of $G$, and $T_v$ be the centralizer of $A_v$ in $G$ over $F_v$. Let $B_v$ be a Borel subgroup of $G$ containing $T_v$. Define $\rho_{v}\in \frac{1}{2} X_*(A_v)$ to be the half sum of all $F_v$-rational $B_v$-positive roots relative to $A_v$. Write $q_v^{1/2}$ for the positive square root of $q_v$. Denote by $\sgn_{\sigma,\rho_v}:T_v(F_v)\ra \{\pm1\}$ a character defined via the following composite map
  $$T_v(F_v)\ra T_v(F_v)/T_v(\cO_v) \simeq X_*(A_v) \ra \{\pm 1\}$$
  where $\lambda\in X_*(A_v)$ is sent to $\lambda(\varpi_v)\in T_v(F_v)/T_v(\cO_v)$ under the isomorphism in the middle and to $(\sigma(q_v^{1/2})/q_v^{1/2})^{\lg \lambda,2\rho_{v}\rg}\in \{\pm 1\}$ under the last map. (In particular $\sgn_{\sigma,\rho_v}(\lambda)=1$ if either $q_v^{1/2}\in \Q$ or $\lg \lambda,\rho_{v}\rg\in \Z$.) Likewise $A_{H,v}$, $T_{H,v}$, $B_{H,v}$, $\rho_{H,v}$ and $\sgn_{\sigma,\rho_{H,v}}$ are defined for $H$.
  Write $\delta^{1/2}_{B_v}:T_v(F_v)\ra \R_{>0}^\times$ for the modulus character, which factors through the character $\lambda\mapsto (q_v^{1/2})^{\lg\lambda,2\rho_v\rg}$ from $X_*(A_v)$ to $\R_{>0}^\times$.

\begin{lem}\label{l:twist-unram-rep} Suppose $v\notin S$ and let $\chi_v:T_v(F_v)\ra \C^\times$ be a continuous character.
  If $\pi_v\in \Irr^{\ur}(G(F_v))$ is a subquotient of $\nind^{G(F_v)}_{B_v(F_v)}(\chi_v)$ then for every $\sigma\in \Aut(\C)$, $\pi_v^\sigma$ is a subquotient of $\nind^{G(F_v)}_{B_v(F_v)}(\chi^\sigma_v\otimes \sgn_{\sigma,\rho_v})$. The exact analogue holds true for $H$.
\end{lem}

\begin{proof}
  Recall that the unnormalized parabolic induction commutes with $\sigma$-twisting, cf. Lemma \ref{l:arithmetic-par-ind}. So $\pi^{\sigma}_v$ is an unramified subquotient of the following representation: all inductions below are from $B_v(F_v)$ to $G(F_v)$.
  $$\nind(\chi_v)^{\sigma} = \ind(\chi_v\otimes \delta_{B_v}^{1/2})^{\sigma}
  = \ind(\chi^{\sigma}_v\otimes (\delta^{1/2}_{B_v})^{\sigma}) = \nind(\chi^{\sigma}_v\otimes (\delta^{1/2}_{B_v})^{\sigma}/\delta^{1/2}_{B_v})).$$
  By definition $(\delta^{1/2}_{B_v})^{\sigma}/\delta^{1/2}_{B_v}=\sgn_{\sigma,\rho_v}$. Since a principal series representation has a unique unramified subquotient, the first part of the lemma follows. The argument for $H$ is the same.
\end{proof}

  We have that $\eta$ is unramified at $v\notin S$, so it comes from a map on $\Fr_v$-cosets $\hat{H}\rtimes \Fr_v\ra \hat{G}\rtimes \Fr_v$, again denoted $\eta$. The Satake isomorphism provides a canonical bijection between the set of $\hat{G}$-conjugacy classes in $\hat{G}\rtimes \Fr_v$ (resp. ($\hat{H}$-conjugacy classes in $\hat{H}\rtimes \Fr_v$) with $\Irr^{\ur}(G(F_v))$ (resp. $\Irr^{\ur}(H(F_v))$). Write $$\eta_*:\Irr^{\ur}(H(F_v))\ra\Irr^{\ur}(G(F_v))$$ for the map induced by $\eta$.

\begin{lem}\label{l:bound-fiber}
  Let $v\notin S$ and suppose that $\eta:{}^L H \ra {}^L G$ is an $L$-morphism with finite kernel. (So $\eta$ is unramified.) Then there exists $N\in \Z_{>0}$ such that every fiber of $\eta_*$ has cardinality at most $N$.
\end{lem}

\begin{rem}
  The $N$ in the lemma can be chosen independently of $v$. For this observe that the order of the Weyl group in $G$ is clearly bounded independently of $v$ and that the size of the kernel of $\eta_{T,*}$ is also uniformly bounded since there are only finitely many $\Fr_v$-actions on $\hat{T}_{H,v}$ and $\hat{T}_v$ as $v$ varies (up to Weyl group actions).
\end{rem}

\begin{proof}

  Obviously the proof is reduced to the case where $\eta$ is injective, which will be assumed throughout.
  Let ${}^L B_{H,v}$ be a Borel subgroup of ${}^L H$ relative to the base field $F_v$ (see \cite[\S3]{Bor79} for this and other related notions in the proof). Then $\hat{B}_{H,v}:={}^L B_{H,v}\cap \hat{H}$ is a Borel subgroup of $\hat{H}$. Since $\eta(\hat{B}_{H,v})$ is a closed solvable subgroup of $\hat{G}$, it is contained in some Borel subgroup $\hat{B}_v$ of $\hat{G}$. Then the normalizer ${}^L B_v$ of $\hat{B}_v$ in $^L G$ is a Borel subgroup of $^L G$. Let $i_{H}:{}^L B_{H,v}\hra {}^L H$ and $i:{}^L B_v\hra {}^L G$ denote the inclusions. Write $\hat{T}_{H,v}$ and $\hat{T}_v$ for the maximal tori in $\hat{B}_{H,v}$ and $\hat{B}_v$. The normalizer $^L T_{H,v}$ of $\hat{T}_{H,v}$ in $^L B_{H,v}$ is a Levi subgroup of $^L B_{H,v}$, and similarly we have a Levi subgroup $^L T_v $ of $^L B_v$. We can identify $^L T_{H,v}$ and $^L T_v$ with the $L$-groups for minimal Levi subgroups $T_{H,v}$ and $T_v$ of $H$ and $G$ over $F_v$, respectively. Clearly we have $\eta({}^L B_{H,v})\subset {}^L B_v$ and so $\eta({}^L T_{H,v})\subset {}^L T_v$. Denote the induced map ${}^L T_{H,v}\ra {}^L T$ by $\eta_{T}$. Notice that $i$, $i_{H}$ and $\eta_T$ are unramified. We have a commutative diagram as below on the left, which induces a commutative diagram on the unramified spectra.
    \beq\label{e:unr-diagram} \xymatrix{{}^L T_{H,v} \ar[r]^-{\eta_{T}} \ar[d]^-{i_{H}} & {}^L T_v \ar[d]^-{i} \\
  {}^L H \ar[r]^-{\eta}  & {}^L G  }, \qquad \xymatrix{\Irr^{\ur}(T_{H,v}(F_v)) \ar[r]^-{\eta_{T,*}} \ar[d]^-{i_{H,*}} & \Irr^{\ur}( T_v(F_v)) \ar[d]^-{i_{*}} \\
  \Irr^{\ur}(H(F_v)) \ar[r]^-{\eta_*}  & \Irr^{\ur}(G(F_v)) }\eeq
  We know (\cite[\S10.4]{Bor79}) how to describe $i_*$ and $i_{H,*}$ using parabolic induction: $i_*(\chi_v)$ is the unique unramified subquotient of $\nind(\chi_v)$ and the analogue is true for $i_{H,*}$. According to the well known classification of unramified representations, we know firstly that $i_*$ and $i_{H,*}$ are surjective and secondly that the fiber of $i_*$ (resp. $i_{H,*}$) has cardinality at most the order of the Weyl group for $T_v$ in $G$ (resp. for $T_{H,v}$ in $H$). This order can be bounded uniformly in $v$. On the other hand $\eta_{T,*}$ has finite fibers. Indeed \cite[\S9.4]{Bor79} identifies $\eta_{T,*}$ with a group homomorphism $$\hat{T}_{H,v}/(\Fr_v-1)\hat{T}_{H,v}\ra \hat{T}_v/(\Fr_v-1)\hat{T}_v$$
  ($\Fr_v$ denoting the geometric Frobenius action), and the above map has finite kernel (\cite[\S6.3, (2)]{Bor79}). All in all, the fibers of $\eta_*$ are finite.

\end{proof}

\begin{lem}\label{l:unram-functoriality} Suppose that $\eta:{}^L H \ra {}^L G$ is C-preserving. Let $v\notin S$.
For each $\pi_{H,v}\in \Irr^{\ur}(H(F_v))$,
\benu
\item  $(\eta_* \pi_{H,v})^{\sigma}=\eta_*(\pi_{H,v}^{\sigma})$.
\item $\Q(\eta_* \pi_{H,v})\subset \Q(\pi_{H,v})$.
\item If $\eta$ has finite kernel and $N$ is as in Lemma \ref{l:bound-fiber} then $[\Q(\pi_{H,v}):\Q(\eta_* \pi_{H,v})]\le N! $.
\eenu
\end{lem}

\begin{proof}
  Let us prove (i). Adopt the setting in the proof of the last lemma. The first observation is that when $H=T_H$ and $G=T$ are tori, (i) follows from the fact that $\eta_*$ is naturally defined over $\Q$ since an algebraic map ${}^L T_H\ra {}^L T$ corresponds to a $\Fr_v$-equivariant map $X_*(\hat{T}_H)=X^*(T_H)\ra X_*(\hat{T})=X^*(T)$. Now consider the general case. For simplicity of notation only in this proof, we use $\nind$ to mean the unique unramified subquotient of the normalized induction. Now by surjectivity of $i_{H,*}$ write $\pi_{H,v}=i_{H,*}(\chi_{H,v})=\nind^H_{B_{H,v}}(\chi_{H,v})$ for a smooth character $\chi_{H,v}:T_{H,v}(F_v)\ra \C^\times$. Put $\chi_v:=\eta_{T,*}(\chi_{H,v})$. From the case of tori we know that $$\eta_{T,*}(\chi^\sigma_{H,v})=\chi_v^\sigma.$$
  Using Lemma \ref{l:twist-unram-rep} and the commutativity of \eqref{e:unr-diagram} we compute
  \beq\label{e:pf-2.25-1}(\eta_*\pi_{H,v})^\sigma=(\eta_*i_{H,*}\chi_{H,v})^\sigma = (i_*\chi_v)^\sigma=\nind(\chi_v)^\sigma= \nind(\chi^\sigma_v\otimes \sgn_{\sigma,\rho_v}).\eeq
  Similarly, noting in addition that $\eta_{T,*}$ is a homomorphism,
  \beq\label{e:pf-2.25-2}\eta_*(\pi_{H,v}^\sigma)=\eta_*(i_{H,*}(\chi_{H,v})^\sigma)=\eta_*(i_{H,*}(\chi^\sigma_{H,v}\otimes \sgn_{\sigma,\rho_{H,v}})\eeq
  $$=\nind(\eta_{T,*}(\chi^\sigma_{H,v}\otimes \sgn_{\sigma,\rho_{H,v}}))=\nind(\eta_{T,*}(\chi^\sigma_{H,v})\otimes\eta_{T,*}( \sgn_{\sigma,\rho_{H,v}}))=\nind(\chi^\sigma_v\otimes \sgn_{\sigma,\eta_*(\rho_{H,v})}).$$
  Since $\eta$ is C-preserving, $\sgn_{\sigma,\rho_v}=\sgn_{\sigma,\eta_*(\rho_{H,v})}$ and thus the proof of (i) is complete.

  Part (ii) is clear from (i). To verify (iii), put $\pi_v:=\eta_*\pi_{H,v}$. By (i), $\pi_v^\sigma\in \eta_*^{-1}(\pi_v)$ for every $\sigma\in \Aut(\C/\Q(\pi_v))$. This yields a homomorphism from $\Aut(\C/\Q(\pi_v))$ to the permutation group on $\eta_*^{-1}(\pi_v)$. Since $|\eta_*^{-1}(\pi_v)|\le N$, the kernel has finite index at most $N!$. This proves (iii).
\end{proof}

\begin{cor} Keep the assumptions of Lemma \ref{l:unram-functoriality}.
\benu
\item Let $v\notin S$. If $\pi_{H,v} \in\Irr^{\ur}(H(F_v))$ is C-arithmetic then $\eta_* \pi_{H,v}$ is C-arithmetic. The converse is true if there is a constant $\kappa$ such that every fiber of $\eta_*:\Irr^{\ur}(H(F_v))\ra \Irr^{\ur}(G(F_v))$ has cardinality at most $\kappa$.
\item Let $\pi_H$ and $\pi$ be automorphic representations of $H(\A_F)$ and $G(\A_F)$ such that $\pi_v=\eta_*(\pi_{H,v})$ for all $v\notin S$. If $\pi_H$ is C-arithmetic then so is $\pi$.
\eenu
\end{cor}

\begin{proof}
  Immediate from (ii) and (iii) of Lemma \ref{l:unram-functoriality}.
\end{proof}

%
%
%
%
%

\begin{rem} Compare our results with the lemmas 6.2 and 6.3 of \cite{BG}, where it is shown that any L-morphism $\eta:{}^L H\ra {}^L G$ carries L-algebraic (resp. L-arithmetic) representations to L-algebraic (resp. L-arithmetic) representations. (It is worth noting that they use Galois groups to form the L-groups; it can fail to be true if Weil groups are used.) One could try to derive our results in \S\ref{sub:transfer-app} directly from their results by twisting but this is not automatic for two reasons: some groups lack twisting elements (in the sense of \cite[\S5.2]{BG}) and some others admit no L-algebraic representations at all. 
\end{rem}

\section{Purity and rationality of local components}\label{s:purity}

  The contents of this section are purely local and the following notation will be used.
\bit
\item $K$ is a finite extension of $\Q_p$ with residue field $\F_q$, $\cO_K$ is its integer ring, and $\Frob_K$ is the geometric Frobenius element in $\Gal(K^{\ur}/K)$.	
\item $W_K$ and $I_K$ are the Weil and inertia groups of $K$,
\item $\Omega$ is an algebraically closed field of characteristic 0 with the same cardinality as $\C$ (usually $\Omega$ is taken to be $\C$ or $\ol{\Q}_l$ for a prime $l$),
\item $\Frob_K\in W_K/I_K$ is the geometric Frobenius element,
\item $v:W_K\ra \Z$ is defined as $W_K\twoheadrightarrow W_K/I_K\simeq \Z$ where the last isomorphism carries $\Frob_K$ to 1,
\item $|\cdot|_{W_K}:W_K\ra \Q^\times$ is a character given by $\tau\mapsto q^{-v(\tau)}$.
  \item $\scusp(\pi)$ denotes the supercuspidal support of $\pi\in \Irr(G(K))$.
  \eit

\subsection{Pure Weil-Deligne representations}\label{sub:pure-WD}

  Our basic definitions are based on those of \cite[p.471]{TY07}. Their definition is slightly more general in that the weight is allowed to be a real number. For our purpose it suffices to consider only integral weights.

  A Weil-Deligne representation (or \textbf{WD representation} for simplicity) of $W_K$ (over $\Omega$) is a triple $(V,\rho,N)$ where $V$ is a finite dimensional $\Omega$-vector space, $\rho:W_K\ra GL(V)$ is a group homomorphism such that $\rho(I_K)$ is finite, and $N\in \End_\Omega(V)$ is a nilpotent operator such that $\rho(\tau)N\rho(\tau)^{-1}=|\tau|_{W_K} N$. It is said to be unramified if $\rho(I_K)$ is the identity and $N=0$, \textbf{Frobenius semisimple} (or ``F-ss'' for short) if $\rho$ is semisimple, and irreducible if $\rho$ is irreducible and $N=0$. Let $(V,\rho,N)^{\Fss}:=(V,\rho^{\semis},N)$ denote the Frobenius semisimplification of $(V,\rho,N)$, where $\rho^{\semis}$ is defined as follows: Fix a lift $\phi\in W_K$ of $\Frob_K$ and let $\rho(\phi)=su$ be the Jordan decomposition with semisimple part $s$. Then $\rho^{\semis}(\phi^n \tau):=s^n\rho(\tau)$ for all $n\in \Z$ and for all $\tau\in I_K$, which defines $\rho^{\semis}$ independently of the choice.

    Let $n\in \Z_{\ge1}$. For a continuous $l$-adic representation $r:\Gal(\ol{K}/K)\ra GL_n(\ol{\Q}_l)$, there is a standard way (depending on whether $l\neq p$ or $l=p$) to associate a Weil Deligne representation $WD(r)$ of $W_K$ as explained on pp.467-470 of \cite{TY07}. (One can view $WD$ as a functor on appropriate categories.)

  We recall the key definitions about purity. Let $w\in \Z$. A $q$-\textbf{Weil number} (resp. \textbf{integer}) of weight $w$ is an algebraic number $\alpha$ (resp. an algebraic integer $\alpha$) such that $|\iota(\alpha)|=q^{w/2}$ for any field embedding $\iota:\ol{\Q}\hra \C$. A WD representation $(V,\rho,N)$ of $W_K$ is \textbf{strictly pure} of weight $w$ if every eigenvalue of the image under $\rho$ of some (hence every) lift of $\Frob_K$ is a $q$-Weil number of weight $w$. We say that $(V,\rho,N)$ is \textbf{mixed} if there exists an increasing filtration of sub WD representations $\{\Fil_iV\}_{i\in \Z}$ on $V$ such that $\Fil_i V=0$ if $i\ll 0$, $\Fil_i V=V$ if $i\gg0$ and $\gr_i V:=\Fil_i V/\Fil_{i+1} V$ is strictly pure of weight $i$ for every $i\in \Z$. A mixed $(V,\rho,N)$ admits a unique filtration such that $N(\Fil_i V)\subset \Fil_{i-2} V$. Let us say $(V,\rho,N)$ is \textbf{pure} of weight $w$ if it is mixed and if $N^i:\gr_{w+i}V\ra \gr_{w-i} V$ is an isomorphism for every $i$ with respect to the unique filtration just mentioned. More generally let $\underline{w}$ be a finite multi-set such that the elements of $\underline{w}$ are distinct integers $w_1,...,w_r$ with multiplicities $m_1,...,m_r$. Then $(V,\rho,N)$ is said to be \textbf{pure of weight $\underline{w}$} if $V=\oplus_{i=1}^r (V_i,\rho_i,N_i)$ with each $(V_i,\rho_i,N_i)$ being pure of weight $w_i$ and of dimension $m_i$. (If so, we have in particular $|\underline{w}|=m_1+\cdots+m_r=\dim V$.)
  Finally a mixed $(V,\rho,N)$ is \textbf{integral} if for some (hence every) lift $\phi\in W_K$ of $\Frob_K$, every eigenvalue of $\phi$ on $V$ is an algebraic \emph{integer} (so that every eigenvalue of $\phi$ on $\gr_i V$ is a $q$-Weil \emph{integer} of weight $i$).

  The above definitions are motivated by Deligne's weight-monodromy conjecture in its integral form (cf. \cite{Del-Hodge1}, \cite[Conj 0.3, 0.5]{Sai03}). The conjecture is equivalent to the one without Frobenius semisimple/semisimplification in the statement.

\begin{conj}\label{c:weight-monodromy-integrality}(cf. \cite[Conj 0.3, 0.5]{Sai03}) Let $l$ be any prime (which could be equal to $p$).
  Let $(V,\rho,N)$ be an F-ss WD representation on a $\ol{\Q}_l$-vector space. If $(V,\rho,N)$ is a subquotient of $WD(H_{\et}^i(X\times_K \ol{K},\ol{\Q}_l))^{\Fss}$ for some proper smooth scheme $X$ over $K$ then it is pure of weight $i$ and integral.
\end{conj}

It is worth noting that when $X$ has a proper smooth integral model over $\cO_K$ and $l\neq p$, the conjecture is known by Deligne's work on the Weil conjectures. In that case the WD representation below is unramified and strictly pure of weight $i$. In the non-smooth (bad reduction) case the conjecture is known when $\dim X\le 2$ by Rapoport and Zink and in some special cases, for instance for certain Shimura varieties. A recent breakthrough by Scholze (\cite{Sch-WM}) provides a proof for any complete intersection in a projective smooth toric variety.
The converse of Conjecture \ref{c:weight-monodromy-integrality}, which is not as deep as the original conjecture, also seems true. (A proof was announced by Teruyoshi Yoshida but has not appeared in print at the time of writing.)

Motivated by Conjecture \ref{c:weight-monodromy-integrality} (as well as its converse) and the Fontaine-Mazur conjecture (\cite[Conj 1]{FM95}, also see \cite[Conj 1.3]{Tay04}), we speculate on the following \emph{global} conjecture, which in particular slightly refines the conjecture by Fontaine and Mazur in a sign aspect. More precisely, their conjecture says that (i) and (iii) below are equivalent if nonnegativity is dropped in (i) and a Tate twist of cohomology is allowed in (iii). The conjecture can be stated for all primes $l$ simultaneously in the language of compatible systems, cf. \cite{Tay04}.


\begin{conj}\label{c:integrality} Let $F$ be a number field and $\rho:\Gal(\ol{F}/F)\ra GL_n(\ol{\Q}_l)$ a continuous semisimple representation unramified outside finitely many places. The following are equivalent.
\benu
\item $\rho$ is de Rham at every place $v|l$ with nonnegative Hodge-Tate weights.\footnote{In our convention the cyclotomic character has Hodge-Tate weight $-1$ (rather than $1$).}
\item the WD representation associated with $\rho$ at every finite place $v$ is integral and pure of weight $\underline{w}$, which is independent of $v$, has entries in $\Z$ and satisfies $|\underline{w}|=n$.
\item $\rho$ appears as a subquotient of $\oplus_{i\ge 0} H_{\et}^i(X\times_F \ol{F},\ol{\Q}_l)$ for some proper smooth scheme $X$ over $F$.
\eenu
\end{conj}
When $\rho$ is furthermore irreducible, we may replace the condition in (ii) with ``integral and pure of weight $w$ for some $w\in \Z$'', and the condition in (iii) with ``...a subquotient of $H_{\et}^i(X\times_F \ol{F},\ol{\Q}_l)$ for some $i\ge 0$...''. (For a given $\rho$ the corresponding $w$ and $i$ are expected to be equal, cf. Conjecture \ref{c:weight-monodromy-integrality}.)

 In Proposition \ref{p:Galois-reps} below we derive a partial result toward Conjecture \ref{c:integrality} from the well known results concerning the construction of Galois representations from automorphic representations. That result will be a key to the finiteness result of \S\ref{sub:sparsity}, where the role of integrality will become clear. This was our original motivation. However Conjecture \ref{c:integrality} is interesting in its own right and we plan to discuss it in more on some other occasion.


\begin{rem}
 Part (ii) may be equivalent to (ii)$'$ below, allowing to exclude finitely many $v$:
\benu
\item[(ii)$'$] the WD representation associated with $\rho$ at \emph{almost} every finite $v$ is pure and integral.
\eenu
Conjecture \ref{c:alg-arith} suggests that it would also be equivalent to:
\benu
\item[(ii)$''$] the WD representation associated with $\rho$ at (almost) every finite $v$ is integral and has its field of rationality contained in some number field $E$ independent of $v$.
\eenu
\end{rem}

\begin{rem}
   We are reduced to a more standard conjecture if we get rid of ``with nonnegative Hodge-Tate weights'' in (i), ``integral'' in (ii) and allow a Tate twist in (iii). In this form we already mentioned that the equivalence of (i) and (iii) is exactly the Fontaine-Mazur conjecture.
\end{rem}

\begin{rem}
  The implication (iii)$\Rightarrow$(i) is known by results of $p$-adic Hodge theory (the solution of the $C_{\mathrm{pst}}$ conjecture and comparison of filtrations in complex and $p$-adic Hodge theories) and the fact that the Hodge filtration on $H_{\et}^i(X\times_F \C,\ol{\Q}_l)$ has jumps only in nonnegative indices with respect to any field embedding $F\hra \C$. According to Conjecture \ref{c:weight-monodromy-integrality}, (iii) should imply (ii). Finally we remark that (i)$\Leftrightarrow$(ii) may be viewed as the arithmetic analogue of Conjecture \ref{c:alg-arith}.
\end{rem}

\begin{rem}
  When $\rho$ is associated with a (classical) cuspidal holomorphic eigenform $f=\sum_{n\ge 1} a_n q^{n}$ of weight $k\in \Z_{\ge1}$ with $a_1=1$ so that $a_n$ are algebraic integers for all $n\ge 1$, then (under a suitable normalization) $\rho$ satisfies (i), (ii) and (iii) with Hodge-Tate weights 0 and $k-1$. Now assume that $a_n\in \Z$ for all $n$. The equivalence (i)$\Leftrightarrow$(ii)$'$, applied to the twist of $\rho$ by the cyclotomic character, amounts to the assertion that $f$ is ordinary, i.e. $a_p$ is a $p$-unit for infinitely many primes $p$.
\end{rem}


  It is useful to know a preservation property under base field extensions.

\begin{lem}\label{l:BC}
  Let $\tilde{\rho}=(V,\rho,N)$ be a WD representation of $W_K$, and $L/K$ be a finite extension. Then $\tilde{\rho}|_{W_{L}}$ is pure (resp. integral) if and only if $\tilde{\rho}$ is pure (resp. integral).
\end{lem}

\begin{proof}
  Straightforward. (The preservation of purity is Lemma 1.4.2 of \cite{TY07}).
\end{proof}

  Given $(V,\rho,N)$ as above and $s\in \Z_{\ge1}$, one constructs a new Weil-Deligne representation $$\Sp_s(V):=(V^s,\rho|\cdot|_{W_K}^{s-1}\oplus \cdots \oplus \rho|\cdot|_{W_K}\oplus \rho,N)$$ such that $N:\rho|\cdot|_{W_K}^{i}\isom \rho|\cdot|_{W_K}^{i-1}$ for $i=1,...,s-1$ and $N=0$ on $\rho$. Note that $\Sp_s(V)$ is uniquely determined up to isomorphism.
  If $(V,\rho,N)$ is pure of weight $w$ then $\Sp_s(V)$ is pure of weight $w+s-1$.

\begin{lem}\label{l:pure-shape}
  Let $n\ge 1$ and $(V,\rho,N)$ be an $n$-dimensional F-ss WD representation of $W_K$. Then there exist
   \bit
   \item $m\in \Z_{\ge1}$, $s_1,...,s_m\in \Z_{\ge 1}$ and
   \item a collection of irreducible $n_i$-dimensional F-ss WD representations $(V_i,\rho_i,0)$, $i=1,...,m$,
   \eit
  such that $V=\oplus_{i=1}^m \Sp_{s_i}(V_i)$. Moreover if $(V,\rho,N)$ is pure of weight $w\in \Z$ then each $V_i$ is strictly pure of weight $w-s_i+1$.
\end{lem}

\begin{proof}
  The first assertion follows from the standard fact that any indecomposable F-ss WD representation is of the form $\Sp_s(V)$ for an irreducible F-ss WD representation $V$. If $(V,\rho,N)$ is pure of weight $w$ then so is each $\Sp_{s_i}(V_i)$. From this and the definition of $\Sp_{s_i}(V_i)$ it is elementary to verify that $V_i$ is strictly pure of weight $w-s_i+1$.
\end{proof}

  Pure Weil-Deligne representations enjoy a remarkable rationality property of importance to us.

\begin{lem}\label{l:pure-WD-rat}
  A pure F-ss WD representation $(V,\rho,N)$ of $W_K$ (of some weight $w\in \Z$) has a number field as a field of rationality.
\end{lem}

\begin{proof}
  By Lemma \ref{l:pure-shape} it suffices to treat $\Sp_{s}(V)$ when $(V,\rho,0)$ is an irreducible F-ss WD representation which is strictly pure of weight $w\in \Z$. Clearly $\Sp_{s}(V)$ can be defined over the same number field over which $(V_i,\rho_i,N_i)$ is defined. Hence we are further reduced to showing that $(V,\rho,0)$ has a number field as a field of rationality when it is irreducible and strictly pure of some weight $w\in \Z$.

  It is enough to verify that the trace function $T:=\tr \rho:W_K\ra \Omega$ has image contained in a finite extension of $\Q$ in $\Omega$. Fix a lift $\phi\in W_K$ of $\Frob_K$. The eigenvalues of $\rho(\phi)$, say $\lambda_1,...,\lambda_n$, are contained in a finite Galois extension $E$ of $\Q$ as they are Weil numbers. We will show that there exists $d\ge 1$ such that for every $m\ge 0$,
  $$\rho(\phi^m\tau)^d=\rho(\phi^m)^d.$$
  Then for every $\tau\in W_K$, the eigenvalues of $\rho(\tau)$ are contained in the set of $\alpha\in \Omega$ such that $\alpha^d\in \{\lambda_1^m,...,\lambda_n^m\}$ for some $m\ge1$. The set of such $\alpha$ clearly generates a finite extension of $E$, in which $T(W_K)$ must be contained.

  Let us show the existence of $d$ as above. For $A,B\in GL_\Omega(V)$ we write $A^B$ for $BAB^{-1}$. The homomorphism $\tau\mapsto \phi\tau\phi^{-1}$ induces a homomorphism $\theta:\phi^{\Z}\ra \Aut(I_K/I_K\cap\ker(\rho))$. Put $i:=|I_K/I_K\cap\ker(\rho)|<\infty$ and $j:=|\Aut(I_K/I_K\cap\ker(\rho))|$. Then $\rho(\tau)^{i}=1$ and $\rho(\tau)^{\rho(\phi^{j})}=\rho(\tau)$ for all $\tau\in I_K$. Then (using  $\rho(\tau)^{\rho(\phi^{j})}=\rho(\tau)$ and  $\rho(\tau)^{i}=1$ in the second and third equalities, respectively)
\begin{equation}\begin{aligned}
\rho(\phi^m\tau)^{ij}\rho(\phi^m)^{-ij}&=\rho(\tau)^{\rho(\phi^m)}\rho(\tau)^{\rho(\phi^{2m})}\cdots\rho(\tau)^{\rho(\phi^{ijm})}\\
&= \left(\rho(\tau)^{\rho(\phi^m)}\rho(\tau)^{\rho(\phi^{2m})}\cdots\rho(\tau)^{\rho(\phi^{jm})}\right)^{i} = 1.
\end{aligned}\end{equation}
Hence we get the desired $d$ by putting $d:=ij$.

\end{proof}


  Later we would like to utilize some results of Arthur, in which local $L$-parameters are used in place of Weil-Deligne representations. We recall the standard way to go between the two. Recall that a local $L$-parameter for $GL_n(K)$ is a continuous homomorphism
  $$\varphi:W_K\times SL_2(\C)\ra GL(V)$$
  for an $n$-dimensional $\C$-vector space $V$ such that $\varphi|_{W_K}$ is semisimple and $\varphi|_{SL_2(\C)}$ is an algebraic representation. For such a $\varphi$ one associates a WD representation $WD(\varphi):=(V,\rho,N)$ such that
  $$\rho(\tau)=\varphi\left(\tau,\left(
                                    \begin{array}{cc}
                                      |\tau|^{1/2}_{W_K} & 0 \\
                                      0 & |\tau|^{-1/2}_{W_K} \\
                                    \end{array}
                                  \right)\right),\quad N=\varphi\left(1,\left(
                                    \begin{array}{cc}
                                      0 & 1 \\
                                      0 & 0 \\
                                    \end{array}
                                  \right)\right).$$
  The association $\varphi\mapsto WD(\varphi)$ defines a bijection between the set of equivalence classes of $L$-parameters for $GL_n(K)$ and the set of isomorphism classes of $n$-dimensional Frobenius semisimple WD representations. (In fact it is a categorical equivalence.) In fact the $L$-parameter $\varphi$ can be defined over any $\Omega$ in place of $\C$ and the various definitions for WD representations at the beginning of \S\ref{sub:pure-WD} carry over to $\varphi$. For instance $\varphi$ gives rise to a pure WD representation if and only if $\varphi|_{W_F}$ is strictly pure of integral weight in the sense defined earlier.

\subsection{Twists of the Local Langlands correspondence}\label{sub:LLC-field-of-rat}

  Let $\rec_K$ denote the local Langlands bijection for $GL_n(K)$ as in \cite{HT01} (cf. \cite{Hen00}) so that for each irreducible smooth representation $\pi$ of $GL_n(K)$, $\rec_K(\pi)$ denotes the associated $n$-dimensional Frobenius semisimple Weil-Deligne representation of $W_K$. Here both representations are considered on $\C$-vector spaces. We introduce a different normalization
  $$\mL_K(\pi):=\rec_K(\pi)\otimes|\cdot|_{W_K}^{-(n-1)/2}.$$
  It was shown in \cite[Lem VII.1.6.2]{HT01} that
  \beq\label{e:mL_K-twist}\mL_K(\pi^{\sigma})=\mL_K(\pi)^\sigma,\quad\forall \sigma\in \Aut(\C/\Q).\eeq
  (To be precise \cite{HT01} shows \eqref{e:mL_K-twist} up to semisimplification, i.e. disregarding $N$, but this easily implies \eqref{e:mL_K-twist} without semisimplification.)
   Hence $\rec_K(\pi^\sigma)=\rec_K(\pi)^\sigma$ for all $\sigma\in \Aut(\C/\Q(q^{1/2}))$.

\begin{lem}
  Let $\pi$ be an irreducible smooth representation of $GL_n(K)$ (on a $\C$-vector space). If $\mL_K(\pi)$ is pure of weight $w\in \Z$ then $\Q(\pi)$ is finite over $\Q$.
\end{lem}

\begin{proof}
  Immediate from Lemma \ref{l:pure-WD-rat} and \eqref{e:mL_K-twist}.
\end{proof}

\subsection{Bound on field of rationality implies bound on ramification}\label{sub:bounded-ram}

  For $j\in \R_{\ge0}$ let $I_K^j$ denote the $j$-th ramification subgroup of $I_K$ with respect to the upper numbering. Similarly for any Galois extension $M$ of $L$ (which are extensions of $K$), we write $\Gal(M/L)^j$ and $\Gal(M/L)_j$ for the upper and lower numbering ramification subgroups of $\Gal(M/L)$. Denote by $\dep$ and $\cond$ the depth and conductor, which are defined for WD representations of $W_K$ as well as irreducible smooth representations of $GL_n(K)$. The depth of a WD representation $(V,\rho,N)$ may not be as standard as the others so we recall it here: $\dep(V,\rho,N)$ is defined to be the infimum among the elements $j\in \R_{\ge0}$ such that $\rho(I_K^j)$ is trivial. The infimum is actually attained and the depth is a rational number.

The following lemma will play a key role in the proof of finiteness results of \S\ref{sub:fin-results}. In the proof all extensions of $E$ (which is a subfield of $\C$) are considered in $\C$.

\begin{lem}\label{l:bounded-depth}\footnote{After furnishing the proof of the lemma, we found that a proof had been given to an essentially same problem by \cite[\S4.(a)]{FM95}. We note two differences. First we work with the field of rationality rather than the field of definition. Second we obtain an explicit bound on $d_{n,A,K}$ which is not immediately available from \cite{FM95}. We also mention an analogous result for crystalline representations, cf. \cite[\S4]{CE04}.} Fix $n\in \Z_{\ge1}$ and $A\in \Z_{\ge1}$. There exists $d=d_{n,K,A}\in \R_{\ge 0}$ (depending on $n$, $A$ and $K$) such that for every $n$-dimensional F-ss WD representation $(V,\rho,N)$ whose field of rationality is an extension of $\Q$ of degree $\le A$, $$\dep(V,\rho,N)\le d.$$
\end{lem}

\begin{proof}
  Consider the representation $\rho|_{I_K}:I_K\ra GL(V)$ with finite image. Let $E$ be the field of rationality of $(V,\rho,N)$. By \eqref{e:mL_K-twist}, $\wedge^i \rho^\sigma\simeq \wedge^i \rho$ for all $\sigma\in \Aut(\C/E)$. We take the trace and see that the degree $n$ characteristic polynomial of $\rho(\tau)$ has coefficients in $E$. Hence each eigenvalue $\lambda$ of $\rho(\tau)$ for $\tau\in I_K$, which are roots of unity, must be contained in a finite extension of $E$ of degree $\le n$, so in a finite extension of $\Q$ of degree $\le nA$. Let $f$ be the least common multiple of the order of $\lambda$  as $\lambda$ runs over all eigenvalues of $\rho(\tau)$ for $\tau\in I_K$ (i.e. the least power $f\ge 1$ such that $\lambda^f=1$). In particular $\rho(\tau^f)$ is a semisimple element with all eigenvalues equal to $1$, i.e. the identity element.

  Put  $
  \mu_{\le nA}:=\bigcup\limits_{[E'':\Q]\le nA} \mu_\infty(E'')
  $
  where $\mu_\infty(E'')$ denotes the set of all roots of unity in $E''$. One sees from an elementary theory of cyclotomic fields that $\mu_{\le nA}$ is a finite set (its cardinality is the least common multiple of $m\in \Z_{\ge1}$ such that $\varphi(m)\le nA$, where $\varphi$ is the Euler totient function). We have $f\le |\mu_{\le nA}|$, an upper-bound which by construction depends only on $n$ and $A$.

   The finite group $H:=I_K/\ker \rho$ is equipped with an embedding $\ol{\rho}:H\hra GL(V)$ induced by $\rho$. Let $E'$ be the finite extension of $\Q$ obtained by adjoining all $f$-th roots of unity. As $H$ has exponent dividing $f$, Brauer's theorem (\cite[\S12.3, Th 24]{Ser77}) implies that $\ol{\rho}$ is defined over $E'$, i.e. there exists a representation $\ol{\rho}':H\hra GL(V_{E'})$ on an $E'$-vector space $V_{E'}$ such that $\ol{\rho}'\otimes_{E'} \C\simeq \ol{\rho}$ as $H$-representations.

    Now choose any prime $l$ relatively prime to $f$, and a place $w$ of $E'$ above $l$. Denote by $k_w$ the residue field of $E'$ at $w$. The $l$-adic representation $\ol{\rho}':H\hra GL_n(E'_w)$ has a model over $GL_n(\cO_{E'_w})$ in the sense that the model becomes isomorphic to $\ol{\rho}'$ after extending scalars to $E'_w$.\footnote{This is true even for continuous $l$-adic representations of any profinite group. The main point is that the $\cO_{E'_w}$-module generated by finitely many translations of an $\cO_{E'_w}$-lattice is still an $\cO_{E'_w}$-lattice.}  We denote the model by the same symbol $\ol{\rho}'$. The kernel of the map $\flat:GL_n(\cO_{E'_w})\ra GL_n(k_w)$ taking matrix entries modulo the maximal ideal of $\cO_{E'_w}$ is a pro-$l$ group, which must have trivial intersection with $H$. Hence $\flat\circ\ol{\rho}'$ is an injection $H\hra GL_n(k_w)$. The upshot is that
    \beq\label{e:bounded-ram-deg}|I_K/\ker \rho|=|H|\quad\mbox{divides}\quad |GL_n(k_w)|.\eeq
  The cardinality $|GL_n(k_w)|$ can be made to depend only on $nA$. Indeed $f$ and $E'$ depend only on $nA$ by construction. By choosing the minimal prime $l$ coprime to $f$, and $w$ above $l$ minimalizing $|k_w|$, we arrange that the cardinality $|GL_n(k_w)|$ depends only on $nA$, cf. \eqref{explicit-d-bound} below.

  Now it is enough to verify the existence of $d\in \Z_{\ge0}$ with the following property: $$\Gal(L/\hat{K}^{\ur})^d=\{1\}$$
  for all finite Galois extension $L$ of $\hat{K}^{\ur}$ such that $[L:\hat{K}^{\ur}]$ divides $|GL_n(k_w)|$, where $\hat{K}^{\ur}$ denotes the completion of the maximal unramified extension of $K$.
  This is a standard exercise. Indeed, writing $e_L$ (resp. $e_K$) for the absolute ramification index of $L$ (resp. $K$) so that elements of $L^\times$ take valuations exactly on $\frac{1}{e_L}\Z$ (if $p$ is normalized to have valuation 1), we know that $\Gal(L/\hat{K}^{\ur})_{d'}=\{1\}$ for all $d'>e_L/(p-1)$ by \cite[IV.2, Exercise 3.c]{Ser79}. The same is certainly true for $d'$-th upper numbering group since the latter is identified with $d''$-th lower numbering group for some $d''\ge d'$. Since $e_L\le |GL_n(k_w)|e_K$, we conclude that the choice of $d=|GL_n(k_w)|e_K/(p-1)$ satisfies the desired property.

\end{proof}

\begin{cor}\label{c:bounded-conductor} Fix $n, A\in \Z_{\ge1}$. There exists $d\in \Z_{\ge 0}$ (depending on $n$ and $K$) such that for every C-algebraic $\pi\in \Irr(GL_n(K))$ satisfying
  $[\Q(\pi):\Q]\le A$, we have that $$\dep(\pi)\le d, \quad \cond(\pi)\le d n.$$
\end{cor}

\begin{proof}
  Keeping \eqref{e:mL_K-twist} in mind, we apply Lemma \ref{l:bounded-depth} to find $d\in \Z_{\ge 0}$ such that for every $\pi$ as above, the WD representation $\mL_{K}(\pi):=(V,\rho,N)$ has depth at most $d_{n,A}$. Since $\dep(\pi)=\dep(V,\rho,N)$ by \cite[Th 2.3.6.4]{Yu09}, we see that $\dep(\pi)\le d$. The assertion on conductor holds true with $f:=d n$ thanks to Lemma \ref{l:depth-conductor} below.

\end{proof}

  The following lemma may be well known but we present a proof here.

\begin{lem}\label{l:depth-conductor}
  For every $\pi\in \Irr(GL_n(K))$, $\cond(\pi)\le n\cdot ( \dep(\pi) + 1 )$.
\end{lem}

\begin{proof}
   Let $(V,\rho,N):=\mL_K(\pi)$.
Since $\rec_K$ and thus $\mL_K$ preserve conductor, we obtain from
 the formula for the Artin conductor of $(V,\rho,N)$ that
\[
\cond(\pi)= \Mcodim \left( V^{I_K} \right)^{N=0} + \int_0^\infty \Mcodim V^{I_K^j} dj.
\]
Each codimension is certainly less than $n$ and by definition the integral is supported on the interval $0\le j \le \dep(V,\rho,N)$. The inequality again follows.

\end{proof}

It is worth emphasizing that the constructive nature of the proof of Lemma \ref{l:bounded-depth} makes it possible to find an explicit bound in that lemma (and also in Corollary \ref{c:bounded-conductor}). The first step of the proof was to adjoin all roots of unity of degree $\le n[E:\Q]$. This yields the explicit bound $f\le |\mu_{n[E:\Q]}|\le (2n[E:\Q])!$. To obtain an effective bound of better quality we establish the following result.
\begin{lem}\label{r:explicit-bound}
	For each $n\ge 1$ there is a constant $c_n>0$ such that the following holds.
For any number field $E$ let $F$ be the finite extension of $E$ generated by all the roots of unity that are of degree $\le n$ over $E$. Then
	\begin{equation*}
	  [F:E] \le c_n [E:\Q]
	\end{equation*}
  \end{lem}

\begin{proof}
We have $F=E(\zeta_N)$ for some $N\ge 1$. Write $N=\prod\limits_{p^r || N} p^{r}$. It is not difficult to show that $\phi(p^{r})$ divides $n![E:\Q]$ but we shall derive a more precise estimate below.

Since the extensions $E(\zeta_{p^{r}})$ are linearly disjoint over $E$,
\begin{equation*}
[F:E] = \prod_{p^r ||N} [E(\zeta_{p^{r}}) : E].
\end{equation*}
Let $E(p):=E\cap \Q(\zeta_{p^r})$. Then similarly $[E:\Q] \ge \prod_{p|N} [E(p):\Q]$.

Since $\Q(\zeta_{p^{r}})$ is Galois, it is linearly disjoint from $E$ over $E(p)$. Thus we have $[\Q(\zeta_{p^{r}}):E(p)] = [E(\zeta_{p^{r}}):E]$ and therefore
\[
\phi(p^{r}) = [E(\zeta_{p^{r}}):E] [E(p):\Q].
\]
Since $[E(\zeta_{p^{r}}):E]\le n$ we can deduce the inequality
\[
\frac{[E(\zeta_{p^{r}}):E]}{[E(p):\Q]}
 \le
\min\left(\phi(p^{r}),n,\frac{n^2}{\phi(p^r)} \right)
\]

Taking a product we deduce the estimate
\[
\frac{[F:E]}{[E:\Q]}
\le
\prod_{p^r||N} \min\left(\phi(p^{r}),n,\frac{n^2}{\phi(p^{r})} \right) \le c_n.
\]
In the last inequality we could extend the product to all prime numbers and the choice $c_n=n^n$ is admissible.
This concludes the proof.
\end{proof}

Using the Lemma~\ref{r:explicit-bound} we have the bound $f\le c_nA$ in the proof of Lemma \ref{l:bounded-depth}. Since there is a prime between $f$ and $2f$ (if $f\ge 2$) by Chebyshev's theorem, we can choose the prime $l<2f$. On the other hand $[k_w:\F_l]\le [\Q(\mu_f):\Q]\le f$, hence $|GL_n(k_w)|\le n^2 l^f\le n^2(2f)^f$. So a value
\begin{equation*}
  d=  \frac{e_K}{p-1} A^{O_n(A)}
\end{equation*}
 is admissible for the conclusion of Lemma \ref{l:bounded-depth} to hold. Actually we shall establish an improved bound using a more efficient argument:

\begin{lem} (i) Let $H$ be a finite subgroup of $\GL_n$ whose traces generate a field $E$ with $[E:\Q]\le A$. The order of a $p$-Sylow of $H$ is at most $\le c_{n,p} A^n$ for some constant $c_{n,p}$ depending only on $n$ and $p$.

  (ii) In the statement of Lemma~\ref{l:bounded-depth} the constant
  \begin{equation}\label{explicit-d-bound}
  d_{n,K,A} \le c_{n,K} A^n
\end{equation}
  is admissible, where $c_{n,K}$ depends only on $n$ and $K$.
\end{lem}
\begin{proof}
  (i) Let $H_p$ be a $p$-Sylow subgroup of $H$. By a result of Roquette, the Schur indices of $H_p$ are $1$ if $p\neq 2$ and $1$ or $2$ if $p=2$ (see~\cite{Yamada:schur} for a direct proof using the Hasse invariant). Since the traces of elements of $H_p\subset H$ are in $E$, this shows that the representation $H_p \subset GL(V)$ can be realized over $E$ if $p\neq 2$ and over $E(\sqrt{-1})$ if $p=2$.

There remains the problem of estimating the order of a finite $p$-group $H_p$ inside $\GL_n(E)$ where $E$ is a number field. Minkowski obtained optimal bounds when $E=\Q$. The method can be extended to a general number field $E$ and Schur gave a different proof using character theory. Serre~\cite{Serre:finite-subgps} treated the case of an arbitrary $E$.
If $p\neq 2$, then:
\[
\frac{\log |H_p|}{\log p} \le m \left \lfloor \frac{n}{t} \right \rfloor
+
\left \lfloor \frac{n}{pt} \right \rfloor
+
\left \lfloor \frac{n}{p^2t} \right \rfloor + \cdots
\]
where $t=[E(\zeta_p):E]$ and $m\in \Z_{\ge 1}$ is the largest integer such that $\zeta_{p^m}\in E(\zeta_p)$. There are some subtle modifications if $p=2$ to get a sharp bound, but this is not relevant for us since we are interested in the behavior for $A$ large.

Certainly $\phi(p^m)\le [E(\zeta_p):\Q] \le t A$. Therefore there is a constant $c_n$ depending on $n$ such that $
|H_p| \le c_n (tA)^{\frac{n}{t}} $.
In particular $|H_p| \le c_{n,p} A^n$ for some constant $c_{n,p}$ depending only on $n$ and $p$.

(ii) This relies on the basic structure of the inertia group $I_K$. For any Galois extension $M$ of $L$ and any integer $j\ge 1$, the quotients $\Gal(M/L)_{j} / \Gal(M/L)_{j+1}$ can be identified with an additive subgroup of the residue field of $M$. Since these are abelian group of $p$-power order we have that $\Gal(M/L)_1$ is a $p$-group (\cite[IV.2, Corollary 3]{Ser79}). We are in position to apply the assertion (i). Then similarly to Lemma~\ref{l:bounded-depth} we can conclude the proof of the estimate~\eqref{explicit-d-bound}.
\end{proof}

\section{Automorphic representations of classical groups}\label{s:aut-classical}

  In this section we recall the endoscopy and the associated Galois representations for automorphic representations of symplectic, orthogonal and unitary groups. A key input is the integrality proposition in \S\ref{sub:Gal-rep} coming from an arithmetic geometry study of Shimura varieties.

\subsection{Galois representations associated to automorphic representations}\label{sub:Gal-rep}

  Let $n\in \Z_{\ge1}$. Let $F^+$ be a totally real field and consider the following two cases:
\benu
\item[(CM)] $F$ is a CM quadratic extension of $F^+$ with complex conjugation $c$ (so that $F^+=F^{c=1}$) or
\item[(TR)] $F=F^+$.
\eenu
  Let $\Pi$ be a regular C-algebraic cuspidal automorphic representation of $GL_n(\A_F)$ such that $\Pi_\infty$ has the same infinitesimal character as an irreducible algebraic representation $\Xi$ of $\Res_{F/\Q}GL_n$. The highest weight for $\Xi$ may be written as $a(\Xi)=(a_{\sigma,i})_{\sigma\in\Hom_{\Q}(F,\C),\,1\le i\le n}$ with $a_{\sigma,i}\in \Z$, viewed as a character of the standard diagonal torus of $\Res_{F/\Q}GL_n$. We further assume in each of the above cases that
\benu
\item[(CM)] $\Pi^\vee\simeq \Pi\circ c$.
\item[(TR)] $\Pi^\vee\simeq \Pi\otimes (\det\circ \chi)$ for $\chi:F^{\times}\bs \A^\times_F\ra \C^\times$ such that  $\chi_v(-1)$ is the same for every $v\in S_\infty$.
\eenu

In case (CM), fix a subset $\Phi^+\subset \Hom_{\Q}(F,\C)$ such that $\Phi^+\coprod \Phi^+\circ c = \Hom_{\Q}(F,\C)$ (called a CM-type). Recall that various notation and notion about WD representations were introduced in \S\ref{sub:pure-WD} and the twist of the local Langlands correspondence $\mL_{F_v}(\Pi_v)$ in \S\ref{sub:LLC-field-of-rat}.

\begin{prop}\label{p:Galois-reps} There exists a family
  of $l$-adic representations (for varying $l$) $$\{R_{l,\iota_l}(\Pi):\Gal(\ol{F}/F)\ra GL_n(\ol{\Q}_l)\}_{l,\iota_l},\quad\mbox{$l$ is a prime and } \iota_l:\ol{\Q}_l\simeq \C$$ such that for every finite place $v$ of $F$,$$\mL_{F_v}(\Pi_v)=
      \iota_l WD(R_{l,\iota_l}(\Pi)|_{\Gal(\ol{F}_v/F_v)})^{\Fss},$$
  $\mL_{F_v}(\Pi_v)$ is a pure Weil-Deligne representation of weight $n-1$, and $\Pi_v$ are essentially tempered. Moreover there exists $s(\Pi_\infty)\in \Z$ (depending only on $\Pi_\infty$ and $F$) such that for every finite $v$, $\mL_{F_v}(\Pi_v){|\cdot|^{-s(\Pi_\infty)/2}}$ is integral.
\end{prop}

\begin{rem}\label{r:Galois-reps} When $v\notin S$ and $v\nmid l$, the proposition tells us that $\mL_{F_v}(\Pi_v)$ is unramified, \emph{strictly} pure of weight $n-1$ and the eigenvalues of $\Frob_v$ on $\mL_{F_v}(\Pi_v)$ are $q_v$-Weil numbers of weight $n-1$ which become algebraic integers after multiplying $q_v^{s(\Pi_\infty)}$.
\end{rem}

\begin{rem}
  A more elementary method seems available in some cases by exploiting the integral structure on the space of algebraic modular forms (cf. \cite{Gro99}) on classical groups which are compact at infinity, without the need of arithmetic geometry and Galois representations. We have not adopted it here as we do not know how to cover all cases with that approach.
\end{rem}

\begin{proof} Except for the last assertion the proposition is a result of combined effort: See \cite[Th A]{BLGGTb} as well as \cite[Th 1.2]{Shi11}, \cite[Th 1.4]{CH13}, \cite[Th 1.1]{Car12} and \cite[Th 1.1]{Caraiani2}.

  The last assertion on integrality remains to be justified. The main idea is that the Galois representation $R_l(\Pi)=R_{l,\iota_l}(\Pi)$ of the proposition is essentially realized in the cohomology of a certain $(n-1)$-dimensional compact Shimura variety $\Sh$ to which some general results in arithmetic geometry (such as \cite[Cor 0.6.(1)]{Sai03}) apply.
  A precise argument requires us to import lots of notation and various pieces of results. For any difficulty caused by this to read our proof we apologize. We will recall at least some of the important notation and facts as we go along.

  In case (TR) it is possible to choose a CM quadratic extension $L$ of $F$ and an algebraic Hecke character $\phi:L^\times\bs \A^\times_L \ra \C^\times$ such that $\phi\phi^c=\chi\chi^c$ and $BC_{L/F}(\Pi)$ is cuspidal. (One can make a choice to ensure cuspidality by arguing as in \S1 of \cite{Clo13}.) Then
  $\Pi':=BC_{L/F}(\Pi)\otimes \phi$ is conjugate self-dual, regular and C-algebraic. Thus the proof of the integrality assertion is reduced to case (CM) for $L$ and $\Pi'$ via Lemma \ref{l:BC}.

  From now on we put ourselves in case (CM). Starting with the case where $n$ is odd, we will derive the integrality result as a consequence of \cite{Shi11} and \cite{TY07}.
   It is desirable to reconcile notation with \cite{Shi11} at the outset to avoid confusion. Only in this proof $G$ denotes the unitary similitude group as in \cite{Shi11}. Our $\Pi$ corresponds to $\Pi^1$ in that paper. The notation $\Pi$ there, designating a representation of $G(\A_E)$, will be written as $\Pi'$ here. (So the finite part of the representation $\Pi'$ of $G(\A_E)\simeq GL_1(\A_E)\times GL_n(\A_F)$ descends to the finite part of an automorphic representation of $G(\A)$.)

   Since integrality may be checked after a series of finite cyclic base changes (Lemma \ref{l:BC}) we may and will assume that conditions (i)-(v) of \S6.1 and the five assumptions at the start of \S7.1 in \cite{Shi11} are satisfied. In particular $F$ contains an imaginary quadratic field $E$. Fix an embedding $\tau_E:E\hra \C$ and choose $\Phi^+$ to be the set of $F\hra \C$ extending $\tau_E$. Recall that $(a_{\sigma,i})_{\sigma\in \Hom(F,\C),1\le i\le n}$ is associated with $\Xi$. Choose a Hecke character $\psi:E^\times \bs \A_E^\times \ra \C^\times$ as in \cite[Lem 7.2]{Shi11}, cf. \cite[Lem VI.2.10]{HT01}. The proof in the latter reference shows that $\psi_\infty(z)=\tau_E(z)^{a_0}$, where $a_0:=-\sum_{\sigma\in \Phi^+,1\le i\le n} a_{\sigma,i}$. Let $\xi$ be the irreducible algebraic representation of $G\times_\Q \C\simeq  GL_1(\C)\times \prod_{\sigma\in \Phi^+} GL_n(F_\sigma)$ of highest weight $(a_0,(a_{\sigma,i})_{\sigma\in \Phi^+,1\le i\le n})$. Put
   $$d_\sigma:=\max_{1\le i\le k} (\max(0,a_{\sigma,i})),\quad s_{\sigma}:=\sum_{1\le i\le k} |a_{\sigma,i}-d_{\sigma}|,\quad t_\xi:=a_0+\sum_{\sigma\in \Phi^+} k d_\sigma,\quad m_\xi=\sum_{\sigma^\in \Phi^+} (s_\sigma+ kd_\sigma),$$
   and $s(\Pi_\infty):=2t_\xi-\min(0,a_0)$. Note that $s(\Pi_\infty)$ depends only on the data defining $\Pi_\infty$. One can check that $2t_\xi-m_\xi=2a_0+\sum_{\sigma,i} a_{\sigma,i}$, cf. \cite[p.98]{HT01}, \cite[p.476]{TY07}.

  Consider the \'{e}tale cohomology
   $$H_{\et}^{n-1}(\Sh,\cL_\xi):=\dirlim{U\subset G(\A^\infty)} H_{\et}^{n-1}(\Sh_U\times_F \ol{F},\cL_\xi)$$
   as $U$ runs over sufficiently small open compact subgroups of $G(\A^{\infty})$. The limit
    is a $\ol{\Q}_l[G(\A^\infty)\times \Gal(\ol{F}/F)]$-module, which is admissible (resp. continuous) with respect to the $G(\A^\infty)$- (resp. $\Gal(\ol{F}/F)$-)action. For small enough $U$, $ H_{\et}^{n-1}(\Sh_U\times_F \ol{F},\cL_\xi)$ is finite dimensional and $$H_{\et}^{n-1}(\Sh_U\times_F \ol{F},\cL_\xi)=a_\xi H_{\et}^{n-1+m_\xi}(\mathcal{A}^{m_\xi}\times_F \ol{F},\ol{\Q}_l(t_\xi))$$
   where $a_\xi$ is an idempotent of \cite[pp.476-477]{TY07}, which gives rise to an element of the Chow group $\mathrm{CH}^{n-1+m_\xi}(\Sh_U \times_F \Sh_U)_{\Q}$. (The subscript $\Q$ indicates that the coefficient ring is taken to be $\Q$.) On the other hand, $H_{\et}^{n-1}(\Sh_U\times_F \ol{F},\cL_\xi)$ is the direct sum \beq\label{e:core-cohomology}
   \left(\bigoplus_{BC(\pi^{S,\infty})\simeq (\Pi')^{S,\infty}}R^{n-1}_l(\pi^\infty)\otimes (\pi^\infty)^U\right)\oplus
   \left(\bigoplus_{BC(\pi^{S,\infty})\ncong (\Pi')^{S,\infty}}R^{n-1}_l(\pi^\infty)\otimes (\pi^\infty)^U\right) \eeq
   where the first (resp. second) sum runs over $\pi^\infty$, the finite part of discrete automorphic representations of $G(\A)$, such that $BC(\pi^{S,\infty})\simeq (\Pi')^{S,\infty}$ holds (resp. does not hold). According to \cite[Cor 6.8]{Shi11},
   there is a positive integer $C_G$ and a $\Gal(\ol{F}/F)$-representation $\tilde{R}'_l(\Pi')$ such that $C_G\tilde{R}'_l(\Pi')=\oplus_{\pi^\infty} R^{n-1}_l(\pi^\infty)$ where the sum is taken over the same set as in the first sum of \eqref{e:core-cohomology}.
   The corollary 6.10 of \cite{Shi11} tells us that $R_l(\Pi)$ in the proposition is given by
  \beq\label{e:R(Pi)}R_l(\Pi):=
     \tilde{R}'_l(\Pi')\otimes \rec_l(\psi).\eeq

   The decomposition \eqref{e:core-cohomology} allows us to find an idempotent $b_\xi$ in the Chow group $\mathrm{CH}^{n-1}(\Sh_U \times_F \Sh_U)_{\Q}$ such that $C_G\cdot \tilde{R}'_l(\Pi')\simeq b_\xi H_{\et}^{n-1}(\Sh_U\times_F \ol{F},\cL_\xi)$, equivariant for the $\Gal(\ol{F}/F)$-action. (First, find an idempotent separating each $\pi^\infty$-part as in the proof of the lemma 2.3 in \cite{TY07}. Then one uses a Hecke algebra element which acts with trace 1 on each $\pi^\infty$ such that $BC(\pi^{S,\infty})\simeq (\Pi')^{S,\infty}$.) Write $c_\xi$ for the pullback of $b_\xi$ along the projection from $\mathcal{A}^{m_\xi}_U\times \mathcal{A}_U^{m_\xi}$ to $\Sh_U\times \Sh_U$.
    Since $H_{\et}^j(\Sh_U \times_F \ol{F},\cL_\xi)$ for $j\neq n-1$ is linearly independent from $(\pi^\infty)^K$ for any $\pi^\infty$ as in \eqref{e:core-cohomology} by \cite[Cor 6.5.(i)]{Shi11}, we may construct a correspondence $\Gamma$ coming from  $\mathrm{CH}(\mathcal{A}^{m_\xi}\times \mathcal{A}^{m_\xi})_{\Q}$ such that $\Gamma$ acts on $H_{\et}^{j}(\mathcal{A}^{m_\xi}\times_F \ol{F},\ol{\Q}_l(t_\xi))$
   as $c_\xi\circ a_\xi$ if $j=n-1+m_\xi$ and 0 if $j\neq n-1+m_\xi$. By construction we have an isomorphism of $\Gal(\ol{F}/F)$-representations
   $$C_G\cdot \tilde{R}'_l(\Pi')(-t_\xi)\simeq \Gamma\cdot H_{\et}^{n-1+m_\xi}(\mathcal{A}^{m_\xi}\times_F \ol{F},\ol{\Q}_l).$$

    Finally we apply the argument of \cite[Prop 3.5]{Sai03},\footnote{The difference is that Saito considers the whole $H_{\et}^{n-1+m_\xi}$ of $\mathcal{A}^{m_\xi}_U\times_F \ol{F}$ whereas we argue only on its subrepresentation. Still it is easy to adapt his argument to our situation.} noting that his condition (3) is satisfied for $\Gamma$ as above. The conclusion is that any lift $\phi_v$ of $\Frob_{v}$ on $H_{\et}^{n-1+m_\xi}(\mathcal{A}_U^{m_\xi}\times_F \ol{F},\ol{\Q}_l)$ has algebraic \emph{integers} as eigenvalues. Hence the WD-representations associated with $H_{\et}^{n-1}(\Sh_U\times_{F_v}\ol{F}_v,\cL_{\xi})(-t_\xi)$ as well as $\tilde{R}'_l(\Pi')(-t_\xi)|_{\Gal(\ol{F}_v/F_v)}$ are integral. By \eqref{e:R(Pi)}, $WD(R_l(\Pi)(-t_\xi)\otimes \rec(\psi)^{-1}|_{\Gal(\ol{F}_v/F_v)})$ is integral. We have
    $$ \mL_{F_v}(\Pi_v)|\cdot|_v^{-s(\Pi_\infty)/2}=
      \iota_l WD(R_{l,\iota_l}(\Pi)|_{\Gal(\ol{F}_v/F_v)})^{\Fss}|\cdot|_v^{-s(\Pi_\infty)/2}$$
      $$= \iota_l WD(R_l(\Pi)(-t_\xi)\otimes \rec(\psi)^{-1}|_{\Gal(\ol{F}_v/F_v)}) \otimes\iota_l
      WD(\rec(\psi_v)^{-1}|\cdot|_v^{\min(0,a_0)}).$$
      Since $WD(\rec(\psi_v)^{-1}|\cdot|_v^{-\max(0,a)})$ is integral by Lemma \ref{l:int-char} below we are done with verifying the integrality of $\mL_{F_v}(\Pi_v)|\cdot|_v^{-s(\Pi_\infty)/2}$ when $n$ is odd.



  It remains to justify integrality when $n$ is even. The argument is essentially the same as above, so it would suffice to point out what modifications are needed. In this case we put ourselves in the setting of \cite{Car12}, which shows that $R_l(\Pi)^{\otimes 2}$ is realized up to an explicit twist in $H_{\et}^{2n-2}(X,\cL_{\xi})$ for the $(2n-2)$-dimensional Shimura varieties therein. By arguing as above, we obtain $s(\Pi_\infty)\in \Z$ such that $WD(R_l(\Pi)^{\otimes 2}|_{\Gal(\ol{F}_v/F_v)})|\cdot|_v^{-s(\Pi_\infty)}$ is integral at every finite place $v$. But the latter implies that $WD(R_l(\Pi)|_{\Gal(\ol{F}_v/F_v)})|\cdot|_v^{-s(\Pi_\infty)/2}$ is integral as well.

\end{proof}


\begin{lem}\label{l:int-char}
  Let $F$ be any number field and $\psi:F^\times\bs \A_F^\times \ra \C^\times$ a Hecke character. At each infinite place $v$, suppose that there are some $m_v\in \Z_{\le 0}$ and some continuous character $F_v^\times \ra \C^\times$ such that $\psi_v(z)=\tau_v(z)^{m_v}$ for all $z\in (F_v^\times)^0$. Then for every finite place $v$ and every uniformizer $\varpi_v$ of $F_v$, $\psi_v(\varpi_v)$ is an algebraic integer.
\end{lem}

\begin{proof}
  There exists an open compact subgroup $U$ of $\hat{\cO}_F=\prod_{w\nmid \infty} \cO_w^\times$ such that $\psi|_U\equiv 1$. Fix a finite place $v$ and a uniformizer $\varpi_v$. By strong approximation there exists $a\in F^\times$ such that $a\in \varpi_v U$ in $(\A_F^\infty)^\times$. Since $a\in\hat{\cO}_F$, $a$ is an algebraic integer. Now
  $$\psi_v(\varpi_v)=\psi^\infty(a)=\psi_\infty(a)^{-1}=\pm\prod_{w|\infty} \tau_w(a)^{-m_v}$$
  where the sign comes from the character $F_\infty^\times/(F_\infty^\times)^0$ with values in $\{\pm 1\}$. The lemma follows.
\end{proof}

\subsection{Quasi-split classical groups}\label{sub:transfer-app}

  Later on several results will be established concerning automorphic representations of quasi-split\footnote{The analogous results for non quasi-split groups are sketched in the last chapter of \cite{Arthur} but might require a few more years for a complete proof.} classical groups. To this end we would like to introduce basic data for symplectic, orthogonal, and unitary groups. Let $F^+$ be a totally real field.  We take $F$ to be $F^+$ in the symplectic and orthogonal cases and a CM quadratic extension of $F^+$ in the unitary case. Both $G$ and $\bG$ below will be connected reductive \emph{quasi-split} groups over $F^+$. Let us suppress the choice of the symplectic, symmetric, or hermitian pairings.

  Define $c\in \Gal(F/F^+)$ to be the identity if $F=F^+$ and the nontrivial element if $F\neq F^+$.
  For $n\ge 1$ let $J_n$ denote the matrix with $(-1)^i$ in $(i,n+1-i)$-th entry for $1\le i\le n$ and zeros off the anti-diagonal. Write $\theta_n$ (resp. $\hat{\theta}_n$) for the automorphism $g\mapsto J_n {}^t g^{-c} J_n^{-1}$ of $\Res_{F/F^+}GL_n$ over $F^+$ (resp. $g\mapsto J_n {}^t g^{-1} J_n^{-1}$ of $GL_n(\C)$). The standard embedding of a symplectic, special orthogonal, or general linear group will be denoted $\std$.

  In all cases below $n\in \Z_{\ge1}$, $\bG=\Res_{F/F^+}GL_n$ (which is just $GL_n$ except the unitary case), $\theta=\theta_n\in \Aut_{F^+}(\bG)$, $\hat{\theta}=\hat{\theta}_n\in \Aut_{F^+}(\hat{\bG})$, $s\in \hat{G}$, and $\eta:{}^L G \hra {}^L \bG$ is an $L$-morphism. We will describe $G$, $s$, and $\eta$ case-by-case. Only for even orthogonal groups, we use the Weil group form of an $L$-group in order to accommodate a half-integral twist, which is needed for $\eta$ to be C-preserving.


\benu
\item symplectic groups: $n$ is odd, $G=Sp_{n-1}$, $s=1$,
$$\eta=(\std,\id):SO_{n}(\C)\times \Gamma_F \hra GL_{n}(\C)\times \Gamma_F.$$

\item orthogonal groups: $n$ is even, $s=1$,
\benu
\item type $B$: $G=SO_{n+1}$, $s=1$,
 $$\eta=(\std,\id):Sp_{n}(\C)\times \Gamma_F \hra GL_{n}(\C)\times \Gamma_F.$$
\item type $D$: let $\delta\in F^\times/(F^\times)^2$ be the discriminant of the underlying quadratic form and $F_\delta:=F(\delta^{1/2})$.
\bit
\item $G=SO_{n}$, $\delta=1$ so that $G$ is a split group, $s=1$, $\eta_0=(\std,\id):Sp_{n}(\C)\times \Gamma_F \hra GL_{n}(\C)\times \Gamma_F$
and define
$$\eta:SO_{n}(\C)\times W_F \hra GL_{n}(\C)\times W_F,\quad \eta:=\eta_0|\cdot|^{1/2}$$
where $|\cdot|$ is the modulus character on $W_F$.
\item $G=SO_{n}$, $\delta\neq 1$ so that $G$ is a non-split group,
 ${}^L G = SO_{n}(\C)\rtimes \Gamma_F$ (with $\Gamma_F$ acting through $\Gal(F_\delta/F)$ on $SO_{n}(\C)$ via order 2 outer automorphism); $s=\mathrm{diag}(-I_n,I_n)$,
 $\eta_0:SO_{n}(\C)\rtimes \Gamma_F \hra  GL_n(\C)\times \Gamma_F$ is an extension of the standard embedding $SO_{n}(\C)\hra GL_n(\C)$ defined on page 51 of \cite{Wal10} (the map $^L \xi$ in case $d^-=n$, $d^+=1$, and $\delta^-=\delta\neq 1$). Define
  $$\eta:SO_{n}(\C)\rtimes W_F \hra  GL_n(\C)\times W_F, \quad \eta:=\eta_0|\cdot|^{1/2}.$$

\eit
\eenu

\item[(ii)'] This is a subcase of (ii); in case (ii)(a) it is the same as above; in case (ii)(b) further assume that $\delta=1$ if $n/2$ is even and $\delta\neq 1$ if $n/2$ is odd.

\item[(iii)] unitary groups: $G=U_n$, $s=1$, ${}^L G = GL_{n}(\C)\rtimes \Gamma_{F^+}$ (with $\Gamma_{F^+}$ acting through $\Gal(F/F^+)=\{1,c\}$, the $c$-action being $\hat{\theta}_n$),$\theta=\theta_n$,
$$\eta:GL_{n}(\C)\times \Gamma_F \hra (GL_{n}(\C)\times GL_n(\C)\rtimes \Gamma_F,\quad
    g\times \gamma \mapsto (g,J_n {}^t g^{-1} J_n^{-1})\rtimes \gamma.$$

\eenu

Set $\epsilon:=0$ in case (i), (ii)(a) and (iii) and $\epsilon:=1$ in case (ii)(b). This auxiliary constant accounts for the modulus character in the definition of $\eta$.

  The reason for introducing (ii)' is the following: a classical group $G$ over $F^+$ admits discrete series at real places (equivalently admits compact maximal tori) exactly when $G$ belongs to (i), (ii)' or (iii).
  In each of (i), (ii) and (iii), $(G,s,\eta)$ is a twisted endoscopic datum for $\bG\rtimes \lg \theta\rg$ in the sense of \cite{KS99}, cf. \cite[\S1.2]{Arthur} for symplectic and orthogonal groups. Observe that in all cases
  $$\bG(\A_{F^+})=GL_n(\A_F).$$

\begin{lem}\label{l:property-of-eta} Put ourselves in (i), (ii) or (iii) as above. Let $v$ be an infinite place of $F^+$ as above.
  \bit
  \item If $\varphi_v:W_{F^+_v}\ra {}^L G$ is regular (i.e. the restriction $\varphi_v|_{W_{\ol{F}^+_v}}$ is not invariant under any nontrivial Weyl element, cf. Definition \ref{d:C-algebraic}) then $\eta\varphi_v$ is also regular.
  \item $\eta$ is C-preserving (Definition \ref{d:C-preserving}).
  \eit
\end{lem}

\begin{proof}
  Both assertions are checked by explicit computations with root data. We will verify the first assertion in case (i) and leave it to the reader in the other cases. We may choose maximal tori $\hat{T}$ and $\hat{\bT}$ of $G$ and $GL_n$ and $\Z$-bases $\{e_i\}$ and $\{f_j\}$ for the cocharacter groups ($1\le i\le \frac{n-1}{2}$, $1\le j\le n$) and such that $\eta$ restricts to $\hat{T}\hra \hat{\bT}$ inducing $e_i\ra f_i-f_{n+1-i}$ on the cocharacter groups. It suffices to show that every regular element of $X_*(\hat{T})\otimes_\Z \C^\times$ maps to a regular element of $X_*(\hat{\bT})\otimes_\Z \C^\times$. Let $\sum_{i} a_i e_i$ with $a_i\in \C^\times$ be regular. Since the Weyl group is generated by permutation of the indices $i$ and sign changes $e_i\ra -e_i$, the regularity means that $a_i$'s are distinct and that $a_i\neq 0$. Then the image $\sum_i a_i(f_i-f_{n+1-i})$ has the property that the coefficients of $f_j$'s are all distinct, i.e. has trivial stabilizer under the Weyl group action in $GL_n$. This shows that regularity is preserved under $\eta$ in case (i).

  It is easy to compute the half sum of positive roots to verify C-preservation in each case. We only deal with case (ii)(b) to explain the role of the extra half-power twist there. Choose $\hat{T}$, $\hat{\bT}$, $\{e_i\}$ and $\{f_j\}$ ($1\le i\le n/2$, $1\le j\le n$) similarly such that $\eta_0$ induces $e_i\ra f_i-f_{n+1-i}$ on $X_*(\hat{T})\ra X_*(\hat{\bT})$. The Borel subgroups can be chosen such that the half sum of positive roots is $\rho_G=(n/2-1)e_1+(n/2-2)e_2+\cdots + e_{n/2-1}$ for $G$ (resp. $\rho_{GL_n}=\frac{n-1}{2}f_1+\frac{n-3}{2}f_2+\cdots + \frac{1-n}{2}f_n$). So $\eta_0(\rho_G)=(n/2-1)(f_1-f_n)+\cdots + (f_{n/2-1}-f_{n/2+1})$, and $\eta(\rho_G)=\eta_0(\rho_G)+\frac{1}{2}(f_1+f_2+\cdots +f_n)$. Hence $\rho_{GL_n}-\eta(\rho_G)$ has integral coefficients in $f_j$'s, showing that $\eta$ is C-preserving (but note that $\eta_0$ is not).

\end{proof}

\begin{lem}
  Assertions (i), (ii) and (iii) of Lemma \ref{l:unram-functoriality} hold true in the even orthogonal case (ii)(b) (even though $\eta$ does not satisfy the hypothesis in that lemma).
\end{lem}

\begin{proof}
  The same argument in the proof of (i) in Lemma \ref{l:unram-functoriality} for $\eta_0:{}^L G\ra {}^L \bG$
\end{proof}

\subsection{Twisted endoscopic transfer for classical groups}\label{sub:tw-end}

  We would like to recall elements of local twisted endoscopy at a \emph{non-archimedean} place $v$ of $F^+$ as these will be important to us. (The corresponding theory at archimedean places is well known.)
  Kottwitz, Langlands and Shelstad (\cite{LS87}, \cite{KS99}) defined transfer factors $\Delta_v(\gamma_G,\gamma_{\bG})$ for all strongly regular semisimple elements $\gamma_G\in G(F^+_v)$ and $\gamma_{\bG}\in \bG(F^+_v)$ at every place $v$ of $F^+$. In fact we will use the Whittaker normalization of transfer factors, to be denoted $\Delta_v^{\Wh}$, which were defined in \cite[\S5.3]{KS99} in the quasi-split case.\footnote{$\Delta_v^{\Wh}$ depends on the extra choice of Whittaker data of \S5.3 of \cite{KS99}, which will be chosen globally. The reference to this choice will be suppressed as the transfer factors are only affected by sign, cf. page 65 of \textit{loc. cit.}, and do not affect the asserted rationality of transfer factors.}
  We say that $\phi_v\in C^\infty_c(G(F^+_v))$ is a $\Delta_v^{\Wh}$-transfer of $f_v\in C^\infty_c(\bG(F^+_v))$ if
  \beq\label{e:KLS-trans}STO^{\bG(F^+_v)}_{\delta}(f_v)=\sum_{\gamma\sim_{\mathrm{st}} \delta} \Delta_v^{\Wh}(\gamma,\delta) O^{G(F^+_v)}_\gamma (\phi_v).\eeq
  for every pair $(\gamma_G,\gamma_{\bG})$ of strongly regular semisimple elements.
 The proof of the fundamental lemma by Ng\^{o}, Waldspurger and others (see \cite{Ngo10}, \cite{Wal97}, \cite{Wal06}, \cite{Wal08}) ensures that a $\Delta_v^{\Wh}$-transfer of $f_v$ exists for every $f_v$ as above.

\begin{prop}\label{p:trans-rational}

  In cases (i), (ii) and (iii) of the previous subsection,
 \beq\label{e:Wh-trans}\Delta_v^{\Wh}(\gamma_G,\gamma_{\bG})\in\Q\eeq
  for all strongly regular semisimple elements $\gamma_G\in G(F^+_v)$ and $\gamma_{\bG}\in \bG(F^+_v)$.
  (In fact $\Delta_v^{\Wh}(\gamma_G,\gamma_{\bG})\in\{0,\pm 1\}$ with the exception of (ii)(b).)




\end{prop}

\begin{proof}

 We will only sketch the argument. Since $\Delta_v^{\Wh}$ differs from $\Delta_0$ of \cite[\S5.3]{KS99} by $\pm 1$ (\cite[p.65]{KS99}) it suffices to prove the claim for $\Delta_0$.
  The transfer factors for classical groups were computed in \cite{Wal10}. In the cases of interest, it is shown that the transfer factor $\Delta_I\Delta_{II}\Delta_{III}$ belongs to $\{0,\pm1\}$ for $(G,s,\eta)$ as in (i), (ii)(a), (iii) and for $(G,s,\eta_0)$ as in (ii)(b). (See the cases of twisted linear groups in \cite[\S1.10]{Wal10}, noting that $\chi$ is a character of order dividing 2 in the odd twisted linear case and that $\mu^-$ and $\mu^+$ may be chosen to be trivial in the case of base change for unitary groups.) Note that Waldspurger suppressed $\Delta_{IV}$ in his formulas but the transfer factor $\Delta_0$ is $\Delta_I\Delta_{II}\Delta_{III}\Delta_{IV}$. In case (i), (ii)(a) and (iii) we see $\Delta_{IV}=1$ following the definition of \cite[\S4.5]{KS99}, so the values of $\Delta_0$ range in $\{0,\pm1\}$. In case (ii)(b) $\Delta_{IV}$ is a nontrivial function involving a half-power of the modulus character (cf. \cite[(4.5.1)]{KS99}) so $\Delta_0$ for $(G,s,\eta_0)$ takes values in $\Q(q_v^{1/2})$ but replacing $\eta_0$ with $\eta$ twists the transfer factor by an extra half-power of the modulus character. As a result $\Delta_0$ with respect to $(G,s,\eta)$ has values in $\Q$.
\end{proof}

  From here until \S\ref{sub:unitary-groups} we will restrict our attention to cases (i) and (ii) above. Case (iii) is excluded until there only because our understanding of representations of unitary groups is still limited. Nevertheless we will treat all three kinds of classical groups on an equal footing at the expense of burdening notation (e.g. we distinguish between $F$ and $F^+$, which is unnecessary in (i) and (ii)) so that the results in this article apply to unitary groups as soon as the analogue of \cite{Arthur} for unitary groups is worked out. In fact our results already produce some partial results in the case of unitary groups by appealing to the progress on twisted endoscopy (base change) for unitary groups in \cite{KK05}, \cite{Moe07}, and \cite{Lab} among others.

  To use results for automorphic representations on quasi-split classical groups as in \cite{Arthur} (symplectic and orthogonal) and \cite{Mok} (unitary), we assume\label{footnote10}\footnote{One can be optimistic that the hypothesis will become unnecessary before long. At the time of revision, Waldspurger has released a series of five preprints (more to come) on the stabilization of the general twisted trace formula. For an extra careful reader, we remark that both \cite{Arthur} and \cite{Mok} depend on the papers [A25] and [A26] of \cite{Arthur}, which have not appeared up to now, and that the proof of the weighted fundamental lemma has not been completely written up, cf. footnote in Appendix A of \cite{BMM}. Proposition 8.2.5 of \cite{Mok} asserts that Ban's result, cited as [Ban] there and proved only for split groups, extends to quasi-split unitary groups but this appears to be a nontrivial point to be justified. Arthur, as well as Mok, refers to work in progress by Mezo and Shelstad on twisted endoscopy for real groups and by Waldspurger on the local twisted trace formula. This seems fine: The former is basically addressed in the preprints cited as [Me] and [S8] in \cite{Arthur}; they have been updated or expanded since Arthur's book was published. The latter appeared in the preprint ``La formule des traces locales tordue''.}

\begin{hypothesis}\label{hypo:TwTF}
  Suppose that the twisted trace formula for $GL_n$ and twisted even orthogonal groups can be stabilized in the sense of \cite[Hypo 3.2.1]{Arthur} and \cite[Hypo 4.2.1]{Mok}.
\end{hypothesis}

 Even though we do not to strive to extract an optimal partial result from the current knowledge, see \S\ref{sub:unitary-groups} for some unconditional results not replying on the above hypothesis in case $G$ is unitary. Now recall from \S\ref{sub:transfer-app} that $\epsilon=0$ unless in case (ii)(b) where $\epsilon=1$.

%
%



\begin{prop}\label{p:trans-rational2}
  For every (finite and infinite) place $v$ of $F^+$, every $f_v\in C^\infty_c(\bG(F^+_v))$, every $\Delta_v^{\Wh}$-transfer $\phi_v\in C^\infty_c(G(F^+_v))$ of $f_v$, and every tempered $L$-parameter $\varphi_v:W_{F^+_v}\times SL_2(\C)\ra {}^L G$, we have an identity
  \beq\label{e:spec-trans}\sum_{\pi_v\in \tilde{LP}(\varphi_v)}  \Theta_{\pi_v}(\phi_v)=\Theta_{\Pi_v,\theta}(f_v),
  \quad \mathrm{where}~\Pi_v=\rec^{-1}(\eta\varphi_v)|\cdot|^{\epsilon/2}\eeq
  for a unique finite subset $\tilde{LP}(\varphi_v)$ of $\Irr^{\temp}(G(F^+_v))$ (independent of $f_v$ and $\phi_v$). The subsets $\tilde{LP}(\varphi_v)$ give a partition of $\Irr^{\temp}(G(F^+_v))$ where $\tilde{LP}(\varphi_v)$ and $\tilde{LP}(\varphi'_v)$ coincide exactly when $\eta\varphi_v$ is equivalent to $\eta\varphi'_v$ as $L$-parameters for $\bG(F^+_v)$ (and are disjoint otherwise).

\end{prop}

\begin{rem}
  Even though this would be clear to the reader, let us clarify the meaning of $\Pi_v$ in the proposition when $v$ splits as $ww^c$ in $F$, which can only happen in the unitary case (then $G(F^+_v)$ is a general linear group). Then $F_v=F\otimes_{F^+} F^+_v\simeq F_w\times F_{w^c}$, thereby one may write $\Pi_v=\Pi_w\otimes \Pi_{w^c}$. On the other hand, $\eta\varphi_v:W_{F^+_v}\times SL_2(\C)\ra GL_n(\C)^{\Hom_{F^+}(F,\C)}$ determines an $L$-parameter $\Phi_w$ for $GL_n(F_{w})$ and an $L$-parameter $\Phi_{w^c}$ for $GL_n(F_{w^c})$. Then $\Pi_v=\rec^{-1}(\eta\varphi_v)$ is defined by $\Pi_w=\rec^{-1}(\Phi_w)$ and $\Pi_{w^c}=\rec^{-1}(\Phi_w)$ in the usual sense. Actually in this case, $\Pi_v=\pi_v\otimes \pi_v$ where $\pi_v$ is the unique member of $\tilde{LP}(\varphi_v)$.
   Similarly if $v=ww^c$ in $F$ in the setting of Corollary \ref{c:pi_v-pure}, we interpret $\mL_{F_v}(\Pi_v)$ and $|\cdot|_v$ as $\mL_{F_w}(\Pi_w)\otimes \mL_{F_{w^c}}(\Pi_{w^c})$ and $|\cdot|_w|\cdot|_{w^c}$, respectively.
\end{rem}

\begin{rem}
  The set $\tilde{LP}(\varphi_v)$ is the local $L$-packet for $\varphi_v$ except when $G$ is an even orthogonal group, in which case it is a union of one or two $L$-packets. See the discussion above and below the theorem 1.5.1 of \cite{Arthur}. Our notation $\tilde{LP}(\varphi_v)$ corresponds to his $\tilde{\Pi}_\phi$.
\end{rem}

\begin{rem} Arthur also proved that
  when $\varphi_v$ is a non-tempered $A$-parameter, the analogue of \eqref{e:spec-trans} holds true if $\Theta_{\pi_v}(\phi_v)$ are summed with suitable signs. We will not need this for our theorems.
\end{rem}

\begin{proof}

  This is part of the main local theorem by Arthur (Theorems 1.5.1 and 2.2.1 of \cite{Arthur}) when $G$ is symplectic or orthogonal and by Mok (Theorems 2.5.1 and 3.2.1 of \cite{Mok}) when $G$ is unitary.


\end{proof}

  The above proposition tells us that for each $\pi_v\in \Irr^{\temp}(G(F^+_v))$ there is a unique (up to equivalence) tempered $L$-parameter $\varphi_v$ such that $\pi_v\in \tilde{LP}(\varphi_v)$. In this case we will write
  \beq\label{e:eta_*}\eta_*(\pi_v):=\rec^{-1}(\eta\varphi_v)|\cdot|^{\epsilon/2},\eeq
 cf. \eqref{e:spec-trans}. Namely $\eta_*: \Irr^{\temp}(G(F^+_v))\ra  \Irr^{\temp}(\bG(F^+_v))$ denotes the local functorial lifting given by $\eta$.

\begin{prop}\label{p:packet-size} There exists $m_G\in \Z_{\ge 1}$ such that for every finite place $v$ and every tempered $L$-parameter $\varphi_v$,
  $|\tilde{LP}(\varphi_v)|\le m_G$. To be explicit, one can choose $m_G=2^n$.
\end{prop}

\begin{proof}
  When $G$ is symplectic or orthogonal, the theorem 2.2.1.(b) of \cite{Arthur} says that there is a bijection between $\tilde{LP}(\varphi_v)$ and the set of characters on the group $\mathcal{S}_{\varphi_v}$ (denoted $\mathcal{S}_{\psi}$ in therein). According to \cite[(1.4.9)]{Arthur} $\mathcal{S}_{\varphi_v}$ is an abelian group whose order divides $2^n$ so the proposition follows. In the case of unitary groups one argues similarly using \cite[(2.4.14)]{Mok}.
\end{proof}

  Now we summarize some results on the global functoriality for classical groups that we will need. (A good number of cases also follow from the method of converse theorem and integral representations but we do not discuss them here). We will cite only \cite{Arthur} (which treats symplectic/orthogonal groups) in the remainder of this subsection without further comments on the unitary group case, believing that the reader understands by now that the completely analogous results in the latter case can be found easily in \cite{Mok}.

  Let us introduce some new notation, which is mostly consistent with that of \cite{Arthur} but not always. Write $\tilde{\Psi}_{\elp}(GL_n)$ for the set of quadruples  $$\psi=(r,\{(n_i,\Pi_i,\nu_i)\}_{i=1}^r)$$  (in which $(n_i,\Pi_i,\nu_i)$ are unordered relative to the index $i$) where
\bit
\item
  $r\in \Z_{\ge 1}$, $n_i\in \Z_{\ge 1}$, $\nu_i\in \Z_{\ge1}$, $\sum_{i=1}^r n_i\nu_i=n$,
\item
   $\Pi_i$ are cuspidal automorphic representations of $GL_{n_i}(\A_F)$, and such that
$\Pi_i^\vee\simeq \Pi_i^c$ for every $i$ and $\Pi_i\ncong \Pi_j$ for ever pair $i\neq j$.
 \eit
  Let $\tilde{\cE}^{\elp}(GL_n)$ denote the set of isomorphism classes of (twisted) endoscopic data for $\bG\rtimes \lg \theta\rg$ as defined in \cite[\S1.2]{Arthur}. Whether we are in case (i), (ii) or (iii), $(G,s,\eta)$ belongs to $\tilde{\cE}^{\elp}(GL_n)$. (In case (iii) our $\eta$ corresponds to the $L$-morphism $\xi_{\chi_+}$ of \cite{Mok}. His $\xi_{\chi_-}$ is not used in our paper.) According to a classification of self-dual parameters as in \cite[\S1.2]{Arthur}, there is a natural decomposition
  $$\tilde{\Psi}_{\elp}(GL_n)=\coprod_{H\in \tilde{\cE}^{\elp}(GL_n)} \tilde{\Psi}_2(H)$$
  so that $\psi$ belongs to $\tilde{\Psi}_2(H)$ if, loosely speaking, it satisfies the characterizing properties of the parameters coming from $H$. See the paragraph preceding \cite[(1.4.7)]{Arthur}.


  Let us explain the construction of local parameters from $\psi\in \tilde{\Psi}_{\elp}(GL_n)$. Put $\cL_{F^+_v}:=W_{F^+_v}\times SL_2(\C)$ if $v\nmid \infty$ and $\cL_{F^+_v}:=W_{F^+_v}$ if $v|\infty$. Define $$\psi_v:\cL_{F^+_v}\times SL_2(\C)\ra {}^L \bG$$ to be the $L$-parameter for $\bG(F^+_v)$ given by $\oplus_{i=1}^r \rec(\Pi_i)\otimes \Sym^{\nu_i-1}(\C^2)$, where each direct summand is the exterior tensor product of $\rec(\Pi_i)$ on $\cL_{F^+_v}$ and $\Sym^{\nu_i-1}(\C^2)$ on $SL_2(\C)$. If $\psi\in \tilde{\Psi}_2(G)$ then it is a nontrivial theorem that $\psi_v$ (or an isomorphic parameter thereof) factors through only $\eta:{}^L G\hra {}^L \bG$ and no embedding of other elliptic endoscopic groups. (See the theorem 1.4.2 and the discussion above (1.5.3) in \cite{Arthur}.) This determines  $\psi^{\flat}_v:\cL_{F^+_v}\times SL_2(\C)\ra {}^L G$ such that $\eta\psi^{\flat}_v\simeq \psi_v$ canonically up to $\Out(G)$-action. (The outer automorphism group has order 1 or 2. See \cite[\S1.2]{Arthur} for details.) It turns out that $\psi_v$ always lands in $\tilde{\Psi}^+_{\mathrm{unit}}(G(F^+_v))$ in the notation of Arthur, which is designed to accommodate local components of discrete automorphic representations of $G$. The precise definition of $\tilde{\Psi}^+_{\mathrm{unit}}(G(F^+_v))$ is not needed for our purpose so not recalled here.

  Now we turn to the purely local setting and explain some local inputs beyond the tempered objects to be used in this paper. Arthur associates to each $\psi_v\in \tilde{\Psi}^+_{\mathrm{unit}}(G(F^+_v))$ (which may not come from a global parameter $\psi$) a finite set $\tilde{\Pi}_{\psi_v}$ consisting of finite length $G(F^+_v)$-representations by extending the definition of tempered $L$-packets, i.e. $\tilde{\Pi}_{\psi_v}$ is the tempered $L$-packet (cf. paragraph above Proposition \ref{p:packet-size}) if $\psi_v$ is a tempered $L$-parameter. Although $\tilde{\Pi}_{\psi_v}$ is designed to play the role of local $A$-packets, it should be noted that members of $\tilde{\Pi}_{\psi_v}$ may be reducible or non-unitary. Let us define $\tilde{AP}(\psi_v)$ to be the set consisting of irreducible subquotients of the members of $\tilde{\Pi}_{\psi_v}$. 




\begin{prop}\label{p:twisted-endoscopy}
  Consider cases (i), (ii), or (iii) of \S\ref{sub:transfer-app}. Suppose that $\pi$ is a discrete automorphic representation of $G(\A_{F^+})$ unramified outside a finite set $S$. Then there exists a unique $\psi=(r,\{(n_i,\Pi_i,\nu_i)\}_{i=1}^r)\in \tilde{\Psi}_2(G)$ such that
   \benu
   \item $\pi_v\in \tilde{AP}(\psi_v)$ at every place $v$ of $F^+$,
      \item If $\pi_v$ is tempered and all $\nu_i$ are trivial then $\eta_*(\pi_v)=\boxplus_{i=1}^r \Pi_{i,v}|\cdot|^{\epsilon/2}$
		at each place $v$ of $F^+$,
   \item at every finite place $v\notin S$, $\pi_v$ is isomorphic to the unramified member of $\tilde{AP}(\psi_v)$, which is unique (relative to the fixed hyperspecial subgroup $U_v^{\hs}$).
  \eenu

\end{prop}

\begin{rem}\label{r:nontempered-case}
  In case some $\nu_i$ is nontrivial so that we are in the nontempered case, one knows from \cite{Arthur} only an equality of infinitesimal characters (i.e. supercuspidal support when $v$ is a finite place) in (ii). If we knew the Ramanujan conjecture for general linear groups, it would be enough assume in (ii) only that $\nu_i$ are trivial.
\end{rem}

\begin{proof}

  The first assertion is implied by \cite[Th 1.5.2]{Arthur}.
  In (ii) $\eta_*(\pi_v)$ is characterized by Proposition \ref{p:trans-rational2}, so the assertion follows from (2.2.3) of \cite{Arthur}.

  The last assertion is deduced from the theorem 1.5.1 of \cite{Arthur}, which implies that $\tilde{AP}(\psi_v)$ possesses at most one unramified representation. (One can identify $\pi_v$ a little more explicitly. When $\pi_v$ is unramified, $\psi_v$ is also unramified (i.e. all $\Pi_i$ are unramified at $v$). Then $\tilde{AP}(\psi_v)$ contains a local $L$-packet for the unramified $L$-parameter given by $\psi$, cf. \cite[Prop 7.4.1]{Arthur}, so $\pi_v$ is the one corresponding to the latter $L$-parameter via the unramified Langlands correspondence.)



\end{proof}

\begin{cor}\label{c:pi_v-pure}
   In the setting of Proposition \ref{p:twisted-endoscopy}, let $\psi=(r,\{(n_i,\Pi_i,\nu_i)\}_{i=1}^r)$ be the associated data to $\pi$ and suppose that $\pi$ is $\xi$-cohomological. Then there exists $s(\xi)\in\Z_{\ge0}$ depending only on $G$ and $\xi$ such that
    \bit
    \item for every finite place $v$, $\mL_{F_v}(\Pi_{i,v}|\cdot|^{\frac{\epsilon+n_i-n}{2}})|\cdot|_v^{-s(\xi)/2}$ is pure of weight $s(\xi)+n-1-\epsilon$ and integral.
     \eit
     If moreover the highest weight of $\xi$ is regular then
        \bit
        \item $\pi_v$ are tempered at all places $v$,
        \item $\eta\varphi_{\pi_v}$ are pure WD representations of weight $-\epsilon$ for all finite places $v$,
         \item $\eta\varphi_{\pi_v}$ are unramified and strictly pure of weight $-\epsilon$ if $v\notin S$.
         \eit


\end{cor}


\begin{proof}

  Let us begin by proving the first assertion.
     Proposition \ref{p:twisted-endoscopy}.(i) at infinite places implies, by the comparison of infinitesimal characters, that $\eta\psi_w|_{W_{\ol{F_w}}}$ is isomorphic to the direct sum over all infinite places $w$ of $F$ of the $L$-parameter for $\Pi_{i,w}$ restricted to $W_{\ol{F_w}}$. (Of course $\ol{F_w}\simeq \C$ for $w|\infty$.) Since $\pi$ is $\xi$-cohomological thus regular and C-algebraic, Lemma \ref{l:property-of-eta} implies that $\eta\psi_w$ is a regular $C$-algebraic parameter.
     From this it follows that $\Pi_{i,w}|\cdot|^{\frac{\epsilon+n_i-n}{2}}$ at $w|\infty$ are regular and C-algebraic.
   One deduces from Proposition \ref{p:Galois-reps} that there exists $s(\Pi_{i,\infty})\in \Z_{\ge 0}$ depending only on the infinite component $\Pi_{i,\infty}$ of $\Pi_{i}$ such that $\mL_{F_y}(\Pi_{i,y}|\cdot|^{\frac{\epsilon+n_i-n}{2}})|\cdot|_y^{-s(\Pi_{i,\infty})/2}$ is pure of weight $s(\Pi_{i,\infty})+n-1-\epsilon$ and integral for each $1\le i\le r$ for every finite place $y$ of $F$.
   Clearly there are only finitely many $W_{\ol{F_w}}$-subrepresentation of $\eta\psi_w|_{W_{\ol{F_w}}}$, so the number of all possible infinitesimal characters for $\Pi_{1,w},...,\Pi_{r,w}$ is finite at each $w|\infty$.
   Since there are only finitely many irreducible representations of $GL_m(\R)$ or $GL_m(\R)$ with fixed $m\in \Z_{\ge 1}$ and fixed infinitesimal character, there are only finitely many possibilities for $\Pi_{i,\infty}$. The proof of the first assertion is complete as soon as $s(\xi)$ is taken to be the maximum of $s(\Pi_{i,\infty})$ over all possible $\{\Pi_{i,\infty}\}_{1\le i\le r}$.



   Now suppose that the highest weight of $\xi$ is regular. According to a standard result on Lie algebra cohomology, $\pi_v$ at $v|\infty$ must be discrete series to be $\xi$-cohomological. Considering infinitesimal characters for $\psi_v$ at $v|\infty$, we see that $\nu_i=1$ for all $1\le i\le r$. Since $\Pi_i|\cdot|^{\frac{\epsilon+n_i-n}{2}}$ is of type (TR) or (CM) for each $i$, Proposition \ref{p:Galois-reps} tells us that $\Pi_{i,v}$ are essentially tempered at all finite places $v$. Since $\Pi_{i,v}$ is already known to be unitary, $\Pi_{i,v}$ is tempered. Hence $\psi_v$ is tempered and  $\tilde{AP}(\psi_v)$ is nothing but the tempered $L$-packet $\tilde{LP}(\psi_v|_{\cL_{F_v^+}})$ at each $v\nmid \infty$, cf. Proposition \ref{p:trans-rational2}. In particular $\pi_v\in\tilde{AP}(\psi_v)$ is tempered. Since $$\eta\varphi_{\pi_v}=\oplus_{i=1}^r \rec(\Pi_{i,v})|\cdot|^{\epsilon/2}=\oplus_{i=1}^r \mL_{F^+_v}(\Pi_{i,v}|\cdot|^{\frac{\epsilon+n_i-n}{2}})|\cdot|^{\frac{n-1}{2}}, \quad v\nmid \infty,$$
   Proposition \ref{p:Galois-reps} and Remark \ref{r:Galois-reps} allow us to verify the properties of $\eta\varphi_{\pi_v}$ in the corollary.


\end{proof}

\section{Finiteness results}\label{s:finiteness}

  The first two subsections prove local finiteness results for unramified and arbitrary representations. After stating a global finiteness conjecture (Conjecture \ref{c:finiteness} below) for C-algebraic representations with bounded coefficient fields in a fairly general setting, we establish the conjecture for general linear groups and quasi-split classical groups.

\subsection{Finiteness for unramified representations}\label{sub:finite-unr}

  Put ourselves in the setting of \S\ref{sub:transfer-app}.


\begin{lem}\label{l:finite-unr} Fix $s\in \Z_{\ge 0}$, $A\in \Z_{\ge1}$, and a finite place $v$ of $F^+$.
  There are only finitely many $\pi_v\in \Irr^{\ur}(G(F^+_v))$ 
  such that\bit \item $\mL_{F^+_v}(\eta_*\pi_v)|\cdot|_v^{-s/2}$ is strictly pure of weight $n-1+s-\epsilon$ and integral and
  \item $[\Q(\pi_v):\Q]\le A$.
  \eit
\end{lem}

\begin{proof}

   Since the map $\eta_*:\Irr^{\ur}(G(F^+_v))\ra \Irr^{\ur}(\bG(F^+_v))$ has finite fibers (Lemma \ref{l:bound-fiber}) it suffices to prove the finiteness of the set of $\Pi_v\in \Irr^{\ur}(GL_n(F_v))$ such that $\mL_{F^+_v}(\Pi_v)|\cdot|_v^{-s/2}$ is strictly pure of weight $n-1+s$ and integral with $[\Q(\Pi_v):\Q]\le  A$. A first observation is that any $\Pi_v=\eta_*\pi_v$ for $\pi_v$ as in the lemma lands in the set just defined, where the inequality follows from Lemma \ref{l:unram-functoriality}.(ii). Next consider the bijection $\mS:\Irr^{\ur}(GL_n(F_v))\ra(\C^\times)^n/S_n$ coming from the Satake isomorphism for $GL_n$. Then each complex number appearing in $\mS(\Pi_v|\cdot|_v^{\frac{-(n-1+s-\epsilon)}{2}})$ must be a root of an (irreducible) monic polynomial $x^m+a_{m-1}x^{m-1}+\cdots + a_0$ with \beq\label{e:fin-unr-1}1\le m\le  A, \quad a_{m-1},...,a_0\in \Z\eeq by integrality and the bound on $[\Q(\Pi_v):\Q]$. The condition on purity and weight (``Weil bounds'') implies that $|\lambda|_v\le q_v^{\frac{-(n-1+s-\epsilon)}{2}}$ for all roots $\lambda\in \C$ of the above polynomial, imposing a constraint \beq\label{e:fin-unr-2}|a_i|_v\le {m \choose i} q_v^{\frac{-(n-1+s-\epsilon)}{2}},\quad \forall 0\le i\le m-1.\eeq As there are only finitely many polynomials satisfying \eqref{e:fin-unr-1} and \eqref{e:fin-unr-2}, we are done.
\end{proof}

\subsection{Rationality of endoscopic transfer}\label{sub:rational-transfer}

  Keep the notation of the previous subsection.
  We start by studying the behavior of the functorial lifting $\eta_*$ relative to automorphisms of $\C$.

\begin{prop}\label{p:coeff-field-G-GL(n)} Let $\pi_v\in \Irr^{\temp}(G(F^+_v))$. 
Then \beq\label{e:sigma-transfer}\eta_*(\pi_v^\sigma)=(\eta_*\pi_v)^\sigma,\quad \forall\sigma\in \Aut(\C).\eeq
If moreover $\Q(\pi_v)$ is finite over $\Q$ then \benu
\item $\Q(\eta_*\pi_v)$ is also finite over $\Q$.
\item $\Q(\pi_v)$ contains $\Q(\eta_*\pi_v)$ and is contained in a finite extension of $\Q(\eta_*\pi_v)$ of degree at most $m_G!$. In particular
$[\Q(\eta_*\pi_v):\Q]\le [\Q(\pi_v):\Q]\le m_G!\,[\Q(\eta_*\pi_v):\Q]$.
\eenu


\end{prop}

\begin{rem}
  Only the left inequality in (ii) will be needed in our main results. The proposition extends Lemma \ref{l:unram-functoriality} from unramified (possibly non-tempered) representations to tempered representations in the case of classical groups.
\end{rem}

\begin{proof}
  Put $\Pi_v:=\eta_*\pi_v$. Since (i) is an immediate consequence of (ii), it suffices to verify (ii).


  When $\pi_v$ is tempered, we would like to verify \eqref{e:sigma-transfer}.
  For any $f_v\in C^\infty_c(\bG(F^+_v))$ let $\phi_v\in C^\infty_c(G(F^+_v))$ be its $\Delta_v^{\Wh}$-transfer.
  For every $\sigma\in \Aut(\C)$ we obtain from Proposition \ref{p:trans-rational} and \eqref{e:KLS-trans} that $$STO^{\bG(F^+_v)}_{\delta}(f^{\sigma}_v)=\sum_{\gamma\sim_{\mathrm{st}} \delta} \Delta_v^{\Wh}(\gamma,\delta) O^{G(F^+_v)}_\gamma (\phi^{\sigma}_v).$$ Hence $\phi^{\sigma}_v$ is a KLS-transfer of $f^{\sigma}_v$.
  On the other hand, twisting \eqref{e:spec-trans} by $\sigma$ leads to an identity
   \beq\label{e:coeff-field-1} \Theta_{\Pi^{\sigma}_v,\theta}(f^{\sigma}_v)=\sum_{\eta_*(\rho_v)=\Pi_v}  \Theta_{\rho^{\sigma}_v}(\phi^{\sigma}_v).\eeq
  Plugging in $f^{\sigma^{-1}}_v$ and $\phi^{\sigma^{-1}}_v$ in place of $f_v$ and $\phi_v$ (noting that $f^{\sigma^{-1}}_v$ is a $\Delta_v^{\Wh}$-transfer of $\phi^{\sigma^{-1}}_v$) we derive
  \beq\label{e:coeff-field-2} \Theta_{\Pi^{\sigma}_v,\theta}(f_v)=\sum_{\eta_*(\rho_v)=\Pi_v} \Theta_{\rho^{\sigma}_v}(\phi_v).\eeq
  Comparing with $\Theta_{\Pi^{\sigma}_v,\theta}(f_v)=\sum_{\eta_*(\rho_v)=\Pi^\sigma_v} \Theta_{\rho_v}(\phi_v)$, cf. \eqref{e:spec-trans}, we obtain an equality of stable characters (evaluated on elements of $ C^\infty_c(G(F^+_v))$)
   \beq\label{e:st-char-equal}\sum_{\eta_*(\rho_v)=\Pi_v}  \Theta_{\rho^{\sigma}_v}=\sum_{\eta_*(\rho_v)=\Pi^\sigma_v}  \Theta_{\rho_v}\eeq
   since $\Delta_v^{\Wh}$-transfers of $C^\infty_c(\bG(F^+_v))$ generate the space of stable distributions on $G(F^+_v)$. (In the language of Remark 1 below Theorem 2.2.1 of \cite{Arthur}, the map $\tilde{f}\mapsto \tilde{f}^G$ is \emph{onto}.) Then \eqref{e:st-char-equal} holds true also as the equality of finite character sums. Since $\Theta_{\pi^\sigma_v}$ appears as a summand on the left hand side, it should also on the other side by linear independence of characters. We have established \eqref{e:sigma-transfer}.

   Formula \eqref{e:sigma-transfer} readily implies that if $\sigma\in \Aut(\C/\Q(\pi_v))$ then $\Pi^\sigma_v=\eta_*(\pi^\sigma_v)=\eta_*\pi_v=\Pi_v$. Therefore $\Q(\Pi_v)\subset \Q(\pi_v)$ and in particular $\Q(\Pi_v)$ is finite over $\Q$.

   Now if $\sigma\in \Aut(\C/\Q(\Pi_v))$ then $\eta_*(\pi_v^\sigma)=\Pi_v^\sigma=\Pi_v$. One deduces from \eqref{e:st-char-equal} that $\pi_v,\pi_v^\sigma\in \eta_*^{-1}(\Pi_v)$. Thereby one obtains a group homomorphism $$\Upsilon:\Aut(\C/\Q(\Pi_v)) \ra \Perm(\eta_*^{-1}(\Pi_v)),\quad \Upsilon(\sigma):\pi_v\mapsto \pi_v^\sigma$$ where $\Perm(\cdot)$ denotes the permutation group. Since $|\eta_*^{-1}(\Pi_v)|\le m_G$ by Proposition \ref{p:packet-size}, the kernel of $\Upsilon$ has index $\le m_G!$ in $\Aut(\C/\Q(\Pi_v))$. Then the fixed field of $\ker \Upsilon$ is a finite extension of $\Q(\Pi_v)$ of degree $\le m_G!$ and contains $\Q(\pi_v)$. The proof of (ii) is finished.
\end{proof}

\begin{rem}
  Alternatively \eqref{e:sigma-transfer} may be proved using a global argument (explained to us by Wee Teck Gan):
  Reduce to the case where $\pi_v$ is a discrete series. When $\pi_v$ is discrete, globalize $\pi_v$ to a $\pi$. Consider $\eta_*(\pi)$ and $\eta_*(\pi^\sigma)$ (assumed isobaric). By comparing $\eta_*(\pi)^\sigma$ and $\eta_*(\pi^\sigma)$ at almost all unramified places, one deduces from strong multiplicity one that $\eta_*(\pi_v)^\sigma=\eta_*(\pi^\sigma_v)$ at the place $v$ of interest. (Use Clozel's result that the Langlands quotient is compatible with $\sigma$.)
\end{rem}

\subsection{Sparsity of arithmetic points in the unitary dual}\label{sub:sparsity}

\begin{prop}\label{p:sparsity-GL(n)}
  Fix $A\ge 1$, a finite place $v$ of $F^+$, an open compact subgroup $K_v\subset \bG(F^+_v)$ and an irreducible algebraic $\bG_\infty$-representation $\Xi$.
  The set of $\Pi_v\in \Irr(\bG(F^+_v))$ satisfying (i) and (ii) below is finite:
  \benu
  \item $\Pi_v$ appears as the $v$-component of some $\Xi$-cohomological isobaric representation $\Pi=\boxplus_{i=1}^s \Pi_i$ of $\bG(\A_{F^+})$ such that $\Pi_i$ are cuspidal and $\Pi_i^\vee\simeq \Pi_i^c \otimes (\det\circ \chi_i)$ for $\chi:F^{\times}\bs \A^\times_F\ra \C^\times$ such that  $\chi_v(-1)$ is the same for every $v\in S_\infty$,
  \item $[\Q(\Pi_v):\Q]\le A$.
  \eenu
\end{prop}

\begin{rem}\label{r:parity-issue} Since C-algebraicity is incompatible with $\boxplus$ (which is locally the Langlands quotient for the normalized induction), $\Pi$ being C-algebraic implies not
  $\Pi_i$ but $\Pi_i|\cdot|^{\frac{n_i-n}{2}}$  is C-algebraic in condition (i).
\end{rem}

\begin{proof}
  By (ii) and Corollary 3.13 the depth (or conductor) of $\Pi_v$ is bounded, so the set of $\Pi_v$ is contained in finitely many Bernstein components.
 We may show that the set of $\Pi_v$ satisfying (i) and (ii) is finite in each Bernstein component $\cB$. Suppose $\Pi^0_v\in \cB$ satisfies (i) and (ii). Write $\mL_{F^+_v}(\scusp(\Pi^0_v))=\oplus_{i=1}^k V_i$ where $V_i$ are irreducible WD representations. For any other $\Pi_v\in \cB$,
    $$\mL_{F^+_v}(\scusp(\Pi_v))=\oplus_{i=1}^k V_i\otimes \mathrm{unr}(\lambda_i)$$
  where $\mathrm{unr}(\lambda_i):GL_1(F^+_v)\ra \C^\times$ is the unramified character mapping every uniformizer of $F^+_v$ to $\lambda_i\in \C^\times$. Since $\Q(\oplus_{i=1}^k V_i)\subset \Q(\Pi^0_v)$ (cf. \eqref{e:mL_K-twist}), we have $[\Q(\oplus_{i=1}^k V_i):\Q]\le A$. Put $E:=\Q(V_1)\Q(V_2)\cdots\Q(V_k)$. Since $\Gal(\ol{\Q}/\Q(\oplus_{i=1}^k V_i))$ acts on $\{V_1,...,V_k\}$ (a multi-set) as permutations,
  $$ [E:\Q]\le k! A\le n! A.$$
  Consider the action of $\Gal(\ol{\Q}/E)$ on the unordered set $\{V_i\otimes \unr(\lambda_i)\}_{i=1}^k$. Clearly there exists an extension $E'/E$ of degree $\le k!$ such that $\Gal(\ol{\Q}/E')$ fixes the isomorphism class of $V_i\otimes \unr(\lambda_i)$ for every $i$. Observe that every $\lambda\in \C^\times$ such that $V_i\otimes \unr(\lambda)\simeq V_i$ satisfies $\lambda^n=1$. (For this consider the equality of the determinants.) Setting $E'':=E(\mu_n)$, we conclude that $\Gal(\ol{\Q}/E'')$ fixes $\lambda_i$ for every $i$. In particular
  \beq\label{e:Sparsity}[\Q(\lambda_i):\Q]\le k!n!\varphi(n)A, \quad \forall 1\le i\le k.\eeq

  By (i) and Proposition \ref{p:Galois-reps} there exists $s>0$ (depending on $\Xi$) such that for every $\Pi_v\in \cB$ satisfying (i) and (ii), $\mL_{F^+_v}(\scusp(\Pi_v))|\cdot|^{-s/2}$ is pure of weight $s+(n-1)$ and integral. (To deduce this, apply Proposition \ref{p:Galois-reps} to each $\Pi_i|\cdot|^{\frac{n_i-n}{2}}$ , cf. Remark \ref{r:parity-issue}.) As Lemma \ref{l:pure-shape} applies to the present situation with all $s_i$ in the lemma equal to 1, we see that both $V_i|\cdot|^{-s}$ and $V_i|\cdot|^{-s}\otimes \mathrm{unr}(\lambda_i)$ are strictly pure of weight in $s+(n-1)$ and integral. We claim that $\lambda_i$ is a Weil $q_v$-number of weight 0 such that $q_v^{s+(n-1)}\lambda_i$ is integral. Indeed, for any any eigenvalue $\omega$ of a lift of geometric Frobenius on $V_i$, we know from the above that $q_v^{\frac{s+n-1}{2}}/\omega$ and $\omega\lambda_i$ are integral. In view of the integrality of $q_v^{s+(n-1)}\lambda_i$ and \eqref{e:Sparsity}, an argument as in Lemma \ref{l:finite-unr} shows that there are only finitely many $\lambda_i$ with these properties. Therefore the set of $\Pi_v$ as above is finite.


\end{proof}

  Let $G$ be as in (i), (ii) or (iii) of \S\ref{sub:transfer-app}. Recall that Arthur and Mok associate to
  $\pi_v\in \Irr^{\temp}(G(F^+_v))$ a tempered $L$-parameter $\varphi_{\pi_v}:W_{F^+_v}\times SL_2(\C)\ra {}^L G$. In a standard manner this extends to the construction of all $L$-packets via Langlands quotients. Now we are about to state a result providing a crucial input for the main results of \S\ref{sub:growth-level}.


\begin{cor}\label{c:sparsity}
  Fix $A\ge 1$ and an irreducible algebraic $\Res_{F^+/\Q} G$-representation $\xi$. Suppose that $\xi$ has regular highest weight.
  Then the set of $\pi_v\in \Irr^{\temp}(G(F^+_v))$ satisfying the two properties below is finite.
  \benu
  \item $\pi_v$ appears as the $v$-component of some $\xi$-cohomological discrete automorphic representation $\pi$ of $G(\A_{F^+})$,
  \item $[\Q(\pi_v):\Q]\le A$.
  \eenu
\end{cor}

\begin{rem}
  In particular the set of such $\pi_v$ is measurable with respect to the Plancherel measure on $G(F^+_v)^{\wedge}$. We caution the reader that its measure may not be zero. Indeed it has positive Plancherel measure precisely when it contains discrete series.
\end{rem}

\begin{rem}\label{r:regularity}
  The regularity assumption on $\xi$ should be unnecessary for the corollary to be true. We imposed it for simplicity and also for the reason that the same hypothesis will be in place for applications in Section~\ref{s:growth-coeff}.
\end{rem}

\begin{proof}

  Let $C(\xi,A)$ be the set of $\pi_v$ as above. We need to show $|C(\xi,A)|<\infty$. Since $\eta_*$ is a finite-to-one map, we will be done if $\eta_*$ is shown to map $C(\xi,A)$ into the union of the sets of Proposition \ref{p:sparsity-GL(n)} for some $\Xi$, where $\Xi$ depends only on $\xi$.

   Each $\pi_v\in C(\xi,A)$ is the $v$-component of some $\pi$ as in the corollary. Let $\psi=(r,\{(n_i,\Pi_i,\nu_i)\}^r_{i=1}$ be the data associated with $\pi$ and put $$\Pi:=\boxplus_{i=1}^r\left( (\Pi_i\otimes |\cdot|^{\frac{1-\nu_i}{2}}) \boxplus(\Pi_i\otimes |\cdot|^{\frac{3-\nu_i}{2}})
   \boxplus \cdots \boxplus (\Pi_i\otimes |\cdot|^{\frac{\nu_i-1}{2}})\right)\otimes |\cdot|^{\epsilon/2} .$$
   The infinitesimal character of $\pi_v$ at each infinite place $v$ of $F^+$ is the same as that of $\xi^\vee_v$, where $\xi=\otimes_{v|\infty} \xi_v$ is a tensor product of irreducible representations with regular highest weights. Remark \ref{r:nontempered-case} tells us that this infinitesimal character transfers via $\eta$ to that of $\Pi_v$. The regularity on the former implies via the explicit description of $\eta$ that the infinitesimal character of $\Pi_v$ is the same as that of (the tensor product of two) irreducible algebraic representations of $\bG$ of regular highest weight. In particular $\nu_i$ must be all trivial in $\psi$.

    So $\Pi=\boxplus_{i=1}^r \Pi_i|\cdot|^{\epsilon/2}$ and $\Pi_v=\eta_*(\pi_v)$ by (ii) of Proposition \ref{p:twisted-endoscopy}, cf. Remark \ref{r:nontempered-case}.  Thanks to Proposition \ref{p:coeff-field-G-GL(n)} we know $[\Q(\Pi_v):\Q]\le A$, which is (ii) of Proposition \ref{p:sparsity-GL(n)}. It remains to verify (i) of that proposition. This is clear except possibly the property that $\Pi$ is cohomological (for some $\Xi$), which we now explain. Since $\pi_\infty$ is $\xi$-cohomological it is regular C-algebraic. This implies that $\Pi_\infty=\eta_*(\pi_\infty)$ is also regular C-algebraic (for $GL_n$) by Lemma \ref{l:property-of-eta} (and the sentence right below Definition \ref{d:C-preserving}). Now Lemma \ref{l:GL(n)-cohomlogical} tells us that $\Pi_\infty$ is $\Xi$-cohomological for some $\Xi$. Moreover $\Xi$ is determined by the infinitesimal character of $\Pi_\infty$ hence also by that of $\pi_\infty$, or just by $\xi$. We are done.

\end{proof}

\subsection{A finiteness conjecture}\label{sub:fin-conj}

  We think this is a good place to state an interesting finiteness conjecture on automorphic representations in the spirit of the Shafarevich conjecture (Theorem \ref{t:intro-Faltings}). 
  Earlier Fontaine and Mazur proposed the analogue of the Shafarevich conjecture for $l$-adic Galois representations (\cite[I.\S3]{FM95}). While their conjecture is still mostly open in dimension $>1$ to our knowledge, we are able to verify \emph{our} conjecture in many cases including $G=GL_n$. This opens up the possibility for an automorphic proof of Fontaine-Mazur's finiteness conjecture via the Langlands correspondence. At the moment we are unable to get many cases of their conjecture since the correspondence is established in only limited cases. We wish to return to this problem in the future.

  In the conjecture $G$ is allowed to be an arbitrary connected reductive group over any number field $F$. Let $S_{\ram}$ be the finite set of finite places $v$ such that $G\times_F F_v$ is ramified (i.e. either non-quasi-split over $F_v$ or non-split over any finite unramified extension of $F_v$.) Recall that hyperspecial subgroups outside $S_{\ram}$ are fixed as in \S\ref{sub:notation} once and for all, and unramified representations are considered with respect to this data. Denote by $Z(\fkg)$ the center of the universal enveloping algebra of $\Lie G(F\otimes_\Q \C)$.


\begin{conj}\label{c:finiteness}
  Fix $A\in\Z_{\ge 1}$, $S$ a finite set of places of $F$ containing $S_{\ram}$ and all infinite places, and a $\C$-algebra character $\chi_\infty:Z(\fkg)\ra \C$. Then the set of discrete automorphic representations $\pi$ of $G(\A_F)$ with the following properties is finite:
  \bit
  \item $\pi^S$ is unramified,
  \item $\pi_\infty$ has infinitesimal character $\chi_\infty$, and
  \item $[\Q(\pi):\Q]\le A$.
  \eit
\end{conj}

\begin{rem}\label{r:modest}
  To state a more modest conjecture, one may replace the condition $[\Q(\pi):\Q]\le A$ by the condition that $\Q(\pi)$ is contained in a fixed finite extension of $\Q$ in $\C$.
\end{rem}

\begin{rem}
  Since there are up to isomorphism only finitely many $\pi_\infty$ with a fixed infinitesimal character $\chi_\infty$, one may replace the above condition on the infinitesimal character by the condition that $\pi_\infty$ is isomorphic to a fixed irreducible $G(F\otimes_\R \C)$-representation $\pi^0_\infty$.
\end{rem}

\begin{rem}
  The $\pi$ as in the conjecture should be C-algebraic according to the ``if'' part of Conjecture \ref{c:alg-arith}, which may well belong to the realm of transcendental number theory and would be difficult to check. Fortunately we can still verify the conjecture in many cases without a priori knowledge that $\pi$ is C-algebraic. cf. \S\ref{sub:fin-results} below.
\end{rem}


\begin{rem}
  It should be stressed that no bound on ramification is imposed at places in $S$. (Otherwise the conjecture would be uninteresting.) Such a bound is only a consequence of the condition that $\Q(\pi)\subset E$, at least in the setting of \S\ref{sub:fin-results} below. The conjecture is certainly false if the condition $\Q(\pi)\subset E$ is omitted, as it is often well known that there are infinitely many discrete automorphic representations if arbitrary ramification is allowed at one place, cf. \cite{Shi-Plan}.
\end{rem}


\begin{rem}
   On the Galois side (as opposed to the automorphic side) the analogues of the Fontaine-Mazur finiteness conjecture for complex and mod $l$ Galois representations have been proposed and investigated by \cite{ABCZ94} and \cite{Kha00}. The result of Anderson-Blasius-Coleman-Zettler~\cite{ABCZ94} is as follows: Given a number field $K$, there are finitely many complex representations of the Weil group $W_K$ of bounded degree and bounded Artin conductor (their proof uses Jordan theorem that finite subgroups of $\GL(d,\C)$ are virtually abelian). Their result confirms some very special cases of the Fontaine-Mazur finiteness conjecture.

\end{rem}

\subsection{Results on the finiteness conjecture}\label{sub:fin-results}

  The aim of this section is to prove Conjecture \ref{c:finiteness} in some important cases. Namely the conjecture will be established first in the case of general linear groups taking Lemma \ref{sub:bounded-ram} and Harish-Chandra's finiteness theorem (Proposition \ref{p:HC-finite} below) as crucial inputs, and next in the case of classical groups via functorial transfer to general linear groups.

\begin{prop}\label{p:HC-finite}
  For any $\C$-algebra character $\chi_\infty:Z(\fkg)\ra \C$ and for any open compact subgroup $U$ of $G(\A_F^\infty)$, the set of isomorphism classes of discrete automorphic representations $\pi$ of $G(\A_F)$ satisfying the following is finite.
  \bit
  \item $\pi^\infty$ has a nonzero $U$-fixed vector and
  \item $\pi_\infty$ has infinitesimal character $\chi_\infty$.
  \eit
\end{prop}

\begin{rem}
  We are using a weaker version of Harish-Chandra's theorem in that our attention is restricted to the discrete automorphic spectrum.
\end{rem}


\begin{proof}
  The proposition results immediately from Harish-Chandra's theorem 1 in \cite{HC68}. (The proof in \textit{loc. cit.} for semisimple groups is extended to the case of reductive groups as explained in \cite[Th 7.4]{Bor07}.)
\end{proof}

\begin{thm}\label{t:fin-GL_n}
  Conjecture \ref{c:finiteness} is true in the case of $G=GL_n$ for any $n\in \Z_{\ge1}$ and any ground field $F$.
\end{thm}

\begin{proof}
  Suppose that $\pi$ satisfies the condition of Conjecture \ref{c:finiteness}. Corollary \ref{c:bounded-conductor} tells us that $\pi_v$ has bounded conductor (depending only on $A$, $F_v$ and $n$) at every $v\in S$. Therefore the cardinality of such $\pi$ is finite by Proposition \ref{p:HC-finite}.
\end{proof}

\begin{thm}\label{t:fin-classical}
  Conjecture \ref{c:finiteness} is true for quasi-split classical groups as in (i), (ii) and (iii) of \S\ref{sub:transfer-app} (if Hypothesis \ref{hypo:TwTF} is assumed).
\end{thm}

\begin{proof} Write $C(G,S,\chi_\infty,A)$ for the set of Conjecture \ref{c:finiteness} (but we adopt the notation of \S\ref{sub:transfer-app} in this proof, so $F^+$ plays the role of $F$ in the conjecture). Consider the association
 $$\begin{array}{ccc}C(G,S,\chi_\infty,A)&\ra& \tilde{\Psi}_2(G)\\
 \pi&\mapsto& \psi=(r,\{(n_i,\Pi_i,\nu_i)\}_{i=1}^r \end{array}$$ as in Proposition \ref{p:twisted-endoscopy}. Since $\pi$ is unramified outside $S$, the associated $\Pi_i$ enjoys the same property for every $i$. Since $\eta$ is C-preserving by inspection, Lemma \ref{l:unram-functoriality}, Proposition \ref{p:twisted-endoscopy}.(iii), and the strong multiplicity one theorem imply that $\Aut(\C/\Q(\pi))$ permutes the set $\{\Pi_1,...,\Pi_r\}$. Hence there exists some $E\supset \Q(\pi)$ with $[E:\Q(\pi)]\le r!$ such that $\Aut(\C/E)$ fixes all $\Pi_1$, ..., $\Pi_r$. Note that $[E:\Q]\le r! [\Q(\pi):\Q]\le n! A$.

 Let us make some observation about infinitesimal characters. It is standard that $\chi_\infty$ corresponds to a collection of complex $L$-parameters $\varphi_{\chi_\infty,w}:W_{\C}\ra \hat{G}$ where $w$ runs over the infinite places of $F$. For each infinite place $v$ of $F^+$, it follows from $\pi_v\in \tilde{AP}(\psi_v)$ that $\psi_v|_{W_{\ol{F^+_v}}}$ is isomorphic to  $\varphi_{\chi_\infty,w}$ up to $\Out(\hat{G})$-action when $v|w$. Let $X_v$ be the set of all infinitesimal characters of $GL_m(F\otimes_{F^+} F^+_v)$ with $1\le m\le n$ corresponding to a $W_{\ol{F^+_v}}$-subrepresentation of $\eta\psi_v|_{W_{\ol{F^+_v}}}$ at each $v|\infty$. Clearly $X:=\prod_{v|\infty} X_v$ is a finite set.
  Proposition \ref{p:twisted-endoscopy}, cf. the proof of Corollary \ref{c:pi_v-pure}, tells us that the infinitesimal character of $\Pi_{i,v}$ belongs to $X_v$.

  Let $D(GL_{\le n},S,X,n!A)$ be the set of cuspidal automorphic representations $\Pi$ of $GL_m(\A_F)$ with $1\le m\le n$ which are unramified outside $S$, have the infinitesimal character of $\Pi_\infty$ in $X$, and satisfy $[\Q(\Pi):\Q]\le n!A$. According to Theorem \ref{t:fin-GL_n}, $D(GL_{\le n},S,X,n!A)$ is a finite set. We have seen that for any $(r,\{(n_i,\Pi_i,\nu_i)\}_{i=1}^r\in \tilde{\Psi}_2(G)$ coming from $\pi\in C(G,S,\chi_\infty,A)$, every $\Pi_i$ belongs to $D(GL_{\le n},S,X,n!A)$. Hence the image of $C(G,S,\chi_\infty,A)$ in $\tilde{\Psi}_2(G)$ is finite. The the proof boils down to showing that each fiber of the arrow $C(G,S,\chi_\infty,A)\ra\tilde{\Psi}_2(G)$ is finite. This is a consequence of Proposition \ref{p:twisted-endoscopy}.(iii): If $\pi$ is in the preimage of $\psi$ then $\pi_v$ at every finite place $v\notin S$ is determined to the unique unramified member of $\tilde{AP}(\psi_v)$. For each $v\in S$ or $v|\infty$, $\pi_v$ must lie in the finite set $\tilde{AP}(\psi_v)$.

\end{proof}

\section{Growth of fields of rationality in automorphic families}\label{s:growth-coeff}

Let $G$ be a quasi-split classical group as in (i), (ii)' or (iii) of ~\S\ref{sub:transfer-app} from here up to \S\ref{sub:approx}. In particular $F^+$ denotes the base field of $G$. Note that this is different from the convention of \cite{ST11cf}, which will be frequently cited in this section, where $F$ is the base field. (We restrict from (ii) to (ii)' since we need results from \cite{ST11cf} established under the assumption that $G$ has discrete series at the real places of $F^+$.) In \S\ref{sub:unitary-groups} we concisely explain how the earlier part of this section can be adapted to obtain an unconditional result for (non-quasi-split) unitary groups. Throughout this section it is assumed that $G$ is nontrivial so that the absolute rank of $G$ is at least one.

\subsection{Growth of fields of rationality in level aspect}\label{sub:growth-level}

We start by recalling the level-aspect families $\cF_\xi(U_x)$ of automorphic representations of $G(\A_{F^+})$ of weight $\xi$ and level subgroups $U_x$ as in~\cite[\S9.3]{ST11cf} or \cite{Shi-Plan}.

Let $U_x$ be a sequence of level $\Fmn_x$-subgroups of $G(\A_{F^+}^\infty)$. Here $\Fmn_x$ is a sequence of integral ideals of $\cO_{F^+}$ such that $\N(\Fmn_x):=[\cO_{F^+}:\Fmn_x]\to \infty$. When $G$ is a split group over $F$, the sequence $U_x$ is defined as
$$ U_x = \ker(G(\cO_{F^+})\ra G(\cO_{F^+}/\Fmn_x))$$
using the Chevalley group scheme for $G$ over $\Z$. In general we refer to \cite[\S8]{ST11cf} for the precise definition via Moy-Prasad filtrations.
We have a product decomposition $U_x=\prod_{v\nmid \infty} U_{x,v}$ such that each $U_{x,v}$ is a compact open subgroup of $G(F^+_v)$. Set $U_x^v:=\prod_{w\nmid \infty,w\neq v} U_{x,w}$.

A variant of the level sequence would be a tower of bounded depth in the sense of~\cite{DH99}, which corresponds to $U_{x+1}\subset U_{x}$ or $\Fmn_x\mid \Fmn_{x+1}$. But here we prefer to work more generally with the condition that $\N(\Fmn_x)\to \infty$. 


Let $\xi$ be an irreducible algebraic representation of $\Res_{{F^+}/\Q}G$ over $\C$, which can be viewed as a representation of $\prod_{v|\infty} G\times_{F^+} F^+_v$ where $v$ runs over infinite places of ${F^+}$ (see~\S\ref{sub:Clozel-BHR}). In this section we assume, except for Corollary~\ref{c:growth-finite}, that the highest weight for every representation of $G\times_{F^+} F^+_v$ induced by $\xi$ is \emph{regular}. The regularity assumption is made mainly because the equidistribution theorems as in \cite{Shi-Plan} and \cite{ST11cf} rely on it. (This is why we also made the assumption earlier for simplicity, cf. Remark \ref{r:regularity}).

Let $S_0$ be a (possibly empty) finite set of places disjoint from $\Fmn_x$ for all $x$. Let $\hat f_{S_0}$ be a well-behaved function on the unitary dual of $G(F^+_{S_0})$ in the sense of~\cite[\S7]{Sau97} and~\cite[\S9.1]{ST11cf}. (Such functions are very useful in prescribing interesting local conditions. Namely we can impose that $\pi_{S_0}$ belongs to a bounded measurable subset of $G(F^+_{S_0})^{\wedge}$ whose boundary has zero Plancherel measure and whose image in the Bernstein variety ($\Theta(G)$ in \cite[pp.164-165]{Sau97}) has compact closure. In fact there is essentially no loss of generality in assuming that $\hat f_{S_0}$ is a characteristic function of such a subset.) Henceforth we will assume that $\pl_{S_0}(\hat{f}_{S_0})>0$ and that $\hat{f}_{S_0}$ takes nonnegative real values on the unitary dual.

Let $\cF_x:=\cF(U_x,\hat f_{S_0},\xi)$ be the set (or family) of discrete automorphic representations $\pi$ of $G(\A_{F^+})$ such that
\begin{itemize}
  \item $\pi_{S_0}$ belongs to the support of $\hat f_{S_0}$, i.e. $\hat f_{S_0}(\pi_{S_0})\neq 0$,
	\item $\pi^\infty$ has a nonzero $U_x$-fixed vector and
	\item $\pi_\infty$ is $\xi$-cohomological (Definition~\ref{d:cohomological}).
\end{itemize}

To be precise this is a multi-set with a density function $a_{\cF_x}(\pi)$ as in~\cite{RS87}, \cite[\S3.3]{Shi-Plan} and~\cite[\S9.2]{ST11cf}. In the present case we have for all $\pi \in \cF_x$,
\begin{equation}
  a_{\cF_x}(\pi):=
  m_{\text{disc}}(\pi)
  \hat f_{S_0}(\pi_{S_0})
  \dim ( (\pi^{S_0,\infty})^{U^{S_0}_{x}} )
\end{equation}
where $m_{\text{disc}}(\pi)$ is the multiplicity of $\pi$ in the discrete automorphic spectrum. The exact nature of the formula for $a_{\cF_x}(\pi)$ does not play an explicit role in what follows but is needed in \cite{ST11cf}, which we are going to cite from there. Note that it is non-negative and if $\hat f_{S_0}$ is a characteristic function, it takes integral values. The cardinality of a subset of the multi-set $\cF_x$ is considered in the obvious sense. For instance $|\cF_x|$ is defined to be the number of $\pi$'s in $\cF_x$ counted with multiplicities $a_{\cF_x}(\pi)$.

\begin{thm}\label{t:growth-coeff}
	Let $v\not \in S_0$ be a fixed place of ${F^+}$ such that either
	\benu
  \item	$U_{x,v}$ is maximal hyperspecial for all but finitely many $x$, or
  \item $\ord_v(\Fmn_x) \to \infty$ as $x\to \infty$,
	\eenu
	Then $|\{\pi\in \cF_x: [\Q(\pi_v):\Q]\le A\}|/|\cF_x|$ tends to 0 as $x\ra \infty$.
\end{thm}

\begin{rem}
We recall that the case $G=\GL(2)_\Q$ under assumption (i) is due to Serre~\cite{Serre:pl}, see Theorem~\ref{t:intro-Serre}.
\end{rem}

\begin{proof}
  We have seen in Corollary~\ref{c:pi_v-pure} that $\pi_v$ is tempered for all $\pi\in \cF_x$. We are in position to apply Corollary \ref{c:sparsity} which implies that the set $\cZ^{\mathrm{ur}}$ (resp. $\cZ$) of all $\pi_v\in G(F^+_v)^{\wedge,\mathrm{ur}}$ (resp. $\pi_v\in G(F^+_v)^{\wedge}$) for $\pi \in \cF_x$ with $[\Q(\pi_v):\Q]\le A$ is finite. For part (i), concerning $\cZ^{\mathrm{ur}}$, we could use alternatively the easier fact that there are only finitely many associated Weil numbers (Lemma \ref{l:finite-unr}). Since $\cZ^{\mathrm{ur}}$ and $\cZ$ are finite, they are certainly a $\pl_v$-regular relatively compact subset of $G(F^+_v)^{\wedge}$.

  We will follow the notation of \cite{ST11cf} and all measures are chosen as in that paper. We have
  \begin{equation}\label{pf:ineq-cFk}
	\begin{aligned}	
		|\{\pi\in \cF_x: [\Q(\pi_v):\Q]\le A\}|
		&= |\{\pi\in \cF_x: \pi_v \in \cZ^{\mathrm{ur}}\}| \\
		& = \frac{\tau'(G)\dim \xi}{\vol(U_x^v)} \hat \mu_{x}(\cZ^{\mathrm{ur}}),
	\end{aligned}
	\end{equation}
where $\tau'(G)$ is the volume of $G({F^+})\bs G(\A_{F^+})/A_{G,\infty}$, and $\hat \mu_{x}(\cZ^{\mathrm{ur}})$ is the automorphic counting measure for $\cF_x$. (See \S6.6 and (9.5) with $S_0=\{v\}$ and $S_1=\emptyset$ in \cite{ST11cf}.) The same as \eqref{pf:ineq-cFk} holds true for $x\gg1$ with $\cZ$ in place of $\cZ^{\mathrm{ur}}$. (We want $x\gg1$ so that every member of $\cZ$ has level at most $\fkn_x$ at $v$.) A key ingredient for both (i) and (ii) is the automorphic Plancherel equidistribution theorem~\cite[Cor~9.22]{ST11cf} (also see \cite[Thm~4.4]{Shi-Plan}), stating that
$ \lim_{x\to \infty} \hat \mu_{x}(\cZ^{\mathrm{ur}}) = \pl_v(\cZ^{\mathrm{ur}})$ and the same for $\cZ$ in place of $\cZ^{\mathrm{ur}}$.

(i) According to \cite[Cor~9.25]{ST11cf}, $\lim_{x\ra \infty}\frac{\tau'(G)\dim \xi }{\vol(U_x^v)|\cF_x|}=1$. Hence the limit in the theorem is nothing but $\pl_v(\cZ^{\mathrm{ur}})$, which is zero. (Note that $\cZ^{\mathrm{ur}}$ is a finite subset of the unramified unitary dual which is a torus of positive dimension and the restriction of the Plancherel measure is absolutely continuous with respect to the Lebesgue measure.)

(ii) By \cite[Prop 5.2]{Shi-Plan} and its extension to the setting of \cite{ST11cf} by the same argument, we have
\beq\label{pf:eq-cFk} \lim_{x\ra \infty}\frac{\tau'(G)\dim \xi }{\vol(U_x)|\cF_x|}=1.\eeq
(The corollary 9.25 of \cite{ST11cf} cannot be applied as it assumes that $\fkn_x$ is prime to $v$. Note that \eqref{pf:eq-cFk} is consistent with the formula in the proof of (i), in which case $\vol(U_{x,v})=1$.)
Therefore $$
 \lim_{x\ra \infty}\frac{\tau'(G)\dim \xi}{\vol(U_x^v)|\cF_x|} \hat \mu_{x}(\cZ)
  = \lim_{x\ra \infty} \vol(U_{x,v}) \hat \mu_{x}(\cZ)=0$$
since we have that $ \vol(U_{x,v}) \to 0$ from that $\ord_v(\Fmn_x) \to \infty$. (Note that $\hat \mu_{x}(\cZ)$ tends to $\pl_v(\cZ)$, which may not be zero due to discrete series in $\cZ$ but has bounded value.) The proof is concluded.
\end{proof}

\begin{rem}\label{r:unr}
	It would be interesting to improve on the condition that $U_{x,v}$ be maximal hyperspecial. This is a question of Serre~\cite[\S6.1]{Serre:pl} in the $\GL(2)$ case. The main obstruction is the presence of square-integrable representations $\pi_v\in \cZ$ with $\pi \in \cF_x$. The proof doesn't extend to these representations because $\pl(\pi_v)>0$.	
\end{rem}

  For convenience we introduce the multi-set
$$\cF_x^{\le A}:=\{\pi\in \cF_x:[\Q(\pi):\Q]\le A\}.$$
It is to be understood that if $\pi\in \cF_x^{\le A}$ then $\pi$ appears with the same multiplicity $a_{\cF_x}(\pi)$ in $\cF_x^{\le A}$. This way we make sense of $|\cF_x^{\le A}|$.

\begin{cor}\label{c:growth-coeff}
	Under the same assumptions, $|\cF_x^{\le A}|/|\cF_x|$ tends to 0 as $x\ra \infty$.
\end{cor}

\begin{proof}
  Obviously $[\Q(\pi_v):\Q]\le [\Q(\pi):\Q]$.
\end{proof}

\subsection{Quantitative estimates}
  One may wonder about the precise size of $\cF_x^{\le A}$ relative to that of $\cF_x$. For instance, the following generalizes another Serre's question for families of modular forms (Remarques 2 below Th\'eor\`eme 6 of \cite{Serre:pl}).

\begin{qu}\label{q:coeff1}
  Does there exist $\delta<1$ such that $|\cF_x^{\le A}|=O(|\cF_x|^{\delta})$?
\end{qu}

As a weaker variant (cf. Remark \ref{r:modest}), for a fixed finite extension $E$ of $\Q$ in $\C$ one may ask whether there exist $\delta<1$ such that $|\{\pi\in \cF_x:  \Q(\pi)\subset E \}|= O(|\cF_x|^{\delta})$.
We establish the following estimate towards a positive answer to Question~\ref{q:coeff1}. 	Define $S_{\unr}$ to be the set of finite places $v$ of ${F^+}$ such that $U_{x,v}$ is hyperspecial at $v$ for all large enough $k$. Let $R_{\unr}$ be the sum of the $F^+_v$-ranks of $G(F^+_v)$ for all $v \in S_{\unr}$ (it could be infinity).

\begin{thm}\label{p:level-prime}
Suppose that $S_{\unr}$ is not empty (but it could be an infinite set).  Then, as $x\to \infty$,
	\begin{equation}\label{eq:p:level-prime}
	|\cF_x^{\le A}| \ll_R  \frac{|\cF_x|}{(\log |\cF_x|)^{R}},
\end{equation}
for all $R\le R_{\unr}$.
\end{thm}

\begin{ex}
  In a typical example if $U_x$ is a principal congruence subgroup of prime level $\fkn_x$ then the set $S_{\unr}$ contains all finite places and $R_{\unr}$ is infinite: the statement holds for all $R>0$ which is an indication for an affirmative answer to Question~\ref{q:coeff1} in this case.	 If $R_{\unr}\ge 1$ is finite then it is best to choose $R=R_{\unr}$. Note that the possibility $R_{\unr}=0$ is excluded from the proposition because $S_{\unr}$  is not empty; this is the case discussed in Remark~\ref{r:unr}.
\end{ex}
	
\begin{proof}
	We fix a finite set of unramified places $S_1\subset S_{\unr}$ disjoint from $S_0$. Let $\cR$ be a rectangle in $G(F^+_{S_1})^{\wedge,\unr,\temp}$. Lemma~\ref{l:app-dual} yields the existence of $\phi_{S_1} \in \cH^{\unr}(G(F^+_{S_1}))^{\le c\kappa}$ which is such that $\hat \phi_\kappa$ approximates the characteristic function of $\cR$. (The definition of $\cH^{\unr}(G(F^+_{S_1}))^{\le c\kappa}$ is recalled in \S\ref{sub:approx} below. The constant $c>0$ depends on a choice fixed once and for all for $G$.)

Applying the automorphic Plancherel theorem with error bound~\cite[Thm~9.16]{ST11cf} to the family $\cF_x$, we deduce that for all integer $\kappa \ge 1$,
\begin{equation*}
  \hat \mu_{\cF_x,S_1}(\cR) = \pl_{S_1}(\cR)  + O( q_{S_1}^{A_{l}+B_{l}\kappa} |\cF_x|^{-C_{l}} ) + O(\kappa^{-R}).
\end{equation*}
where $A_{l},B_{l},C_{l}>0$ are absolute constants and $R\ge 0$ is the sum of the ranks of $G(F^+_v)$ for $v\in S_1$. Note that by choosing $S_1\subset S_{\unr}$ arbitrary large, the integer $R$ is arbitrary large subject to the condition that $R\le R_{\unr}$.

The optimal choice is $\kappa = O(\log |\cF_x|)$, which yields
\begin{equation}\label{small-balls}
  \hat \mu_{\cF_x,S_1}(\cR) = \pl_{S_1}(\cR) + O( (\log |\cF_x|)^{-R} ).
\end{equation}
Note that the constant in the remainder term does not depend on $\cR$. In particular $\cR$ can be chosen to be a single element in which case $\pl_{S_1}(\cR)=0$ since the Plancherel measure is atomless. We deduce that the following estimate holds for any finite set $\cZ$ in $G(F^+_{S_1})^{\wedge,\unr,\temp}$,
\[
\hat \mu_{\cF_x,S_1}(\cZ) \ll \frac{|\cZ|}{(\log |\cF_x|)^R}.
\]

We apply this to the set
\[
\cZ:=\{\pi_{S_1}:\ \pi\in \cF_x, \ [\Q(\pi):\Q]\le A\},
\]
since it follows as before from Corollary~\ref{c:sparsity} (or alternatively from Lemma \ref{l:finite-unr} and the first assertion of Corollary \ref{c:pi_v-pure}) that $\cZ$ is a finite set. Thus we can conclude the proof of the proposition since
\[
|\{\pi\in \cF_x: [\Q(\pi):\Q]\le A\}| \ll |\cF_x| \hat \mu_{\cF_x,S_1}(\cZ).
\]
\end{proof}

We now consider the case where the automorphic family admits ramification at only finitely many fixed places $S$. Theorem~\ref{p:level-prime} applies for any $R>0$ since $R_{\unr}$ is infinite, but we can prove a stronger bound. Indeed Theorem \ref{t:fin-GL_n} may be rephrased as a strong answer to Question \ref{q:coeff1}. For this it is \emph{unnecessary} to assume that the highest weight of $\xi$ is regular (thus $\pi_\infty$ may not be a discrete series). In fact we will prescribe a condition at infinity which is weaker than the $\xi$-cohomological condition. For a $\C$-algebra morphism $Z(\fkg)\ra \C$ (cf. \S\ref{sub:fin-results}) and an open compact subgroup $U_x\subset G(\A_{F^+}^\infty)$, define $\cF(U_x,\chi_\infty)$ to be the set of discrete automorphic representations $\pi$ of $G(\A_{F^+})$ such that (for the corollary it is unimportant to think of $\cF(U_x,\chi_\infty)$ as a multi-set, i.e. the multiplicity of each member may be taken to be one)
\bit
\item $(\pi^\infty)^{U_x}\neq 0$,
\item $\pi_\infty$ has infinitesimal character $\chi_\infty$.
\eit

\begin{cor}\label{c:growth-finite} Fix $A\in \Z_{\ge1}$. Let $G$ be either
\bit
\item $G=GL_n$ over an arbitrary number field $F$ or
\item $G$ is a quasi-split classical group of \eqref{sub:transfer-app} over a totally real field $F$.\footnote{For uniformity of notation we write $F$ rather than $F^+$ here.}
\eit
Suppose that there exists a finite set $S$ such that for every $x$, the level subgroup $U_x$ has the form $U_x=U_{S,x}U^{S,\infty}$, where $U^{S,\infty}$ is a product of hyperspecial subgroups of $G(F_v)$ for all finite $v\notin S$. Then there is a constant $C=C(A,G,\chi_\infty,S)$ such that for all $x$
$$|\{\pi\in \cF(U_x,\chi_\infty): [\Q(\pi):\Q]\le A\}| \le C.$$
\end{cor}

\begin{proof}
  Immediate from Theorems \ref{t:fin-GL_n} and \ref{t:fin-classical}.
\end{proof}

For instance when $G=GL_2$ over $\Q$, the theorem applies to C-algebraic automorphic representations arising from Maass forms, namely those with Laplace eigenvalue $1/4$. It is worth comparing our results in this subsection with previous work in the case of elliptic curves:

\begin{rem}\label{e:elliptic-curves}
  We briefly discuss the most basic case of $G=\GL(2)_{\Q}$, weight $2$ and $\Q(\pi)=\Q$ (that is $A=1$). See also the remarks following~\cite[Th\'eor\`eme~7]{Serre:pl}. Modular forms of weight two with integer coefficients are attached to elliptic curves and thus more precise results than~\eqref{eq:p:level-prime} are available.

For an integer $N\in \Z_{\ge 1}$, let $Ell(N)$ be the number of isogeny classes of elliptic curves over $\Q$ of conductor $N$. The following is currently known~\cite[\S3.1]{DK00}:
\begin{equation*}
  X^{\frac56} \ll \sum_{1\le N\le X} Ell(N) \ll_\epsilon X^{1+\epsilon}.
\end{equation*}
On the other hand, by counting $S$-integral points on curves of given genus, it is shown by Helfgott--Venkatesh~\cite[\S4.2]{HV06} that $Ell(N)=O(N^\delta)$ for some $\delta<\frac12$, improving earlier bounds by Evertse, Silverman and Brumer. The numerical value is improved further in~\cite{EV07} into $\delta=0.169\ldots$.
\end{rem}

\subsection{Order of growth}\label{sub:growth-gl2}
It follows from Theorem~\ref{t:growth-coeff} that there are automorphic representations $\pi_x\in \cF_x$ such that $[\Q(\pi_x):\Q]\ra\infty$ as $x\ra \infty$. It is interesting to study the order of growth of $[\Q(\pi_x):\Q]$ as $x\to \infty$. We establish the following which generalizes a result of Royer~\cite[Thm~1.1]{Royer:dimension-rang} in the case of $G=\GL(2)_\Q$. By the degree of a Weil number $\alpha$ (or any algebraic number) we will mean $[\Q(\alpha):\Q]$.

\begin{prop}\label{p:order-growth}
  Let assumptions be as in (i) in Theorem~\ref{t:growth-coeff}. Then as $x\to \infty$ there exists an automorphic representation $\pi_x\in \cF_x$ such that
  \begin{equation}\label{e:p:og}
  [\Q(\pi_x):\Q] \gg (\log \log \N(\Fmn_x) )^{\frac12}.
\end{equation}
\end{prop}

\begin{proof} Consider the set of local representations $\pi_v\in G(F^+_v)^{\wedge,unr}$ as $\pi_v$ ranges over $\cF_x$. We see from~\eqref{small-balls} that there are $\gg \log \N(\fkn_x)$ distinct such representations $\pi_v$. On the other hand the number of $q_v$-Weil integers of weight 1 and degree $\le d$ is at most $q_v^{O(d^2)}$. (The $O(d^2)$-bound is easily seen from the argument of the last paragraph in the proof of Lemma \ref{l:finite-unr}.)
\end{proof}

In the depth aspect, that is under condition (ii) in Theorem~\ref{t:growth-coeff}, we can also give a lower bound for the order of growth. Suppose that $\Fmn_x$ is supported on a fixed finite set of primes. Then using the estimate in~\eqref{explicit-d-bound} we can deduce that there exists $\pi_x\in \cF_x$ such that
 \[
 [\Q(\pi_x):\Q] \gg (\log \N(\Fmn_x))^{\frac1n}.
  \]
 We have removed a logarithm compared with the order of growth~\eqref{e:p:og} obtained in the level aspect.

The remainder of this subsection is devoted to discuss the case of $G=\GL(2)_\Q$ and weight $2$ forms, where interestingly there is another method to establish the bound~\eqref{e:p:og}. This is based on the following result about curves over finite fields which is of independent interest.
\begin{prop}[Serre {\cite[\S7]{Serre:pl}}]\label{t:Serre-finite}
  There are only finitely many curves over $\F_q$ whose Jacobian is isogenous to a product of abelian varieties of dimension at most $d$.
\end{prop}

The method of Serre is effective, see~\cite[p.\,93]{Serre:pl} for the example of $q=2$ and $d=1$. It doesn't produce immediately an explicit upper-bound in general but there have been several works in this direction, in particular we quote the following.
\begin{prop}[Elkies--Howe--Ritzenthaler~\cite{EHR:genus-bounds}]\label{p:elr}
  Let $S\subset [0,\pi]$ be a finite set. If $C/\F_q$ is a curve of genus $g$ with Frobenius angles in $S$, then
 \begin{equation*}
   g \le 23 |S|^2 q^{2|S|} \log q.
\end{equation*}
\end{prop}

The proof of Proposition~\ref{t:Serre-finite} and of the effective bounds such as in Proposition~\ref{p:elr} is based on trigonometric inequalities. Precisely one uses the fact that there are $\theta_j \in [0,\pi]$, $1\le j \le g$, such that $q^{\frac12}e^{i\theta_j}$ (and also $q^{\frac12}e^{-i\theta_j}$) are $q$-Weil integers of weight 1 and
\[
2 q^{\frac{n}{2}} \sum_{j=1}^{g} \cos(n \theta_j) \le q^n+1, \quad \text{for any integer $n\ge 1$.}
\]
(The $\theta_j$ are the Frobenius angles and this holds because the right-hand side of the inequality minus the left-hand side is equal to $\#C(\F_{q^n})$, the number of points of $C$ over $\F_{q^n}$.)

The Proposition~\ref{p:elr} implies the following effective estimate in the case of simple isogeny factors of dimension at most $d$.
\begin{cor}\label{prop:DJ}
If the Jacobian of a curve of genus $g$ over $\F_q$ is isogenous to a product of abelian varieties of dimension at most $d$, then
 \begin{equation*}
   g \le  q^{q^{O(d^2)}}.
 \end{equation*}
The underlying constant in $O(d^2)$ is absolute (independent of $q$ and $d^2$).
\end{cor}
\begin{ex}
  Let $q=p$ be a prime number and $r\in \Z_{\ge1}$. The Fermat curve
 \[
 C_r: X^{p^r+1} + Y^{p^r+1} + Z^{p^r+1} =0
 \]
 is such that all eigenvalues of Frobenius are $2r$-th roots of $-p^r$ (see~\cite{GR:Fermat}). Thus $\Jac(C_r)$ is isogenous to a product of abelian varieties over $\F_p$ of dimension at most $4r$. On the other hand $C_r$ has genus $\frac{p^r(p^r-1)}{2}$.
 Also it may be verified that the exponent of the class group $\Jac(C)(\F_p)$ is $\le p^r+1$, which is asymptotically the square-root of the genus and may be compared with~\eqref{MM} below. Note that $C_r$ is a \emph{hermitian curve} over $\F_{p^{2r}}$ and it is a maximal curve in the sense that $C_r(\F_{p^{2r}})$ is of cardinality $1+p^{2r}+2gp^r$ which achieves equality in the Weil bound.
\end{ex}

In fact the same result as in Corollary~\ref{prop:DJ} was established around 2000 by A.~de Jong using a different method. We would like to thank de Jong for explaining his result to us which had remained unpublished.
\begin{proof}[Alternative proof of the Corollary. (de Jong)]
  A theorem of Madan-Madden~\cite{MM:exp} states that the exponent $E$ of the class group of a curve $C$ of genus $g$ over $\F_q$ satisfies
\begin{equation}\label{MM}
  E \gg \left(\frac{g}{\log^3 g}\right)^{\frac14}.
\end{equation}
(Note that their arguments do apply uniformly in $q$ and thus the above multiplicative constant is absolute, though this is not explicitly stated in their paper. Precisely it can be verified that each estimate in their proof improves when $q$ gets large)

On the other hand let $\Fr_q$ be the $q$-th power Frobenius endomorphism of $\Jac(C)$ and let $P\in \Z[X]$ be its minimal monic polynomial. Note that $P$ has integral coefficients because $\Fr_q$ is an element of the endomorphism ring of $\Jac(C)$ which is an order in a semisimple algebra. Since $\Fr_q$ is a semisimple endomorphism by Tate's theorem, $P(\Fr_q)$ acts as 0 on $\Jac(C)$. Since $\Fr_q$ acts as the identity on $\Jac(C)(\F_q)$, we deduce that $P(1)\in \Z$ acts as $0$ on $\Jac(C)(\F_q)$. Therefore
\begin{equation*}
  E \mid P(1).
\end{equation*}
The polynomial $P$ divides the product of the characteristic polynomials of Frobenius on the abelian varieties which are the simple isogeny factors of $\Jac(C)$. By assumption these abelian varieties have dimension $\le d$ and there are $q^{O(d^2)}$ isogeny classes of them by counting the Weil $q$-integers of weight 1 given via Honda-Tate theory, cf. the proof of Proposition \ref{p:order-growth}. Thus
\[
P(1) \le q^{q^{O(d^2)}}.
\]
Note that $P(1)\neq 0$ because $\Fr_q$ is always a non-trivial endomorphism.
Combining the three estimates we conclude the proof of the proposition.
\end{proof}

\begin{proof}[Alternative proof of Proposition~\ref{p:order-growth} for $G=\GL(2)_{\Q}$ in weight $2$]
  Consider the modular curve $X_0(N)$ which is a smooth algebraic curve over $\Q$ of genus $g_0(N)$.
Let $(A_i)_{i\in I}$ be the simple isogeny factors of its Jacobian $J_0(N)$, counted with multiplicity, so that there exists an injective isogeny $\prod_{i\in I}A_i\hra J_0(N)$ over $\Q$ (\cite[Prop 10.1]{Mil86}). By the theorem of Eichler-Shimura we are reduced to finding a lower bound for the maximal dimension
\[
d := \max_{i\in I} \dim A_i.
\]

Suppose that the fixed prime $p$ does not divide $N$. From now on we work over $\Q_p$ and with a small abuse of notation we still write $X_0(N)$, $J_0(N)$ and $A_i$ for their base change $X_0(N)\otimes_\Q \Q_p$, $J_0(N)\otimes_\Q \Q_p$ and $A_i\otimes_\Q \Q_p$ respectively.

There exists an integral model $\XN$ over $\Z_p$ and its reduction modulo $p$ is smooth irreducible over $\F_p$. Also there exists a relative Picard scheme $\JN$ which is a smooth abelian group scheme over $\Z_p$. The generic fiber $\JN\otimes_{\Z_p} \Q_p$ can be identified with $J_0(N)$. Since $J_0(N)$ has good reduction at $p$, the N\'eron--Ogg--Shafarevich criterion tells us that $A_i$ has good reduction at $p$ for each $i\in I$. Let $\mathcal{A}_i$ denote the integral model of $A_i$ over $\Z_p$ which is an abelian scheme. By the property of a N\'{e}ron model, the injection $\prod_{i\in I}A_i\hra J_0(N)$ extends to an injection $\prod_{i\in I}\mathcal{A}_i\hra \JN$. (The latter is an injection because the kernel is flat over $\Z_p$ with trivial group scheme as the generic fiber.) As an injection between abelian schemes of the same dimension, it is also an isogeny.

%
 Reducing modulo $p$ we find that $\JN \otimes_{\Z_p} \F_p$ is isogenous to the product $\prod\limits_{i\in I} \mathcal{A}_i\otimes_{\Z_p} \F_p$. Each simple isogeny factor of $\JN \otimes_{\Z_p} \F_p$ is a factor of $\mathcal{A}_i\otimes_{\Z_p} \F_p$ for some $i\in I$. In particular all isogeny factors of $\JN \otimes_{\Z_p} \F_p$ have dimension $\le d$.

Since $\XN \otimes_{\Z_p} \F_p$ is an irreducible smooth curve of genus $g_0(N)$ whose Jacobian can be identified with $\JN \otimes_{\Z_p} \F_p$ we are in position to apply Proposition~\ref{prop:DJ} which yields
\[
g_0(N) \le p^{p^{O(d^2)}}.
\]
Since $g_0(N)\asymp N$ as $N\to \infty$, this concludes the proof of Proposition~\ref{p:order-growth} for $G=\GL(2)_\Q$.
\end{proof}


\subsection{Uniform approximation in the unitary dual}\label{sub:approx}
In this subsection we record some lemmas on approximation by functions in the local Hecke algebra of bounded degree. Only in this subsection let $G$ be a connected reductive group over a $p$-adic field $K$. Write $U^{\hs}$ for a fixed hyperspecial subgroup of $G(K)$ and $\Omega_K$ for the Weyl group for $G$ relative to $K$.

We begin with the classical problem of approximating periodic functions by trigonometric polynomials. The following result is a version with sharp constants that comes from the work of Beurling in the 1930's and rediscovered by Selberg in the context of the large sieve inequality.
We identify $\T=\R/\Z$ with the unit circle $S^1$ inside $\C$. Thus a trigonometric polynomial is viewed as an element of $\C[z,z^{-1}$.
\begin{lem}[Vaaler~\cite{Vaaler85}]\label{lem:vaaler}
  Let $f$ be a function on $\T$ of bounded variation $V(f)\in \R_{\ge 0}$. For every integer $\kappa\in \N$ there are trigonometric polynomials $P_\kappa^{\pm}$ of degree $\kappa$ such that $P_\kappa^- \le f \le P_\kappa^+$ and
\begin{equation}
  \int_{\T} P_\kappa^+ - P_\kappa^- = \frac{V(f)}{\kappa+1}.
\end{equation}
In particular $||P_\kappa^{\pm}||_1 \le ||f||_1+ \frac{V(f)}{\kappa+1}$ by the triangle inequality.
Also the $n$-th coefficients of $P^{\pm}_\kappa$ are uniformly bounded by $\ll \frac{V(f)}{|n|}$ for all $n\neq 0$.
\end{lem}
\begin{proof} This is \cite[Thm~19]{Vaaler85} where it is also shown that the constants are sharp if $f$ is a sign function. We briefly recall the the construction of the polynomials:
  \begin{equation*}
	P_{\kappa}^{\pm}(z)=\sum_{|n|\le \kappa}
	\left[\hat J\left( \frac{n}{\kappa+1} \right) \hat f(n) \pm \frac{n}{(2\kappa+2)} \hat K\left( \frac{n}{\kappa+1} \right) \hat g(n) \right] z^n,
  \end{equation*}
 for all $z\in \T$. Here $\hat f(n)$ (resp. $\hat g(n)$) are the Fourier coefficients of $f$ (resp. the variation function of $f$). The Beurling functions $J$ and $K$ are entire of exponential type $2\pi$ with Fourier transform:
\begin{equation*}
  \hat J(t) := \pi t (1-|t|) \cot(\pi t) + |t|, \quad \hat K(t):=1-|t|, \quad |t|<1.
\end{equation*}
The properties of $J$ and $K$ and some arguments in Fourier analysis imply the first two assertions of the lemma. Since $\hat f(n), \hat g(n)\ll \frac{V(f)}{|n|}$ for all $n\neq 0$, we deduce the third and final assertion on the decay of coefficients.
\end{proof}

The Satake isomorphism induces a topological isomorphism $G(K)^{\wedge,\unr,\temp}\simeq \hat A_c/\Omega_K$ where $\hat A_c\simeq \T^r$ is a complex torus with $r$ the $K$-rank of $G$. For $\phi\in\cH^{\unr}(G(K))$ we write $\hat{\phi}$ for the corresponding function on the real torus $\hat{A}_c$ or its quotient $\hat{A}_c/\Omega_K$. The truncated Hecke algebra $\cH^{\unr}(G(K))^{\le \kappa}$ is defined in \cite[\S2]{ST11cf} so that the following holds (which is all we need to know here): there exists a constant $c>0$ (depending on a fixed choice of basis in the character group of a maximal torus in $G$ over $\ol{K}$) such that for every $\kappa\in \Z_{>0}$, the set of $\phi\in\cH^{\unr}(G(K))$ such that $\hat{\phi}$ is a ($\Omega_K$-invariant) polynomial of degree $\le \kappa$ on $\hat A_c$ contains $\cH^{\unr}(G(K))^{\le \kappa/c}$ and is contained in $\cH^{\unr}(G(K))^{\le c\kappa}$. (Use \cite[\S2.4]{ST11cf} to see this.)

\begin{lem}\label{l:app-dual} Let $c>0$ be as above. For every integer $\kappa \ge 1$, and every rectangle $\cR \subset G(K)^{\wedge,\unr,\temp}$, there is a Hecke function $\phi_\kappa \in \cH^{\unr}(G(K))^{\le c \kappa}$ such that $\hat \phi_\kappa \ge 0$ on $G(K)^{\wedge,\unr,\temp}$, $\hat \phi_\kappa \ge 1$ on $\cR$ while $\pl(\hat \phi_\kappa) \ll \pl(\cR)+\kappa^{-r}$ and $|\phi_\kappa|\ll 1$. Here $r$ is the rank of $G(K)$.
\end{lem}

\begin{proof}
  We can apply Lemma~\ref{lem:vaaler} to the characteristic function $\mathds{1}_I$ of any interval $I$ of $\T$, in which case the total variation is $V(\mathds{1}_I)=2$. Then it is not difficult to deduce the following statement in higher dimension. Let $\cR=I_1\times \ldots \times I_r$ be a rectangle in $\T^r$. There are trigonometric polynomials $P_{\kappa}^{\pm}$ of degree $\le \kappa$ in $r$ variables such that $P_\kappa^- \le \mathds{1}_{\cR} \le P_\kappa^+$ and
  \begin{equation}\label{pf:Pkappa}
    \int_{\T^d} P_\kappa^+ - P_\kappa^- \ll \kappa^{-r}.
   \end{equation}

We choose $\hat \phi_\kappa$ to be the $\Omega_K$-average of $P_\kappa^+$. Then $\phi_\kappa\in \cH^{\unr}(G(K))$ giving rise to $\hat\phi_\kappa$ via the Satake isomorphism belongs to $ \cH^{\unr}(G(K))^{\le c\kappa}$. Note that the first two assertions follow from the inequality $\mathds{1}_{\cR} \le P_\kappa^+$.

The estimate of $\pl(\hat \phi_\kappa)$ follows from~\eqref{pf:Pkappa} and the fact that the Plancherel density on $G(K)^{\wedge,\unr,\temp}$ given by Macdonald formula is uniformly bounded below (see~\cite[Prop.~3.3]{ST11cf}). In other words we used that the Lebesgue measure on $\hat A_c/\Omega_K$ is absolutely continuous with respect to the Plancherel measure.

The Harish-Chandra Plancherel formula $\phi(1)=\pl(\hat \phi)$ holds for all smooth functions $\phi$ thus in particular for all $\phi\in \cH^{\unr}(G(K))$. In the unramified case (see \cite[Thm VIII.1.1]{Wal03} for the general case) we have more generally the relation
\begin{equation}\label{pf:pl-inv}
  \phi(g) = \int_{G(K)^{\wedge,\unr,\temp}} \hat \phi(\pi) M_\pi(g) d\hat{\mu}^{\mathrm{pl}}(\pi),\quad g\in G(K),
\end{equation}
where $M_\pi(g) = (v_\circ,g v_\circ)$ is a spherical matrix coefficient of $\pi$, that is $v_\circ$ is a unit $U^{\hs}$-fixed vector in the representation space $V_\pi$.
Let us justify formula~\eqref{pf:pl-inv} by computing the trace of $\pi(\phi) \circ \pi(g)^{-1}$ on $V_\pi$ and the Plancherel formula for $\phi(1)$. Note that $\pi(\phi)$ has image in $\C v_\circ$ because $\phi$ is left $U^{\hs}$-invariant. Using also the right $U^{\hs}$-invariance, we infer that
\[
\pi(\phi) w = \hat \phi(\pi) (w,v_\circ) v_\circ
\]
for all vector $w\in V_\pi$. Thus $\pi(\phi) g^{-1}v_\circ = \hat \phi(\pi) M_\pi(g) v_\circ$.
Since $\pi(\phi) \circ \pi(g)^{-1}$ maps $V_\pi$ into $\C v_\circ$, this implies that its trace is $\hat \phi(\pi) M_\pi(g)$.

From~\eqref{pf:pl-inv} we deduce that $|\phi_\kappa(g)|\le \phi_\kappa(1)$ for all $g\in G(K)$. Thus we deduce from the estimate for $\pl(\hat \phi_\kappa)$ that $|\phi_\kappa|\ll 1$.
\end{proof}

\subsection{The case of unitary groups}\label{sub:unitary-groups}

In this subsection let $G$ be a unitary group as in \S\ref{sub:transfer-app} or its inner form and assume that $[F^+:\Q]\ge 2$.
We would like to explain unconditional results on the growth of field of rationality which are already available from our current knowledge. Let us be brief; 
eventually complete unconditional results for non quasi-split unitary (resp. symplectic/orthogonal) groups will follow from our earlier arguments once the unitary group analogue of \cite{Mok} (resp. \cite{Arthur}) is extended to inner forms and Hypothesis \ref{hypo:TwTF} is verified.

We assert below that Theorem \ref{t:growth-coeff}.(i) and Theorem \ref{p:level-prime} hold true for unitary groups without any hypothesis. Let $\cF_x=\cF(U_x)$ be a level aspect family constructed for $G$, now a unitary group, as in \S\ref{sub:growth-level}. Let us define $S_{\unr}$ and $R_{\unr}$ for $G$ and $\cF_x$ as in Theorem \ref{p:level-prime}.

\begin{thm}\label{t:unitary-groups} Suppose that the highest weight of $\xi$ is regular, that $[F^+:\Q]\ge 2$, and
  that $S_{\unr}\neq \emptyset$ so that $R_{\unr}$ is defined. Then for all $R\le R_{\unr}$,
  $$|\cF_x^{\le A}|/|\cF_x|\ll_R |\cF_x|/(\log |\cF_x|)^R,\quad \mbox{as}~x\ra\infty.$$
\end{thm}

  The argument is the same as in Theorem \ref{p:level-prime} (also see Theorem \ref{t:growth-coeff}). The theorem relies on some of the earlier results, which we need to justify for unitary groups, but this is not so complicated as we are concerned only with the unramified local components here. The necessary results are provided by \cite[Cor 5.3]{Lab}, especially the weaker analogue of Proposition \ref{p:twisted-endoscopy} (here ``weaker'' means that no information is available at finitely many $v$ where $\pi$, $\eta$ or the extension $F/F^+$ is ramified at $v$). In Corollary \ref{c:pi_v-pure}, only the first assertion is needed and derived from the latter substitute. Then the methods of proof for Theorems \ref{t:growth-coeff} and \ref{p:level-prime} justify Theorem \ref{t:unitary-groups} once it is noted that the final main ingredients, namely Lemma \ref{l:finite-unr} and the level-aspect Plancherel equidistribution theorem with error terms (\cite{ST11cf}), are still valid for unitary groups.


\subsection{Concluding remarks}

  As we have noted earlier, the arguments and main results of this paper should apply to non-quasi-split classical groups as soon as the work \cite{Arthur} and \cite{Mok} are extended to those groups.
  There are several directions in which our work may be generalized. An obvious problem is to deal with other reductive groups. As for the growth of field of rationality, we raised the question of removing the hypotheses from Theorem \ref{t:intro-growth} and power saving in Question \ref{q:coeff1}. Any quantitative refinement such as power saving would be of arithmetic significance, already in the case of weight 2 modular forms and field of rationality $\Q$, cf. Remark~\ref{e:elliptic-curves}. Another widely open question is how much of \S\ref{s:growth-coeff} remains valid for families in the weight aspect (for instance as defined in \cite{ST11cf}). In this respect even the case of modular forms is still unsolved (Maeda's conjecture). Note that the finiteness of Weil numbers in the argument for Theorem \ref{t:intro-Serre} fails if weight grows to infinity. Finally we would like to mention Hida's recent study of field of rationality (``Hecke field'' in his terminology) for $p$-adic families of modular forms and arithmetic applications (\cite{Hid11}, \cite{Hid12}), providing a perspective different from ours.

\bibliographystyle{abbrv}
\bibliography{bib,bib2,all}

\end{document}